\documentclass[11pt]{article}
\usepackage[utf8]{inputenc}
\usepackage{amsmath}
\usepackage{amssymb}
\usepackage{amsfonts}
\usepackage{amsthm}
\usepackage{mathabx} 
\usepackage{bbold}
\usepackage{comment}
\usepackage[top=2cm, bottom=2cm, left=2.8cm, right=2.8cm]{geometry}
\usepackage[linktocpage]{hyperref}
\usepackage[british]{babel}
\usepackage[shortlabels]{enumitem}
\usepackage{stmaryrd}

\usepackage{sansmathfonts}
\usepackage[sfdefault]{atkinson}
\usepackage[T1]{fontenc}

\usepackage{graphicx}
\graphicspath{ {./images/} }
\usepackage{subfig}
\usepackage[export]{adjustbox}

\usepackage{svg}

\newtheoremstyle{indented}{3pt}{3pt}{\addtolength{\leftskip}{3em}}{}{\bfseries}{.}{.5em}{}

\theoremstyle{indented}

\theoremstyle{plain}
\newtheorem{Theorem}{Theorem}
\newtheorem{Corollary}[Theorem]{Corollary}
\newtheorem{Th}{Theorem}[section]
\newtheorem{Prop}[Th]{Proposition}

\newtheorem{Cr}[Th]{Corollary}
\newtheorem{Lemma}[Th]{Lemma}

\newtheorem*{Th*}{Theorem}
\newtheorem*{Lemma*}{Lemma}
\newtheorem{Claim}[Th]{Claim}

\theoremstyle{definition}
\newtheorem{Def}[Th]{Definition}
\newtheorem{Setting}[Th]{Setting}
\newtheorem{Ex}[Th]{Example}
\newtheorem{Rem}[Th]{Remark}

\newtheorem*{Def*}{Definition}

\theoremstyle{plain}
\newtheorem{manualtheoreminner}{Theorem}
\newenvironment{manualtheorem}[1]{%
  \renewcommand\themanualtheoreminner{#1}%
  \manualtheoreminner
}{\endmanualtheoreminner}

\newtheorem{manualcorinner}{Corollary}
\newenvironment{manualcor}[1]{%
  \renewcommand\themanualcorinner{#1}%
  \manualcorinner
}{\endmanualcorinner}

\usepackage{xcolor}
\usepackage{color}

\definecolor{darkpink}{HTML}{D01C8B}
\definecolor{lightpink}{HTML}{F1B6DA}
\definecolor{lightgreen}{HTML}{B8E186}
\definecolor{darkgreen}{HTML}{4DAC26}
\definecolor{darkred}{rgb}{0.8, 0.0, 0.0}
\definecolor{darkpurple}{HTML}{7B3294}

\hypersetup{
colorlinks,
linkcolor={darkpurple},
citecolor={darkpurple},
urlcolor={darkpurple}
}

\newcommand{\mrm}{\mathrm}

\newcommand{\ds}{\displaystyle}

\newcommand{\mcl}{\mathcal}
\newcommand{\mfk}{\mathfrak}
\newcommand{\gen}[1]{\langle {#1} \rangle}
\newcommand{\partent}[1]{\left\lfloor {#1} \right\rfloor}

\newcommand{\eqdef}{=\mathrel{\mathop:}}
\newcommand{\defeq}{\mathrel{\mathop:}=}

\newcommand{\N}{\mathbb{N}}
\newcommand{\Z}{\mathbb{Z}}

\renewcommand{\epsilon}{\varepsilon}

\newcommand{\g}{\gamma}
\newcommand{\G}{\Gamma}

\newcommand{\w}{\omega}
\newcommand{\s}{\sigma}
\renewcommand{\S}{\Sigma}
\renewcommand{\a}{\alpha}
\renewcommand{\b}{\beta}

\newcommand{\mclH}{\mathcal{H}}
\newcommand{\mclG}{\mcl{G}}
\newcommand{\mclA}{\mcl{A}}
\newcommand{\mclB}{\mcl{B}}
\newcommand{\mclS}{\mcl{S}}
\newcommand{\mclR}{\mcl{R}}
\newcommand{\mclK}{\mcl{K}}
\newcommand{\mclC}{\mcl{C}}

\newcommand{\ModS}{\mathrm{Mod}(\Sigma)}
\newcommand{\ModSprime}{\mathrm{Mod}(\Sigma')}
\newcommand{\Mod}[1]{\mathrm{Mod}(#1)}
\newcommand{\Modext}{\mathrm{Mod}^{\pm}(\Sigma)}
\newcommand{\Modz}{\mathrm{Mod}^0(\Sigma)}
\newcommand{\Modzprime}{\mathrm{Mod}^0(\Sigma')}
\newcommand{\Modact}[1]{\mathrm{Mod}_{ #1 }(\Sigma)}
\newcommand{\Modactprime}[1]{\mathrm{Mod}_{ #1 }(\Sigma')}
\newcommand{\Modfix}[1]{\mathrm{Mod}_{ {#1 }^*}(\Sigma)}

\newcommand{\cxcourbes}[1]{\mcl{C}(#1)}

\newcommand{\Stab}[1]{\mrm{Stab}(#1)}
\newcommand{\Prob}{\mrm{Prob}}
\newcommand{\Supp}{\mrm{Supp}~}
\newcommand{\varphimax}{\varphi_{\text{max}}}

\newcommand{\dsig}{\delta_{\Sigma}}
\newcommand{\nsig}{{n_{\Sigma}}}

\title{Self measurable embeddings of mapping class groups}
\author{Alexandra Gurieva}
\date{\today}

\begin{document}
\maketitle

\begin{abstract}
Given a finite-type orientable surface without boundary, outside of a few low-complexity surfaces, we show that any measurable embedding of the associated mapping class group into itself is a measure equivalence. This is a measure-theoretic analog to the result of Ivanov and McCarthy that mapping class groups are co-Hopfian.
\end{abstract}

\tableofcontents

\vfill

\noindent {\href{https://alexandragurieva.pages.math.cnrs.fr/}{Alexandra Gurieva}}\\
		\textcolor{black!70}{ENS de Lyon,  Unité de Mathématiques Pures et Appliquées, 69007, Lyon, France \\
		\textit{e-mail:~}\texttt{firstname.lastname@ens-lyon.fr}}

\newpage
 \section*{Introduction}

 Rigidity of finitely-generated groups comes in a variety of flavours. Algebraic, geometric and measure-theoretic rigidities are well-studied fundamental examples. A group is algebraically rigid when is has a small automorphism group. Geometric rigidity of a group is quantified by the small amount of quasi-isometries it admits to various groups. In measured group theory, a group is rigid when its measure equivalence class is small. For instance, mapping class groups are rigid in these these three ways: \cite{ivanov_automorphism_1997} for automorphisms, \cite{behrstock_geometry_2012} and \cite{hamenstaedt_geometry_2007} for quasi-isometries, \cite{kida_measure_2010} for measure equivalence. There also exist such notions of rigidity for embeddings.
Algebraically, one can study subgroup embeddings of a group into itself or to other groups. One strong rigidity phenomenon for self-embeddings is the \textit{co-Hopfian} property: any injective homomorphism of the group into itself is necessarily an automorphism. Mapping class groups not only have this co-Hopfian property \cite{ivanov_injective_1999}, but they also verify a geometric analog of it: any self quasi-isometric embedding is a quasi-isometry \cite{bowditch_large-scale_2018}. Our aim is to show a measure-theoretic version of the co-Hopfian property for mapping class groups: any self measurable embedding is a measure equivalence.

A countable group $H$ \textit{measurably embeds} into a countable group $G$ when there exists a measure-preserving action of $G\times H$ on a standard measured space $\Omega$ such that 
\begin{itemize}[noitemsep, nolistsep]
    \item the action $G\curvearrowright \Omega$ is free and admits a measurable fundamental domain of finite positive measure;
    \item the action $H\curvearrowright \Omega$ is free and admits a measurable fundamental domain.
\end{itemize}
The space $\Omega$ is called a \textit{measurable embedding coupling}. A fundamental domain for a group action is a subset of the space that intersects each orbit exactly once. The groups $G$ and $H$ are \textit{measure equivalent} exactly when the fundamental domain of the $H$-action is also required to have finite measure. In this case, $\Omega$ is a \textit{measure equivalence coupling}. Note that if $H$ is a subgroup of $G$, then it measurably embeds into $G$ by taking $\Omega=G$ with $H$ acting by translation on the right and $G$ on the left. Note also that measurable embeddings are transitive (see Proposition~\ref{transitivité du plongement ME}).

The notion of measurable embedding has been defined for instance in \cite{sako_class_2009,demir_measurable_2022} and under the name \textit{measure subgroup} in \cite{delabie_quantitative_2022}. To our knowledge, the first result about measurable embeddings is the Gaboriau--Lyons theorem \cite{gaboriau_measurable-group-theoretic_2009}, stating that any countable nonamenable group measurably contains a nonabelian free group $F_2$. This exhibits flexibility for measurable embeddings compared to the algebraic setting. Indeed, there exist nonamenable groups that do not contain $F_2$ as a subgroup \cite{olchanskiy_question_1980}. Another important flexibility phenomenon, coming from measure equivalence flexibility \cite{ornstein_ergodic_1980}, is that all countable infinite amenable groups measurably embed into each other. Amenability is also preserved by measurable embeddings: when a group $H$ measurably embeds into an amenable group $G$, then $H$ is amenable as well (see Remark~\ref{rem amen is preserved by meas emb}). There have also been other results on monotone invariance for measurable embeddings, namely Ozawa's Class $\mcl S$ in \cite{sako_class_2009} and the Haagerup property in \cite{berendschot_measure_2025}. More recently, \cite{abgrall_algebraic_2026} gave a characterisation of measurable embeddings between square-free right-angled Artin groups.

In this paper, we exhibit a case of rigidity for measurable embeddings. For this, we study mapping class groups, as they have many other rigidity properties. Given a connected orientable finite-type surface $\S$ without boundary, the mapping class group $\ModS$ of $\S$ is the group of orientation-preserving diffeomorphisms of $\S$ up to isotopy. In our study, we exclude some surfaces that are too small for their mapping class group to exhibit rigidity, the size of a surface being quantified by its topological complexity. For a surface $\S$ of genus $g$ with $p$ marked points, the \textit{topological complexity} of $\S$ is $\kappa(\S)=3g-3+p$. It is the maximal cardinality of a collection of pairwise disjoint curves on $\S$. Mapping class groups of surfaces of complexity $0$ are amenable, and those of complexity $1$ are measure equivalent to all nonamenable free groups. Therefore, the surfaces we consider have topological complexity at least $2$. 

On the measure-theoretic side, Kida beautifully showed in \cite{kida_measure_2010} that mapping class groups are superrigid for measure equivalence: the only countable groups that are measure equivalent to $\ModS$ are necessarily virtually isomorphic to $\ModS$ up to finite kernels. Our work is a continuation of Kida's work, using the rich theory Kida developed for measure equivalence of mapping class groups in \cite{kida_introduction_2009} and \cite{kida_measure_2010}. We use the measure-theoretic tools provided by Kida, so the new tools we develop specifically for the theory of measurable embeddings of mapping class groups are more combinatorial in nature.

On the self-embeddings side, Ivanov and McCarthy showed in \cite{ivanov_injective_1999} that mapping class groups are co-Hopfian. Geometrically, Bowditch \cite{bowditch_large-scale_2018} showed that when $\kappa(\S)\geq 4$, any quasi-isometric embedding of $\ModS$ into itself is a quasi-isometry. Our main theorem is the analogous result in the measure-theoretic setting.
\begin{Theorem}\label{main thm}
Let $\S$ be a connected orientable surface of finite type without boundary and with genus $g\ge 0$ and $p\ge 0$ marked points such that $3g-3+p\geq 2$. 

Then, any measurable embedding coupling from $\ModS$ to $\ModS$ is a measure equivalence coupling.
\end{Theorem}
This means that if one takes two free commuting measure-preserving actions of $\ModS$ admitting Borel fundamental domains on a common standard measured space $\Omega$, if one action has a finite measure fundamental domain, then so does the other.
Hence, in algebraic, geometric and measurable points of view, the mapping class group of a surface does not embed properly into itself. A consequence of our theorem together with Kida's theorem \cite{kida_measure_2010} is that we now know all the groups that are measurably bi-embeddable with $\ModS$.
\begin{Corollary}\label{cr biembedability}
    Let $\S$ be a connected orientable surface of finite type without boundary and with genus $g\ge 0$ and $p\ge 0$ marked points such that $3g-3+p\geq 2$.

    If a countable group $H$ measurably embeds into $\ModS$ and $\ModS$ measurably embeds into $H$, then $H$ is virtually isomorphic to $\ModS$ up to finite kernels.
\end{Corollary}

Given $G$ and $H$ two countable groups and $G\curvearrowright X$ and $H\curvearrowright X$ two probability measure-preserving actions on a standard probability space $X$, if there is an inclusion between $H$-orbits and $G$-orbits, then one can construct a measurable embedding between the groups (see e.g.\ \cite[Remark~2.36]{delabie_quantitative_2022}). This construction is based on the same ideas as the dictionary between measure equivalence and orbit equivalence (see e.g.\ \cite{furman_orbit_1999} and \cite[$\mrm{P_{ME}5}$]{gaboriau_examples_2005}). We however do not know if from a measurable embedding it is possible to get two p.m.p.\ actions with inclusion of orbits, so the setting of embeddings provides only an implication but not a dictionary. Using this implication and taking $G$ and $H$ to be the same mapping class group, we deduce that an inclusion of orbits forces an equality of orbits. Then, \cite{kida_orbit_2008} implies \textit{virtual conjugacy} (see \cite[Definition~1.3]{kida_orbit_2008}) of said actions when they are ergodic.

\begin{Corollary}\label{cr OE rigidity}Let $\S$ be a connected orientable surface of finite type without boundary and with genus $g\ge 0$ and $p\ge 0$ marked points such that $3g-3+p\geq 2$. Let $G$ and $H$ be two copies of $\ModS$. Assume that there exist free probability measure-preserving actions of $G$ and $H$ on a standard probability space $X$ such that for almost every $x\in X$, $H\cdot x\subseteq G\cdot x$. 
 
Then, for almost every $x\in X$, $H\cdot x=G\cdot x$. If additionally the actions are ergodic, then they are virtually conjugate.
\end{Corollary}

It is unknown in general whether measurable embeddings between two different mapping class groups exist, although we do know the following monotone invariance.
\begin{Corollary}\label{cr monotonicity of complexity}
    Let $\S$ and $\S'$ be connected orientable surfaces without boundary and of finite type. If $\ModS$ measurably embeds into $\Mod {\S'}$, then ${\kappa(\S)\leq \kappa(\S')}$. 
\end{Corollary}
Note that for two different surfaces $\S$ and $\S'$, if $\ModS$ embeds either algebraically or quasi-isometrically into $\Mod {\S'}$, then ${\kappa(\S)\leq \kappa(\S')}$ because $\kappa(\S)$ is the maximal rank of a free abelian subgroup of $\Mod{\S}$ (see \cite{behrstock_dimension_2008} for quasi-isometries).
Corollary~\ref{cr monotonicity of complexity} is a slight generalisation of {\cite[Corollary~5.29]{kida_introduction_2009}}.  Kida's key idea to prove it is the following: if a direct product of $n$ nonabelian free groups $F_2\times\cdots\times F_2$ measurably embeds into $\ModS$, then $n\leq\lfloor( \kappa(\S)+1)/2\rfloor$.
Additionally, the product of exactly $\lfloor( \kappa(\S)+1)/2\rfloor$ factors $F_2\times\cdots\times F_2$ embeds as a subgroup into $\ModS$ so this bound is optimal. This embedding can be seen combinatorially on the surface $\S$. Indeed, this bound corresponds to the maximal amount of components when cutting $\S$ into disjoint subsurfaces of complexity at least $1$, i.e.\ the smallest subsurfaces with a nonamenable mapping class group (which then contains $F_2$ as a subgroup). Cutting up $\S$ in this way is what we call a \textit{maximal decomposition} of $\S$, and the collection of pairwise disjoint curves used to cut up the surface is called a \textit{maximal decomposition multicurve}. Maximal decompositions of $\S$ will be the central tool in our study of self measurable embeddings of $\ModS$.
\medbreak
One important way to exhibit rigidity when proving Theorem~\ref{main thm} is to apply the theorem proven by Irmak \cite{irmak_superinjective_2006} and generalised by Shackleton \cite{shackleton_combinatorial_2007} that states that any induced subgraph embedding of the curve graph into itself is an automorphism of the curve graph. The \textit{curve graph} $\cxcourbes{\S}$ is an infinite graph whose vertices are essential simple closed curves on $\S$ up to isotopy and whose edges correspond to disjoint curves. The action of the mapping class group on the curve graph contains enough information to understand the group, as shown by Ivanov in the crucial theorem \cite[Theorem~1]{ivanov_automorphism_1997}: the group of automorphisms of the curve graph is virtually isomorphic to $\ModS$. These two theorems combined give a strong form of combinatorial rigidity for self-embeddings of curve graphs. To prove Theorem~\ref{main thm}, the goal becomes to build a self-embedding of the curve graph from the data of a self measurable embedding of the mapping class group, to which we then apply the latter combinatorial rigidity.

A crucial step in Ivanov and McCarthy's proof that $\ModS$ is co-Hopfian \cite{ivanov_injective_1999} is to show that any injective homomorphism from $\ModS$ into itself sends Dehn twists to powers of Dehn twists. For an essential simple closed curve $c$ (up to isotopy), the \textit{Dehn twist} associated to $c$ is a mapping class that acts as identity outside of an annulus around $c$ and twists the annulus around the curve. Similarly to Ivanov-McCarthy's proof, the crucial step of our proof is to show that a measurable embedding sends Dehn twists to powers of Dehn twists, in a measure-theoretic sense that we will define below. After this step, we are able to construct a measurable and equivariant map from the self measurable embedding coupling to the set of self-embeddings of the curve graph. This map takes values in the finite-index overgroup $\operatorname{Aut}\,\cxcourbes{\S}$ of $\ModS$, by the combinatorial rigidity of self-embeddings of the curve graph \cite{ivanov_automorphism_1997,irmak_superinjective_2006,shackleton_combinatorial_2007} explained above. Thus, we obtain a map that sends the abstract self measurable embedding coupling $\Omega$ to the \textit{tautological} coupling of $\ModS$ acting by right and left multiplication on $\operatorname{Aut}\,\cxcourbes{\S}$. Such a map is often called a \textit{tautening map} and provides enough information on $\Omega$, namely that it is a measure equivalence coupling. This part of the proof is inspired by the fact that the existence of a tautening map for a self measure equivalence coupling has been a crucial step to show measure equivalence rigidity, an idea originally developed in \cite{furman_gromovs_1999} and then appearing for instance in \cite{monod_orbit_2006,kida_measure_2010,bader_integrable_2013,guirardel_measure_2025,escalier_graph_2026}.

We now define what it means for a subgroup to be sent to another subgroup by the measurable embedding, as there is only a coupling of the two groups, and no direct map from the source group $H$ to the target group $G$. For this, we work with the groupoid-theoretic point of view on measure equivalence developed by Kida in \cite{kida_introduction_2009} and that has become classical when working on measure equivalence rigidity. Given a measurable embedding coupling from a group $H$ to a group $G$ with respective fundamental domains $D_H$ and $D_G$, there are actions $G\curvearrowright D_H$ and $H\curvearrowright D_G$. Then, on the finite measure space $X=D_H\cap D_G$, the orbits of these two actions coincide (see the proof of Lemma \ref{lemma measurable embedding provides groupoid and cocycles} for details). The common collection of arrows on $X$ forms a \textit{measured groupoid} $\mclG$, with two cocycles $\w : \mclG\longrightarrow H$ and $\rho : \mclG\longrightarrow G$ that label the arrows by the respective actions of $H$ and $G$. The cocycle $\w$ comes from the action of $H$ on the space $D_G$ of finite measure, so it as well-behaved as in a measure equivalence coupling. The asymmetry of measurable embeddings as opposed to measure equivalence comes in the fact that $\rho$ loses any form of surjectivity onto $G$.
\medbreak
From now on, we study self measurable embeddings of mapping class groups, so let us fix ${H\simeq G\simeq \ModS}$.
The first copy $H$ of $\ModS$ at the target of $\w$ is the \textit{source} of the measurable embedding and the second copy $G$ at the target of $\rho$ is the \textit{target} of the measurable embedding. In our proof, a general idea is to fix some suitable information on the source copy of $\S$ and prove that this information is transported through the measurable embedding onto the target copy of $\S$. In practice, a subgroup $H'$ of $H$ is \textit{transported into} a subgroup $G'$ of $G$ by the measurable embedding on some positive measure Borel subset $U$ of $X$ when 
$\rho\left(\w^{-1}(H')|_U\right)\subseteq G'$.

The key step of the proof (Proposition~\ref{prop recap courbes dans courbes}) is to show that a cyclic subgroup generated by a Dehn twist $T_c$ is transported into another Dehn twist cyclic subgroup by the measurable embedding, up to restricting to a Borel subset $U$ of $X$. However, the subgroupoid $\w^{-1}\left(\gen{T_c}\right)|_U$ alone does not provide enough rigidity to get information through $\rho$. For instance, it could a priori be sent by $\rho$ into a cyclic subgroup generated by a pseudo-Anosov mapping class of $\S$ that fixes no curve on the surface. The image of $\w^{-1}\left(\gen{T_c}\right)|_U$ by $\rho$ could also be very far from being a subgroup of $\ModS$. 

From this point on, the proof becomes different for surfaces of complexity $2$ as opposed to surfaces of complexity at least $3$. When $\kappa(\S)=2$, there is a proof derived directly from Kida's work (see the proof of Proposition~\ref{nonsep curves prop}). When $\kappa(\S)\geq 3$, new combinatorial arguments are needed to force rigidity. The right objects to apply $\w^{-1}$ to are maximal decompositions of $\S$. The first step of our proof is to reinforce Kida's measure-theoretic work on maximal decompositions in {\cite[Section~5]{kida_mapping_2008}}. We show that given a maximal decomposition multicurve $a_1,\dots,a_{n-1}$, there exists another maximal decomposition multicurve $\a_1,\dots,\a_{n-1}$ such that 
$$\rho\left(\w^{-1}(\gen{T_{a_i}})|_U\right)\subseteq \gen{T_{\a_1},\dots,T_{\a_{n-1}}}$$
for every $i$ and for some positive measure Borel subset $U$ of $X$ (see Point~\ref{ii max decompos} of Lemma~\ref{lemma max decompo are sent to max decompo}).

Then, the goal becomes to eliminate all but one twist from this multitwist in the image. For this, we first work on a specific type of separating curves of $\S$ which we call \textit{leaf curves}. A leaf curve separates a subsurface of complexity $1$ (or sometimes $2$) from the rest of $\S$. It is called this way because the maximal decompositions we study have \textit{dual graphs} that are finite trees (see Definition~\ref{def dual graph}), and leaf curves correspond to leaves of such a tree. Showing that any Dehn twist along a leaf curve is sent to exactly one Dehn twist via the measurable embedding constitutes the most technical part of the proof (Proposition~\ref{prop leaf type curves}). For this, we study all the possible images of maximal decompositions through the measurable embedding. Then, we study interactions between different intersecting maximal decomposition multicurves. After proving that a leaf curve is sent to a leaf curve, dealing with other curves is based on the following idea: a curve of $\S$ is uniquely determined by the leaf curves that are disjoint from it. Thus, once we know the image of leaf curves by the measurable embedding, we can deduce the image of other types of curves. The idea of obtaining some rigidity for leaf curves and then passing to a broader category of curves appears in work on the strongly separating curve graph by Bowditch \cite{bowditch_rigidity_2016} and on quasi-isometric rigidity of quotients of mapping class groups by Mangioni and Sisto \cite[Sections~6,~7]{mangioni_rigidity_2025}.

\paragraph{Structure of the paper.} In Section~\ref{prelims mapping class groups}, we first give preliminaries on mapping class groups and surfaces and then we study the combinatorics of maximal decompositions of a given surface. Sections~\ref{section groupoids} and \ref{subsection meas embeddings} consist of preliminaries on measured groupoids and measurable embeddings. In Section~\ref{section adams and free into direct products}, we adapt the classical Adams argument \cite{adams_indecomposability_1994} to measurable embeddings and we prove a new Lemma~\ref{lemma free product into direct product} on measurable embeddings that may be of interest to the reader independently from mapping class groups. Section~\ref{section groupoids and mapping class groups} gives an overview of Kida's classification of groupoids with cocycles to mapping class groups. Section~\ref{section proof theorem} contains the proof of Theorem~\ref{thm existence of tautening map}, which is a stronger and more technical version of our main theorem. Finally, Section~\ref{conclusion subsection} gives the proof of our main Theorem~\ref{main thm} as well as Corollaries~\ref{cr biembedability}, \ref{cr OE rigidity} and \ref{cr monotonicity of complexity}.

\paragraph{Acknowledgements.} The author warmly thanks her PhD adviser Camille Horbez for their many mathematical discussions, as well as for his constant support and enthusiasm. She thanks Adrien Abgrall for showing her how to make illustrations, Sasha Bontemps for a helpful discussion about maximal decompositions of surfaces and Giorgio Mangioni for pointing her to \cite{bowditch_rigidity_2016}. She is also grateful to her colleagues who came to her first talk about this project.

\paragraph{Statement on large language models.}
Every proof of this project was found without large language models (LLM). No sentence or word of this article was written by an LLM, everything was written and drawn by the author.
While working on this project, out of curiosity, she did prompt an LLM once to prove Corollary~\ref{prop in modz all twists are free groups}. and thrice to find references and bibliography. Each time, it was inconclusive and none of the answers were used in this text.

\section{Mapping class groups and combinatorics of surfaces}\label{prelims mapping class groups}

\subsection{Surfaces and their mapping classes}

\paragraph{Surfaces.} A \textit{surface} $\S$ is an orientable $2$-manifold of finite type, where \textit{finite type} means that the fundamental group of $\S$ is finitely generated. Up to homeomorphism, there is exactly one connected surface \textit{of type} $\S_{g,p,b}$ with genus $g$, $p$ marked points and $b$ boundary components. Here, we work with $p$ non necessarily zero, so any surface of type $\S_{0,p,b}$ with $p+b=3$ will be called a \textit{pair of pants}. More generally, we call \textit{$n$-punctured sphere} any surface of type $\S_{0,p,b}$ with $p+b=n$. Similarly, an \textit{$n$-punctured torus} is a surface of type $\S_{1,p,b}$ with $p+b=n$.

For a surface $\S$, a \textit{curve} of $\S $ is the isotopy class of an injective map $S ^1\longrightarrow \S$ that does not bound a disk in $\S$ with $0$ or $1$ marked points and that is not isotopic to a boundary component. We denote by $\cxcourbes{\S}$ the set of curves on $\S$. To be clear, when we say curve, we will always mean isotopy class of essential simple closed curves. When we say that two curves are distinct, we always mean that they are distinct isotopy classes.

There is a well-defined minimal geometric intersection map
$$i : \begin{array}{rcl}
     \cxcourbes{\S}^2 &\longrightarrow &\N  \\
     (a,b)&\longmapsto & \ds \min_{\substack {\alpha\in a \\ \beta\in b }}~\#\alpha\cap\beta~.
\end{array}$$
We then say that two curves $a,b\in\cxcourbes{\S}$ are \textit{disjoint} when $i(a,b)=0$, i.e.\ there exist isotopy representatives that are actually disjoint on $\S$. Note that a curve is thus disjoint from itself. In the contrary, two curves $a,b\in\cxcourbes{\S}$ \textit{intersect} if and only if $i(a,b)\geq 1$. A \textit{multicurve} of $\S$ is a collection of pairwise disjoint curves of $\S$. We say that two multicurves are \textit{disjoint} if their union is a multicurve and \textit{intersect} if their union is not a multicurve.

The \textit{topological complexity} of a connected surface $\S$ of type $\S_{g,p,b}$ is $\kappa(\S)=3g+p+b-3$. It is the number of curves in a pants decomposition of the surface, i.e,\ the biggest cardinality of a multicurve of $\S$. It is also the maximal rank of a free abelian subgroup of $\ModS$. We choose to define $\kappa$ as in \cite{shackleton_combinatorial_2007} and \cite{bowditch_large-scale_2018} but others (for instance \cite{kida_measure_2010}) choose to define the topological complexity of a surface as the simplicial dimension of $\cxcourbes{\S}$. In our case, the simplicial dimension of $\cxcourbes{\S}$ is $\kappa(\S)-1$ because a maximal simplex of $\cxcourbes{\S}$ has $\kappa(\S)$ vertices, so it is of dimension $\kappa(\S)-1$.
Moreover, we will use the notation $\dsig$ to be $0$ when $\kappa(\S)$ is odd and $1$ when $\kappa(\S)$ is even.

The set $\cxcourbes{\S}$ has a graph structure and is called the \textit{curve graph}. For simplicity of notation, we give the same notation $\cxcourbes{\S}$ to the set of vertices of this graph as to the graph itself. Two curves $a,b\in\cxcourbes{\S}$ are adjacent in the graph $\cxcourbes{\S}$ if and only if they are distinct and $i(a,b)=0$. For instance, a multicurve of $\S$ is a collection of vertices of $\cxcourbes{\S}$ that spans a complete subgraph. Crucially, the curve graph of a connected surface $\S$ of positive complexity is hyperbolic, as proven by Masur and Minsky \cite{masur_geometry_1999}. Thus, there is a well-defined boundary $\partial\cxcourbes{\S}$, which is not compact because the graph $\cxcourbes{\S}$ is locally infinite.

A \textit{subsurface} $Q$ of $\S$ is the isotopy class of an embedded surface in $\S$. The \textit{boundary} $\partial Q$ of $Q$ is the multicurve of $\S$ composed of the isotopy classes of $\bar Q\setminus \mathring Q$. Note that $Q$ can be open, closed or neither, it depends on which curves of $\partial Q$ are in $Q$. For instance, if $Q$ is open then it has $\# \partial Q$ marked points and when $Q$ is closed, it has $\#\partial Q$ boundary components. 
We say that a curve $c\in\cxcourbes{\S}$ \textit{intersects} the subsurface $Q$ \textit{nontrivially} when $c$ and $Q$ do not have disjoint isotopy representatives on $\S$ and we say that $c$ and $Q$ are \textit{disjoint} if they do have disjoint isotopy representatives. For instance, the multicurve $\partial Q$ is disjoint from $Q$.

When $Q$ is a subsurface of $\S$, $\cxcourbes{Q}$ is the set of curves of $Q$ as a surface, so in particular $\cxcourbes{Q}$ is a subset of $\cxcourbes{\S}$ that does not contain $\partial Q$. The graph $\cxcourbes{Q}$ is a subgraph of $\cxcourbes{\S}$, so two curves are adjacent in $\cxcourbes{Q}$ when they are disjoint. This means that in our setting when $\kappa(Q)=1$, there are no edges in $\cxcourbes{Q}$ and when $\kappa(Q)=0$, $\cxcourbes{Q}$ is empty. 

Given a subset $\mclC\subseteq\cxcourbes{\S}$, the subsurface $\S_{\mclC}$ \textit{filled} by $\mclC$ is the closed subsurface that contains $\mclC$ and such that for any connected component $S$ of $\S_{\mcl C}$, 
\begin{itemize}[noitemsep, nolistsep, label=-]
    \item $S$ is not a pair of pants,
    \item if $S$ is an annulus, $\partial Q\in\mcl C$,
    \item if $S$ is not an annulus, $\forall c\in\cxcourbes{Q}, \exists c'\in\mcl C,~ i(c,c')\geq 1$.
\end{itemize}

For $Q$ a subsurface of $\S$, denote by $Q^*$ the subsurface of $\S$ that is filled by all the curves that are disjoint from $Q$.

\paragraph{Mapping class groups.} 

\begin{Def}\label{def mapping class group, subgroups of subsurfaces, modz def}
    Let $\S$ be a connected surface.
    \begin{itemize}
        \item The \textit{mapping class group} of $\S$ consists of isotopy classes of orientation-preserving diffeomorphisms of $\S$: $$\ModS = \mrm{Diffeo^+}(\S)/\,\mrm{isotopy}~.$$
        An element $f\in\ModS$ is called a \textit{mapping class}.
        \item The \textit{extended mapping class group} of $\S$ is$$\Modext = \mrm{Diffeo}(\S)/\,\mrm{isotopy}$$
        where $\mrm{Diffeo}(\S)$ consists of diffeomorphisms of $\S$ that preserve or invert the orientation. The group $\ModS$ is an index $2$ subgroup of $\Modext$.
        \item Define $\Modz$ as the subgroup of $\ModS$ acting trivially on the first homology group of $\S$ modulo 3, i.e.\ 
        $$\Modz=\ker\big[\ModS\longrightarrow\mrm{Aut}(\mrm{H}_1(\S,\Z/3\Z))\big]~.$$
        \item Let $Q$ be a subsurface of $\S$. Then, $\Modact{Q}$ is the subgroup of $\Modz$ consisting of the mapping classes of $\Modz$ that are identity outside of the closure $\bar Q\subseteq \S$ of the subsurface $Q$ of $\S$.
    \end{itemize}
\end{Def}

A mapping class $f$ of $\ModS$ is said to be \textit{pure} if there exists a (possibly empty) multicurve $\s$ on $\S$ such that every connected component $Q$ of $\S\setminus \s$ is $f$-invariant and either $f$ fixes every curve of $Q$ or $f$ fixes no curve of $Q$. In practice, we will work with $\Modz$ instead of $\ModS$ because it has only pure mapping classes, as shown by Ivanov.

\begin{Th}[Ivanov {\cite[Theorem~3]{ivanov_subgroups_1992}}, Handel--Mosher {\cite{handel_subgroup_2020}}]\label{modz th}
    Let $\S$ be a connected surface with $\kappa(\S)\geq 1$.
    
    Then, $\Modz$ is a finite-index torsion-free subgroup of $\ModS$ consisting of only pure mapping classes.
\end{Th}
\begin{Rem}
    Ivanov works on compact surfaces. When the surface has marked points or boundary components, the latter theorem can be deduced from the Handel--Mosher theorem \cite{handel_subgroup_2020}, by noticing that in this case $\ModS$ has a well-behaved map to the outer automorphism group of some free group.
\end{Rem}

The extended mapping class group acts on the curve graph $\cxcourbes{\S}$ by graph automorphisms because it sends disjoint curves to disjoint curves. The group of graph automorphisms of $\cxcourbes{\S}$ is denoted $\mrm{Aut}\,\cxcourbes{\S}$. By the following theorem of Ivanov, generalised by Luo, it is enough to study automorphisms of the curve graph to get information on the mapping class group.

\begin{Th}[Ivanov {\cite[Theorem~1]{ivanov_automorphism_1997}}, Luo \cite{luo_automorphisms_2000}]\label{th ivanov}
Let $\S$ be a connected surface with $\kappa(\S)\geq 2$.
\begin{itemize}[noitemsep, nolistsep]
    \item If $\kappa(\S)=2$, then $\Modext \longrightarrow \mrm{Aut}\,\mcl{C}(\S)$ has finite index image and finite kernel.
    \item If $\kappa(\S)\geq 3$, then $\Modext \simeq \mrm{Aut}\,\mcl{C}(\S)$ is an isomorphism.
\end{itemize}
\end{Th}
Note that by Ivanov's Theorem~\ref{modz th}, the subgroup $\Modz$ has no torsion so it can be seen directly as a finite-index subgroup of $\mrm{Aut}\,\cxcourbes{\S}$ via the injective homomorphism from $\Modz$ into $\mrm{Aut}\,\mcl{C}(\S)$.

\paragraph{Dehn twists.} The twist on the annulus $S^1\times[1,2]$ is defined by ${(\theta,t)\longmapsto(\theta+2\pi(t-1),t)}$. For an annulus subsurface $A$ of a surface $\S$, take a diffeomorphism $\varphi_A$ of $\S$ that is isotopic to the twist on $A$ and to the identity on $\S\setminus A$. The \textit{Dehn twist} around a curve $c\in\cxcourbes{\S}$ is the isotopy class of $\varphi_A$ when $A$ is an annulus neighborhood of an essential simple closed curve in the isotopy class $c$. 

The Dehn twist around a curve $c$ is denoted $T_c$ and is an element of $\ModS$. A \textit{multitwist} around a multicurve $\s$ of $\S$ is the free abelian subgroup $\gen{T_c}_{c\in\s}$ of $\ModS$ generated by the Dehn twists around the curves in $\s$.
For example, when $Q$ is a subsurface of $\S$, then we get the following multitwist
    $$\Modact{Q}\cap\Modfix{Q}=\gen {T_c}_{c\in\partial Q}~.$$

In the following proposition, we list useful properties of Dehn twists. For more information on Dehn twists, see e.g.\ \cite[Chapter~Three]{farb_primer_2012}.

\begin{Prop}\label{denh twist prop} Let $\S$ be a connected surface.
    \begin{enumerate}[noitemsep, nolistsep]
        \item \label{intersection formula dehn twists}\cite[Proposition~3.2]{farb_primer_2012} For all $a,b\in\cxcourbes{\S}$ and $k\in\Z$,
        $$i(T_a^k(b),b)=|k|\,i(a,b)^2~.$$
        \item \label{dehn twists injectivity}\cite[Fact~3.9]{farb_primer_2012} In particular, Dehn twists have infinite order and for every $a,b\in\cxcourbes{\S}$ and $m,n\in\Z\setminus\{0\}$,
        $$T_a^m=T_{b}^n\iff a=b\text{ and }m=n~.$$
        \item \label{dehn twists squared give free groups}\cite[Theorem~3.14]{farb_primer_2012} If $i(a,b)\geq 2$, then $T_a$ and $T_b$ generate $F_2$. If $i(a,b)=1$ and $j,k$ are such that $|j|,|k|\geq 2$, then $T_a^j$ and $T_b^k$ generate $F_2$.
    \end{enumerate}
\end{Prop}

\begin{Prop}\label{nonsep dehn twists are not in modz}
    Let $\S$ be a connected surface and let $T_a$ be a Dehn twist around a nonseparating curve $a$. If $T_a^k\in \Modz$, then $|k|\geq 2$.
\end{Prop}
\begin{proof}
    The curve $a$ being nonseparating, there exists a curve $b$ on $\S$ such that $i(a,b)=1$. Then, the homology classes $[a]$ and $[b]$ can be completed into a basis of $\mrm{H}_1(\S,\Z/3\Z)$ and by \cite[Proposition~6.3]{farb_primer_2012}, 
    $$T_a\cdot[b]=[a]+[b]\text{ and }T_a^{-1}\cdot[b]=[a]-[b]~.$$
    Thus, both $T_a$ and $T_a^{-1}$ act nontrivially on $\mrm{H}_1(\S,\Z/3\Z)$ so they are not in $\Modz$.
\end{proof}

\begin{Cr}\label{prop in modz all twists are free groups}
    Let $\S$ be a connected surface. For every $a,b$ curves such that $i(a,b)\geq 1$, ${\gen{T_a}\cap \Modz}$ and $\gen{T_b}\cap \Modz$ generate a nonabelian free group.
\end{Cr}
\begin{proof}
    When $i(a,b)\geq 2$, $T_a$ and $T_b$ generate a nonabelian free group by Point~\ref{dehn twists squared give free groups} of Proposition~\ref{denh twist prop}. When $i(a,b)=1$, by Proposition~\ref{nonsep dehn twists are not in modz}, $\gen{T_a}\cap \Modz=\gen{T_a^j}$ and $\gen{T_b}\cap \Modz=\gen{T_b^k}$ for some $|j|,|k|\geq 2$. By Point~\ref{dehn twists squared give free groups} of Proposition~\ref{denh twist prop}, $\gen{T_a^j,T_b^k}\simeq F_2$ a nonabelian free group of two generators.
\end{proof}

\subsection{Maximal decompositions of a surface}\label{section combinatoire max decompos}

A useful tool, generalising Kida's work in \cite[Section~5.4]{kida_introduction_2009}, is maximal decompositions of a given surface $\S$ into subsurfaces with nonamenable mapping class groups. The smallest complexity surfaces that have nonamenable mapping class groups are the surfaces of complexity $1$, which are exactly $\S_{0,p,b}$ with $p+b=4$ or $\S_{1,p,b}$ with $p+b=1$. In other words, these are $4$-punctured spheres or once-punctured tori.

It will not always be possible to cut up a surface into only pieces of complexity $1$. We will show that this is possible when the complexity is odd. When the complexity is even, one piece of the decomposition will be of complexity $2$, which are exactly $\S_{0,p,b}$ with $p+b=5$ ($5$-punctured spheres) or $\S_{1,p,b}$ with $p+b=2$ (twice-punctured tori).

\begin{Def}\label{def dual graph}
    Let $\s$ be a multicurve of a connected surface $\S$. The graph dual to $\s$, denoted $\G_{\s}=(E,V)$, is defined as follows.
    \begin{itemize}[noitemsep, nolistsep, label=-]
        \item The vertices $V$ correspond to closures of connected components of $\S\setminus \s$.
        \item The edges $E$ correspond to curves of $\s$.
        \item A curve $c\in \s$ connects as an edge the two vertices of $V$ that correspond to the closures of components of $\S\setminus \s$ that share $c$ in their boundary.
    \end{itemize}
\end{Def}

\begin{Rem}\label{rem equation star}
    The graph $\G_{\s}$ is finite, non oriented and can have loops. Each vertex $v$ is a surface with some topological complexity $\kappa(v)$. Recall that the topological complexity is the number of curves needed for a pants decomposition. Therefore, 
    $$\kappa(\S)=\#E+\sum_{v\in V} \kappa(v)~.$$
    In other words, if we label each vertex $v$ by $\kappa(v)$ and each edge by $1$, then the complexity of the whole surface is simply the sum of all labels.
    Now, denote by $b_1$ the number of edges outside of a covering tree of $\G_{\s}$. Then,
    $$\#V-\#E=1-b_1~.$$
    This gives us the following useful equation on the dual graph $\G_{\s}$:
    \begin{equation*}\label{equation star}
        \kappa(\S)+1=2\#V+\sum_{v\in V}\left(\kappa(v)-1\right) + b_1~.\tag{$\bigstar$}
    \end{equation*}
\end{Rem}

First, let us describe the shape of decompositions of a surface into the maximal amount of components with nonamenable mapping class groups.

\begin{Lemma}\label{lemma n(g,p)}

Fix a connected surface $\S$ with complexity $\kappa(\S)\ge 1$. Let 
$$\nsig=\left\lfloor\frac{\kappa(\S)+1}{2}\right\rfloor~.$$
Let $\s$ be a multicurve of $\S$. Let $\G_{\s}=(V,E)$ be the graph dual to $\s$. Assume that for every component $v\in V$, $\kappa(v)\ge 1$. 

Then, $\#V\leq \nsig$. Moreover, if $\#V=\nsig$ then the decomposition has the following form.
\begin{itemize}
    \item When $\kappa(\S)$ is odd, $\G_{\s}$ is a tree and each component of $\S\setminus \s$ has complexity $1$.
    \item When $\kappa(\S)$ is even, the decomposition has exactly one of the two following forms.
    \begin{itemize}[noitemsep, nolistsep]
        \item The graph $\G_{\s}$ is a tree with exactly one component of complexity $2$ and all the other components have complexity~$1$.
        \item The graph $\G_{\s}$ has exactly one edge outside a covering tree, and all components have complexity~$1$.
    \end{itemize}
\end{itemize}
\end{Lemma}

\begin{proof}
    Recall (\ref{equation star}): 
        $$\kappa(\S)+1=2\#V+\sum_{v\in V}\left(\kappa(v)-1\right) + b_1~.$$
    On one hand, all the $(\kappa(v)-1)$ and $b_1$ are positive or null so 
    $$\#V\leq \left\lfloor\frac{\kappa(\S)+1}{2}\right\rfloor=\nsig.$$

    On the other hand, to maximise $\#V$ when $\kappa(\S)$ is odd, the only solution is to take $\kappa(v)=1$ for all $v$ and $b_1=0$, so we indeed get a tree of complexity $1$ components. When $\kappa(\S)$ is even, either there is a cycle ($b_1=1$) and all components are complexity $1$ or it is a tree ($b_1=0$) and there is exactly one complexity $2$ component.
\end{proof}

\begin{Rem}\label{rem si decompo est un arbre à petites composantes alors cest maximal}
    Recall that $\dsig$ is defined so that $\nsig=\left(\kappa(\S)+1-\dsig\right)/{2}$.
    If a decomposition multicurve $\s$ of $\S$ verifies
    \begin{itemize}
        \item $\G_{\s}$ is a tree,
        \item $\sum_{v\in V(\G_{\s})}\left(\kappa(v)-1\right)=\dsig$,
    \end{itemize}
    then (\ref{equation star}) gives
        $\kappa(\S)+1=2\#V+\dsig$
    so that $\#V=\nsig$. Thus $\G_{\s}$ necessarily has $\nsig$ vertices.
\end{Rem}

\begin{Ex}\label{exemple de decompo pour sphere n trous}
Let $\S$ be an $n$-punctured sphere, i.e.\ a surface of type $\S_{0,p,b}$ with ${b+p=n}$ and assume that $n\geq 4$. Let us build a decomposition $\s$ of $\S$ by induction on $n$ such that $\G_{\s}$ is a tree with $$\sum_{v\in V(\G_{\s})}\left(\kappa(v)-1\right)=\dsig~.$$ 
\begin{itemize}
    \item If $n\in\{4,5\}$, the decomposition multicurve is taken empty.
    \item If $n\geq 6$, take a curve $c$ of $\S$ such that the connected components of $\S\setminus c$ are a $4$-punctured sphere $S$ and an $(n-2)$-punctured sphere $\S'$. Then, apply the induction hypothesis to $\S'$ and get a decomposition multicurve $\s'$, as illustrated in Figure~\ref{fig decompo sphere n trous}. Notice that $\dsig=\delta_{\S'}$ so we have $$\sum_{v\in V(\G_{\s'})}\left(\kappa(v)-1\right)=\dsig~.$$  Define $\s=\s'\sqcup\{c\}$ and notice that $\G_{\s}$ is the tree $\G_{\s'}$ with an additional leaf corresponding to the $4$-punctured sphere $S$. The complexity of $S$ is $1$, so we still have $$\sum_{v\in V(\G_{\s})}\left(\kappa(v)-1\right)=0+\sum_{v\in V(\G_{\s'})}\left(\kappa(v)-1\right)=\dsig~.$$
    \begin{figure}[ht]
	\centering
	\includesvg[width=0.6\textwidth]{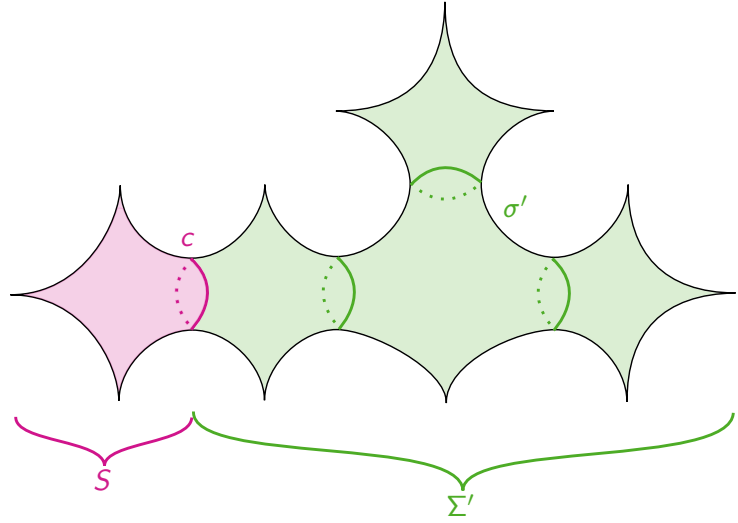}
    \caption{Induction step while constructing a $0$-maximal decomposition in Example~\ref{exemple de decompo pour sphere n trous}.}
	\label{fig decompo sphere n trous}
\end{figure}
\end{itemize}
Note that $\G_{\s}$ has $\nsig$ vertices by Remark~\ref{rem si decompo est un arbre à petites composantes alors cest maximal}. Thus, this construction illustrates Lemma~\ref{lemma n(g,p)} in the case of a tree dual graph for an $n$-punctured sphere.
\end{Ex}

It will not be enough to study decompositions of a surface with a maximal number of components, so let us define $k$-maximal decompositions. Recall that $\dsig=1$ when $\kappa(\S)$ is even and $\dsig=0$ when $\kappa(\S)$ is odd. See Figure~\ref{fig 1max et 0max pour 4 trous} for an example of both a $0$-maximal and a $1$-maximal decompositions for a surface of genus $4$.

\begin{Def}\label{def kmax decompo}
    A $k$-maximal decomposition of a connected surface $\S$ is a multicurve $\s$ such that $\S\setminus \s$ has $\nsig-k$ connected components, of which exactly $2k+\dsig$ components have complexity~$2$.
\end{Def}

\begin{figure}[ht]
	\centering
	\includesvg[width=0.6\textwidth]{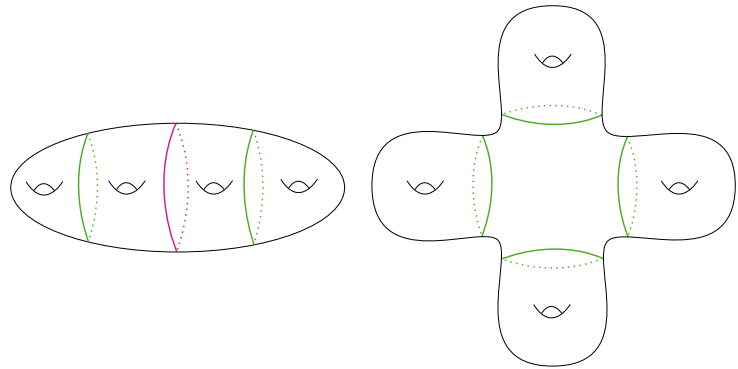}
    \caption{A $1$-maximal decomposition on the left and a $0$-maximal decomposition on the right for the surface of genus $4$.}
	\label{fig 1max et 0max pour 4 trous}
\end{figure}

It follows from Lemma~\ref{lemma n(g,p)} that $0$-maximal decompositions of the surface $\S$ are exactly the decompositions that have a tree as dual graph and have a maximal amount $\nsig$ of components.

\begin{Lemma}\label{lemma structure of k max decopos} Let $\s$ be a $k$-maximal decomposition of a connected surface $\S$.
    \begin{enumerate}[noitemsep, nolistsep]
        \item \label{point 1 lemma structure of k max decopos}The dual graph $\G_{\s}$ is a tree. The components that are not of complexity $2$ have complexity~$1$. 
        \item \label{point 2 lemma structure of k max decopos}There is no curve $c$ disjoint from $\s$ such that the multicurve $\s\sqcup\{c\}$ gives a decomposition with more connected components of complexity at least $1$.
    \end{enumerate}
\end{Lemma}

\begin{proof}Denote $\G_{\s}=(V,E)$ and recall (\ref{equation star}):
    $$\kappa(\S)+1=2\#V+\sum_{v\in V}\left(\kappa(v)-1\right) + b_1~.$$
    By definition, 
    $$2\nsig+\dsig=\kappa(\S)+1~.$$
    Note that $\#V=\nsig-k$ and $$\sum_v\left(\kappa(v)-1\right)\geq 2k+\dsig$$
    which gives $$2\# V+\sum_v\left(\kappa(v)-1\right)\geq 2\nsig-2k+2k+\dsig =\kappa(\S)+1~.$$
    Thus, the above inequality is an equality by equation (\ref{equation star}), which forces $b_1$ to be zero and $$\sum_v\left(\kappa(v)-1\right)=2k+\dsig$$
    which means that all the other components have complexity~$1$. This proves Point~\ref{point 1 lemma structure of k max decopos}.

    If $c$ is a curve on a subsurface $S$ of complexity $1$ or $2$, then $S\setminus c$ has at most one component of positive complexity, which proves Point~\ref{point 1 lemma structure of k max decopos}.
\end{proof}

We distinguish three different types of separating curves according to the role they play in a maximal decomposition. In a $k$-maximal decomposition, a leaf curve is a curve that separates one component of the decomposition from the rest.

\begin{Def}\label{def types de courbes}
        A separating curve $c$ on a connected surface $\S$ is said to be:
    \begin{itemize}[noitemsep]
        \item a \textit{leaf curve} if one of the connected components of $\S\setminus c$ is of complexity $1$ or $1+\dsig$;
        \item \textit{strongly separating} if none of the two connected components are of complexity $0$;
        \item \textit{weakly separating} if one of the two connected components is of complexity $0$.
    \end{itemize}
    Denote by $\mclC_{leaf}$ the subset of $\cxcourbes{\S}$ composed of leaf curves and $\mclC_{nl}$ the complement, i.e. the \textit{non-leaf curves}. A subsurface of $\S$ of complexity $1$ or $1+\dsig$ having one leaf curve as its boundary is called a \textit{leaf subsurface}.
\end{Def}

Note that in Figure~\ref{fig 1max et 0max pour 4 trous}, the central pink curve of the $1$-maximal decomposition is a strongly separating curve that cannot be completed into a $0$-maximal decomposition. All the green curves of Figure~\ref{fig 1max et 0max pour 4 trous} are leaf curves.

\begin{Lemma}\label{lemma strongly sep curves into decompo}
Any strongly separating curve can be completed into a multicurve which is either a $0$-maximal decomposition or a $1$-maximal decomposition.
\end{Lemma}

\begin{proof}
    Let $c$ be a strongly separating curve of $\S$.
    
Assume that $\kappa(\S)$ is even. Let us show that $c$ can be completed into a multicurve which is a $0$-maximal decomposition. Indeed, one component of $\S\setminus c$ has odd complexity, so there are only complexity $1$ surfaces in a $0$-maximal decomposition and the other component of $\S\setminus c$ has even complexity so exactly one complexity $2$ component in a $0$-maximal decomposition. Gluing these two components together, we get a $0$-maximal decomposition of the whole surface.

Now, assume that $\kappa(\S)$ is odd. Let us show that $c$ can be completed into a multicurve which is either a $0$-maximal decomposition or a $1$-maximal decomposition. Indeed, either both components of $\S\setminus c$ have odd complexities and the $0$-maximal decompositions on both sides have only complexity $1$ components, so gluing them together gives a $0$-maximal decomposition. Or both components of $\S\setminus c$ have even complexities so each side has one complexity $2$ component, which makes the total decomposition $1$-maximal because it contains two complexity $2$ surfaces.
\end{proof}

We record the following proposition that we will use only in the proof of Proposition~\ref{prop leaf type curves}. Note that the proof of this proposition is constructive so in particular it gives examples of $0$-maximal decompositions.

\begin{Prop}\label{prop D les decompos particulieres avec g+p/3}
Let $\S$ be a connected surface without boundary and with $\kappa(\S)\geq 3$. Let $D$ be the collection of $0$-maximal decomposition multicurves $\s$ of $\S$ verifying
$$\# \left(\s\cap \mclC_{leaf}\right) >\# \left(\s\cap \mclC_{nl}\right)~.$$
Then, any leaf curve of $\S$ can be completed into a multicurve in $D$.
\end{Prop}
\begin{proof}
    Fix $g$ to be the genus of $\S$ and $p$ the number of marked points of $\S$. Fix $c$ a leaf curve of $\S$ and fix $S$ the open leaf subsurface with boundary $c$. By definition, the subsurface $S$ has complexity $1$ or $1+\dsig$. Let us prove the proposition by constructing a multicurve that contains $c$ and that is in $D$.
\medbreak
\textbf{Case 1: $\kappa(\S)\leq 6$.} When $\S$ has complexity at most $6$, we show that any $0$-maximal decomposition of $\S$ is in $D$. Indeed, any $0$-maximal decomposition $\s$ of $\S$ (for instance, one that contains $c$) has a dual graph $\G_{\s}$ which is a tree with $\nsig$ vertices by Lemma~\ref{lemma structure of k max decopos}. But when $\kappa(\S)\leq 6$,
\begin{align*}
    \nsig=\frac{\kappa(\S)+1-\dsig}{2}\leq 3
\end{align*}
so $\G_{\s}$ is a line of at most three vertices. Thus, all edges of $\G_{\s}$ are connected to a leaf. This means exactly that $\s$ has only leaf curves.

If $\kappa(\S)\leq 4$, then $\{c\}$ gives a $0$-maximal decomposition of $\S$. When $\kappa(\S)\in\{5,6\}$, then for any leaf curve $c'$ in $\cxcourbes{\S\setminus S}$, the multicurve $\{c,c'\}$ is a $0$-maximal decomposition of $\S$.
\medbreak
\textbf{Case 2: $\kappa(\S)> 6$.}\begin{Claim}\label{claim subcase 2a 2b}
    There exists a $0$-maximal decomposition $\s$ of $\S$ that contains $c$ and verifies $$\# \left(\s\cap \mclC_{leaf}\right) \geq 2\nsig/3~.$$
\end{Claim}

Given $\s$ as in Claim~\ref{claim subcase 2a 2b}, let us conlude the proof of the proposition by showing that $\s\in D$. Note that $\#\s=\nsig -1$ because the dual graph $\G_{\s}$ is a tree with $\nsig$ vertices by Lemma~\ref{lemma structure of k max decopos} and the curves of $\s$ are in one-to-one correspondence with the edges of $\G_{\s}$. The number of non-leaf curves of $\s$ is
\begin{align*}
    \# \left(\s\cap \mclC_{nl}\right)&= \nsig -1-\# \left(\s\cap \mclC_{leaf}\right)\\
                                    &\leq \nsig -1-\frac{2\nsig}{3}=\frac{\nsig}{3}-1\\
                                    &< \frac{2\nsig}{3}\leq \# \left(\s\cap \mclC_{leaf}\right)
\end{align*}
so $\s\in D$. Therefore, we have constructed a $0$-maximal decomposition multicurve in $D$ that completes the leaf curve $c$, which proves the proposition. It is left to prove Claim~\ref{claim subcase 2a 2b}.

\begin{proof}[Proof of Claim~\ref{claim subcase 2a 2b}]
\textbf{Subcase 2.a: $\kappa(\S)> 6$ and $\kappa(S)=1$.} We will construct a $0$-maximal decomposition multicurve that contains $c$ and has $g+\left\lfloor p/3\right\rfloor$ leaf curves in total. First, fix $S$ as the first component of the decomposition and $c$ as the first leaf curve.
\begin{itemize}
    \item If $S$ is a surface of type $\S_{1,0,1}$, then $\S\setminus S$ has genus $g-1$. Separate $g-1$ surfaces of type $\S_{1,0,1}$ and get $g-1$ more leaf curves. Then, separate $\left\lfloor p/3\right\rfloor$ surfaces of type $\S_{0,3,1}$ and get $\left\lfloor p/3\right\rfloor$ additional leaf curves.
    \item If $S$ is a surface of type $\S_{0,3,1}$, then $\S\setminus S$ has $p-3$ marked points. Separate $\partent{p/3}-1$ subsurfaces of type $\S_{0,3,1}$ and get $\partent{p/3}-1$ more leaf curves. Then, separate $g$ surfaces of type $\S_{1,0,1}$ to get $g$ leaf curves. 
\end{itemize}
In both cases, as illustrated in Figure~\ref{fig l leaf curves}, we get $\ell$ leaf curves with $$\ell=g+\partent {p/3}~.$$

\begin{figure}[ht]
    \centering
    \includesvg[width=0.8\textwidth]{claim118}
    \caption{The $g+\partent{p/3}$ leaf curves and the surface $\S'$ that is left.}
    \label{fig l leaf curves}
\end{figure}

Recall that $$\nsig =\frac{3g+p-2-\dsig}{2}$$
so that $$\ell\geq g+\frac{p-2}{3}\geq \frac{3g+p-2-\dsig}{3}=\frac{2\nsig}{3}~.$$
In the middle, there remains a closed subsurface $\S'$ with $p-3\left\lfloor p/3\right\rfloor$ marked points and $g+\left\lfloor p/3\right\rfloor$ boundary components, the $\ell$ leaf curves. The assumption $\kappa(\S)> 6$ ensures us that $\S'$ is neither a pair of pants nor an annulus, so in particular the $\ell$ leaf curves constructed above are pairwise distinct. In summary, $\S'$ has type $\S_{0,p',b'}$ with $$p'+b'=g+p-2\left\lfloor p/3\right\rfloor\geq 4~.$$
Note that $\kappa(\S')$ and $\kappa(\S)$ have the same parity.
Construct a $0$-maximal decomposition $\s'$ of $\S'$ as in Example~\ref{exemple de decompo pour sphere n trous}. The vertices of $\G_{\s'}$ all have complexity~$1$, except for one complexity $2$ vertex when $\kappa(\S')$ is even. Let $\s$ be the multicurve composed of the $\ell$ leaf curves constructed above and of the curves of $\s'$. Then, $\G_{\s}$ is the tree $\G_{\s'}$ to which we have added $\ell$ additional leaves of complexity $1$. By Remark~\ref{rem si decompo est un arbre à petites composantes alors cest maximal}, $\s$ is a $0$-maximal decomposition of $\S$.

\medbreak
\textbf{Subcase 2.b: $\kappa(\S)> 6$ and $\kappa(S)=2$.} Note that in this case, $\dsig=1$ by definition of leaf curves, so $\kappa(\S)$ is necessarily even. We will construct a $0$-maximal decomposition multicurve that contains $c$ and has $g+\left\lfloor (p-1)/3\right\rfloor$ leaf curves in total. First, fix $S$ as the first component of the decomposition and $c$ as the first curve.
\begin{itemize}
    \item If $S$ is a surface of type $\S_{1,1,1}$, then $\S\setminus S$ has genus $g-1$ and $p-1$ marked points. Separate $g-1$ surfaces of type $\S_{1,0,1}$ and get $g-1$ more leaf curves. Then, separate $\left\lfloor (p-1)/3\right\rfloor$ surfaces of type $\S_{0,3,1}$ and get $\left\lfloor (p-1)/3\right\rfloor$ additional leaf curves.
    \item If $S$ is a surface of type $\S_{0,4,1}$, then $\S\setminus S$ has $p-4$ marked points and genus $g$. Separate $\partent{(p-1)/3}-1$ subsurfaces of type $\S_{0,3,1}$ and get $\partent{(p-1)/3}-1$ more leaf curves. Then, separate $g$ surfaces of type $\S_{1,0,1}$ to get $g$ additional leaf curves. 
\end{itemize}
In both cases, we get $\ell=g+\partent {(p-1)/3}$ leaf curves. Again, the inequality $\ell\geq 2\nsig/3$ holds because
\begin{align*}
    \ell\geq g+\frac{p-1-2}{3}= g+\frac{p-2-\dsig}{3}=\frac{2\nsig}{3}~.
\end{align*}
There remains a surface $\S'$ with $$p'=p-1-3\left\lfloor (p-1)/3\right\rfloor$$ marked points and $$b'=g+\left\lfloor (p-1)/3\right\rfloor$$ boundary components (the $\ell$ constructed leaf curves). Here, $\kappa(\S)$ is even and at least $7$ so $\kappa(\S)\geq 8$. This ensures that $\S'$ is neither a pair of pants nor an annulus. In summary, $\S'$ has type $\S_{0,p',b'}$ with $$p'+b'=g+(p-1)-2\left\lfloor (p-1)/3\right\rfloor\geq 4~.$$
Note that $\kappa(\S)$ being even means that $g+p$ is odd, so $\kappa(\S')$ is odd.
Construct a $0$-maximal decomposition $\s'$ of $\S'$ as in Example~\ref{exemple de decompo pour sphere n trous}. The vertices of $\G_{\s'}$ all have complexity~$1$. Let $\s$ be the multicurve composed of the $\ell$ leaf curves constructed above and of the curves of $\s'$. Then, $\G_{\s}$ is the tree $\G_{\s'}$ with $\ell-1$ additional leaves of complexity $1$ and one additional leaf of complexity $2$ (corresponding to $S$). By Remark~\ref{rem si decompo est un arbre à petites composantes alors cest maximal}, $\s$ is a $0$-maximal decomposition of $\S$.\qedhere 
\end{proof}\qedhere\end{proof}

\section{Measurable embeddings and measured groupoids}\label{section meas embeddings and groupoids}
A \textit{Borel isomorphism} is a measurable bijection such that its inverse is also measurable. A \textit{standard Borel space} is a measurable space which is Borel isomorphic to a separable and complete metric space. A \textit{standard measured space} is a standard Borel space equipped with a Borel measure, and it is a \textit{standard probability space} when the space has measure $1$. An action of a countable group $G$ on a standard measured space $(X,\mu)$ by Borel isomorphisms is \textit{measure-preserving} when $\mu(g A)=\mu(A)$ for any Borel subset $A\subseteq X$ and any $g\in G$. 

\subsection{Measured groupoids}\label{section groupoids}
\paragraph{Measured groupoids and cocycles.} Measured groupoids generalise both groups and equivalence relations of measure-preserving group actions. Informally, we think of a groupoid over a measured space $X$ as a measurably varying set of arrows with source and range in $X$. A cocycle is then viewed as a group-valued labelling of the arrows that respects the composition of arrows.
\begin{Def}
    A \textit{Borel groupoid} over a standard Borel space $X$ is a standard Borel space $\mclG$ equipped with the following: 
    \begin{itemize}[[noitemsep, nolistsep, label=-]
        \item two measurable maps $s,r : \mclG \longrightarrow X$ called the \textit{source} and \textit{range} maps;
        \item a measurable map $({\mfk g},{\mfk h})\longmapsto {\mfk g}{\mfk h}$ defined on $\mclG^{(2)}=\{({\mfk g},{\mfk h})\in\mclG ~|~ r({\mfk h})=s({\mfk g})\}$
     and called the \textit{composition} map;
        \item a measurable map 
        $$\begin{array}{rcl}
            X & \longrightarrow & \mclG \\
            x & \longmapsto & {\mfk e}_x
        \end{array}$$
        such that for any $x\in X$ and any ${\mfk g},{\mfk h}\in\mclG$ with $s({\mfk g})=x=r({\mfk h})$,
        \begin{align*}
            s({\mfk e}_x)=r({\mfk e}_x)=x,~{\mfk g}{\mfk e}_x={\mfk g}\text{ and }{\mfk e}_x{\mfk h}={\mfk h}
        \end{align*}
       called the \textit{unit} map;
        \item a measurable map 
        $$\begin{array}{rcl}
            \mclG & \longrightarrow & \mclG \\
            {\mfk g} & \longmapsto & {\mfk g}^{-1}
        \end{array}$$
        such that for any ${\mfk g}\in\mclG$, 
        $$s({\mfk g}^{-1})=r({\mfk g}),~r({\mfk g}^{-1})=s({\mfk g}),~{\mfk g}{\mfk g}^{-1}={\mfk e}_{r({\mfk g})}\text{ and } {\mfk g}^{-1}{\mfk g}={\mfk e}_{s({\mfk g})}$$
        called the \textit{inverse} map.
    \end{itemize}
    The space $X$ is called the \textit{base space} of the groupoid $\mclG$. The groupoid $\mclG$ is called \textit{discrete} when $s^{-1}(\{x\})$ is countable for every $x\in X$.

    A \textit{subgroupoid} of $\mclG$ is a Borel subset $\mclH\subseteq \mclG$ that is stable by composition and inverse, and that contains $\mfk e_x$ for all $x\in X$. The subgroupoid \textit{generated} by two subgroupoids $\mclH$ and $\mclH'$ of $\mclG$ is the smallest subgroupoid of $\mclG$ that contains $\mclH$ and $\mclH'$. It is denoted $\gen{\mclH,\mclH'}$.
\end{Def}

A discrete groupoid is a countable group exactly when the base space $X$ is reduced to a singleton. An important example of discrete Borel groupoid comes from an action of a countable group $G$ on a standard Borel space $X$ by Borel isomorphisms. This groupoid is denoted $G \ltimes X$, the elements of $G \ltimes X$ are couples $(g,x)\in G\times X$ with $s(g,x)=x$ and $r(g,x)=g x$. The composition law is $(g, hx)(h,x)=(gh,x)$ and the units are ${\mfk e}_x=(1_G,x)$ for $g,h\in G$ and $x\in X$.

\begin{Def}
    The \textit{restriction} of a Borel groupoid $\mclG$ to a Borel subset $A\subseteq X$ is a Borel groupoid over the space $A$ defined by $\mclG|_A=\{g\in\mclG~|~s(g),r(g)\in A\}~.$
\end{Def}

Note that the restricted space $(G\ltimes X)|_A$ is a groupoid, but does not necessarily come from a group action. The fact that groupoids are preserved by restriction motivates their definition as a generalisation of group actions on measured spaces. 

\begin{Def}
    A \textit{bisection} of a Borel groupoid $\mclG$ is a measurable subset of $\mclG$ such that $s|_B$ and $r|_B$ are Borel isomorphisms onto their images $s(B)$ and $r(B)$. 
\end{Def}

For instance, any groupoid of the form $(G\ltimes X)|_A$ can be written as the countable union of the bisections $B_{g}=\{g\}\times (g^{-1}A\cap A)$ for $g\in G$. In general, by the Lusin-Novikov theorem (see e.g.\ {\cite[Theorem~18.10]{kechris_classical_1995}}), any discrete Borel groupoid $\mclG$ can be written as a countable union of bisections of $\mclG$.

\begin{Def}
    A \textit{measured groupoid} is a discrete Borel groupoid $\mclG$ over a standard measured space $(X,\mu)$ such that for any bisection $B\subseteq \mclG$, $\mu(s(B))=\mu(r(B))$.
\end{Def}

Note that if a probability measure-preserving action $G \curvearrowright X$ is free, the associated groupoid $G\ltimes X$ corresponds to the usual countable Borel equivalence relation ${\mcl{R}=\{(x,g x)\}_{x\in X, g\in G}}$.

There is a natural notion of groupoid homomorphism and isomorphism, defined for instance in \cite[Definition~3.11]{guirardel_measure_2025}. In particular, a cocycle is a groupoid homomorphism from a measured groupoid to a group. Informally, one can understand a cocycle as labels on the arrows of a groupoid that respect the composition laws of the groupoid and the group.

\begin{Def}
    Let $\mclG$ be a measured groupoid over a standard measured space $X$. Let $G$ be a countable group. A \textit{cocycle} is a measurable map $\rho : \mclG \longrightarrow G$ such that 
    $$\forall ({\mfk g},{\mfk h})\in\mclG^{(2)},~\rho({\mfk g}{\mfk h})=\rho({\mfk g})\rho({\mfk h})~.$$
    The \textit{kernel} of $\rho$ is $\ker \rho =\{{\mfk g}\in\mclG~|~\rho({\mfk g})=1_{G}\}$ and we say that $\rho$ has \textit{trivial kernel} if $\ker\rho=\{{\mfk e}_x\}_{x\in X}$.
\end{Def}

When considering groupoids of the form $(G\ltimes X)|_A$, we naturally define $$\begin{array}{ccc}
    (G\ltimes X)|_A & \longrightarrow & G \\
     (g,x) & \longmapsto &g 
\end{array}$$
which is in fact a cocycle with trivial kernel.

\begin{Def}
    A measured groupoid $\mclG$ over $(X,\mu)$ is of \textit{infinite type} if for any Borel subset $A\subseteq X$ of positive measure and for almost every $x\in A$, the set $\mclG|_A\cap s^{-1}(\{x\})$ is infinite.
\end{Def}

\begin{Rem}\label{rem trivial kernel infinite-type groupoid to infinite group}
    Let $\mclG$ be a measured groupoid and let $G$ be a group. Let $\rho : \mclG\longrightarrow G$ be a cocycle with trivial kernel. If $\mclG$ is of infinite type, then the group $G$ is infinite.
\end{Rem}

\begin{Def}
    A cocycle $\rho :\mclG\longrightarrow G$ is of \textit{action type} when it has trivial kernel and when for any infinite subgroup $H$ of $G$, the groupoid $\rho^{-1}(H)$ is of infinite type.
\end{Def}

Note that the notions of infinite type and action type are naturally preserved by restrictions of the groupoid to positive measure Borel subsets.

An example of action-type cocycles is given by probability measure-preserving actions by Poincaré recurrence, so the assumption of finiteness of the measure is necessary in the following lemma.

\begin{Lemma}[{\cite[Proposition~2.26]{kida_introduction_2009}}, {\cite[Lemma~3.22]{guirardel_measure_2025}}]\label{cocycles from pmp actions are action-type}
    Let $G$ be a countable group admitting a measure-preserving action on a standard measured space $(X,\mu)$ with $\mu(X)$ finite. Then, the groupoid $(G\ltimes X)|_A$ is of infinite type and the cocycle $(G\ltimes X)|_A \longrightarrow G$ is of action type for every Borel subset $A\subseteq X$ of positive measure.
\end{Lemma}

\begin{Def}
    Let $\mclG$ be a measured groupoid over a standard measured space $X$ and let $\rho : \mclG\longrightarrow G$ be a cocycle into a countable group $G$. Assume that $G$ has a measure-preserving action on a standard measured space $Y$. A measurable map $\varphi :X\longrightarrow Y$ is $(\mclG,\rho)$\textit{-equivariant} if there exists a conull Borel subset $X^*$ of $X$ such that $\forall g\in\mclG|_{X^*}$, 
    $$\varphi(r(g))=\rho(g)\varphi(s(g))~.$$
\end{Def}

\paragraph{Normalisation and amenability.} 

\begin{Def}
    Let $\mclG$ be a measured groupoid over a standard measured space $X$, $\mclH$ a subgroupoid of $\mclG$ and $B$ a bisection of $\mclG$. Then, $\mclH$ is $B$\textit{-invariant} if there exists a conull Borel subset $X^*\subseteq X$ such that for any ${\mfk b},{\mfk b}'\in B\cap \mclG|_{X^*}$ and ${\mfk h}\in\mclG|_{X^*}$ such that ${\mfk b}'{\mfk h}{\mfk b}^{-1}$ is well defined, it holds that 
    $${\mfk h}\in\mclH\iff {\mfk b}'{\mfk h}{\mfk b}^{-1}\in\mclH~.$$
\end{Def}

\begin{Def}
    Let $\mclH,\mclA$ be subgroupoids of a measured groupoid $\mclG$ over a space $X$.
    Then, $\mclH$ \textit{normalises} $\mclA$ if there exist bisections $(B_n)_{n\in\N}$ such that $\mclH=\bigcup_n B_n$ and $\mclA$ is $B_n$-invariant for every $n$.
\end{Def}
Note that if $\mclH$ normalises $\mclA$, then $\mclH|_A$ normalises $\mclA|_A$ for any Borel subset $A\subseteq X$.

\begin{Lemma}\label{lemma if groups normalise then groupoids normalise}
    Let $\rho : \mclG\longrightarrow G$ be a cocycle from a measured groupoid $\mclG$ to a countable group $G$. If a subgroup $H$ of $G$ normalises another subgroup $A$ then $\rho^{-1}(H)$ normalises $\rho^{-1}(A)$.\qed
\end{Lemma}

The notion of \textit{amenable} measured groupoid was coined by Kida and is a generalisation of Zimmer amenable measure-preserving actions defined in \cite{zimmer_amenable_2020}. We refer the reader to \cite[Section~4]{kida_introduction_2009} and \cite[Definition~3.36]{guirardel_measure_2025} for the formal definition. Here, we choose to give only the results we will need about how amenability of groupoids interacts with amenability of groups via cocycles. Note that we will also use specific results proven by Kida about amenable groupoids with cocycles towards mapping class groups, which we give in Section~\ref{section classification subgroupoids}.

\begin{Lemma}[{\cite[Proposition~4.33]{kida_introduction_2009}}, {\cite[Corollary~3.42]{guirardel_measure_2025}}]\label{lemma triv kernel cocycle amenable group amenable groupoid}
    Let $\mclG$ be a measured groupoid and let $G$ a countable group. Let $\rho : \mclG\longrightarrow G$ be a cocycle with trivial kernel. 
    
    If $G$ is amenable, then $\mclG$ is amenable.
\end{Lemma}

Kida's definition of amenability of a groupoid is stable by restrictions to positive measure Borel subsets of the base space. It is necessary to reinforce the notion of nonamenability for it to also be stable by restrictions in the following way.

\begin{Def}
    A groupoid $\mclG$ over a space $X$ is \textit{everywhere nonamenable} if for every Borel subset $A\subseteq X$ of positive measure, $\mclG|_A$ is not amenable.
\end{Def}

\begin{Lemma}[{\cite[Lemma~3.20]{kida_measure_2010}, \cite[Lemma~3.43]{guirardel_measure_2025}}]\label{lemme action-type groupe non moyennable donne groupoide non moyennable}
    Let $\mclG$ be a measured groupoid over a standard probability space $X$ and let $\rho : \mclG\longrightarrow G$ be an action-type cocycle into a countable group $ G$. 
    
    If $ G$ contains a nonabelian free group, then $\mclG$ is everywhere nonamenable.
\end{Lemma}

\subsection{Measurable embeddings}\label{subsection meas embeddings}

Recall that an action of a countable group $G$ on a standard measured space $(\Omega,m)$ is measure-preserving when each $ g\in G$ acts by a Borel isomorphism on $\Omega$ and for every Borel subset $A$, $m( g A)=m(A)$. For such an action, a \textit{Borel fundamental domain} is a Borel subset $D$ such that $\Omega = \ds\bigsqcup_{ g\in G} g D$, or equivalently the map $(g,d)\in G\times D\longmapsto gd\in\Omega$ is a Borel isomorphism.

\begin{Def}
    Let $G,H$ be two countable groups. We say that $H$ \textit{measurably embeds} into $G$ if there exists a standard measured space $(\Omega,m)$ with a measure-preserving action of $G\times H$ such that
    \begin{itemize}[noitemsep, nolistsep]
        \item the action $G\curvearrowright \Omega$ is free and admits a Borel fundamental domain of finite positive measure,
        \item the action $H\curvearrowright \Omega$ is free and admits a Borel fundamental domain.
    \end{itemize}
    Such a space $\Omega$ is called a \textit{measurable embedding coupling} of $H$ into $G$. 
    
    If additionally the action $H\curvearrowright \Omega$ admits a Borel fundamental domain of finite measure, then we say that $H$ and $G$ are \textit{measure equivalent} and $\Omega$ is a \textit{measure equivalence coupling} of $H$ and~$G$.
\end{Def}

\begin{Lemma}\label{lemma measurable embedding provides groupoid and cocycles}
    Let $G,H$ be two countable groups. Let $(\Omega,m)$ be a measurable embedding coupling of $H$ into $G$.
    
    Then, there exist
    \begin{itemize}[noitemsep, nolistsep]
        \item a Borel subset $X$ of $\Omega$ of finite positive measure such that $\left(G\times H\right) X$ is conull in $\Omega$,
        \item a measured groupoid $\mclG$ with base space $X$, an action-type cocycle $\w : \mclG\longrightarrow H$ and a trivial-kernel cocycle $\rho : \mclG\longrightarrow G$.
    \end{itemize}
    Up to replacing $m$ by $m/m(X)$, one can consider $X$ to be a standard probability space.
\end{Lemma}

\begin{proof}
By \cite[Lemma~2.27]{kida_introduction_2009}, it is possible to choose Borel fundamental domains $D_G$ for $G$ and $D_H$ for $H$ such that $$m\left(\Omega\setminus\left(G\times H\right) X\right)=0~, $$where $X=D_G\cap D_H$. In particular, this implies that $X$ has positive measure. As $D_G$ has finite measure, $X$ also has finite measure. Note that \cite[Lemma~2.27]{kida_introduction_2009} is stated in the case of a measure equivalence coupling, but the proof is exactly the same for a measurable embedding coupling.

The actions of $G$ and of $H$ on $\Omega$ commute, so $G$ acts on $D_H$ and $H$ acts on $D_G$ via measure-preserving actions we denote by $\cdot$ to differentiate them from the actions on $\Omega$. Expliciting those actions provides us with two cocycles. For $d\in D_{ H}$ and $ g\in G$, define $ g\cdot d$ as being the unique representative in $D_{ H}$ of the $ H$-orbit of $ g d$, i.e.\
$$\{ g\cdot d\}= D_{ H}\cap  H( g d)~.$$
Similarly, for $h\in H$ and $d\in D_G$, $$\{ h\cdot d\}= D_{ G}\cap  G( h d)~.$$
Taking the exact element of $ H$ that brings $ g d$ to $D_{ H}$ defines the following cocycle:
$$\w : \begin{array}{ccl}
     G\ltimes D_{ H} & \longrightarrow &  H \\
     ( g, d)&\longmapsto & \text{the unique } h\in H \text{ such that }  h g d= g\cdot d~.
\end{array}$$
Symmetrically, the action $ H\stackrel{\textstyle \cdot}{\curvearrowright}D_{ G}$ gives the cocycle 
$$\rho : \begin{array}{ccl}
     H\ltimes D_{ G} & \longrightarrow &  G \\
     ( h, d)&\longmapsto & \text{the unique } g\in G \text{ such that }g h d =h\cdot d~.
\end{array}$$
Now, restricting to $X=D_G\cap D_H$, there is a groupoid isomorphism
$$\begin{array}{ccc}
     \big(( G\times H)\ltimes \Omega\big)|_X &\simeq &( G\ltimes D_{ H})|_X\\
     \left(( g, h),x\right)& \longmapsto & ( g,x)
\end{array}$$
because for fixed $ g\in G$ and $x\in X$, the only possible $ h\in H$ such that $( g, h)x\in X$ is $\w( g,x)$. The cocycle $\w$ becomes 
$$\w : \begin{array}{ccl}
     \big(( G\times H)\ltimes \Omega\big)|_X &\longrightarrow & H\\
     \left(( g, h),x\right)& \longmapsto &  h~.
\end{array}$$
But we also have the groupoid isomorphism
$$\begin{array}{ccc}
     \big(( G\times H)\ltimes \Omega\big)|_X &\simeq &( H\ltimes D_{ G})|_X\\
     \left(( g, h),x\right)& \longmapsto & ( h,x)
\end{array}$$
which now gives the cocycle $\w$ as the natural cocycle from the action of $H$ on $D_G$
$$\w : \begin{array}{ccl}
     ( H\ltimes D_{ G})|_X &\longrightarrow & H\\
     ( h,x)& \longmapsto &  h~.
\end{array}$$
Symmetrically, $\rho : ( G\ltimes D_{ H})|_X\longrightarrow G$.
To sum up, define the common groupoid
$$\mclG = (H\ltimes D_{G})|_X \simeq (G\ltimes D_{H})|_X~ .$$
Via this identification, the cocycles defined above arise from group actions
$$\w : (H\ltimes D_{G})|_X \longrightarrow H~ ,$$
$$\rho : (G\ltimes D_{H})|_X \longrightarrow G~.$$
Thus, since $D_G$ has finite measure, the cocycle $\w$ is of action type by Lemma~\ref{cocycles from pmp actions are action-type}.
The difference with measure equivalence lies in the fact that there is no reason for $\rho $ to be of action type because it comes from an action on a space of infinite measure. The cocycle $\rho $ still has trivial kernel because it comes from a group action.
\end{proof}

\begin{Rem}\label{alg embedding is meas embedding}
    For countable groups $G$ and $H$, an injective homomorphism $f : H\longrightarrow G$ is a special case of measurable embedding. Indeed, $G=\Omega$ is a measurable embedding coupling of $H$ into $G$, where $G$ acts by left translations and $f(H)$ by right translations. This coupling corresponds to $X$ being one single point of $G$. The groupoid $\mclG$ given by Lemma~\ref{lemma measurable embedding provides groupoid and cocycles} is then isomorphic to the group $H$, the cocycle $\w$ is the identity on $H$ and the cocycle $\rho $ is the embedding~$f$.
\end{Rem}

\begin{Rem}\label{rem amen is preserved by meas emb} Take a measurable embedding from a countable group $H$ to a countable group $G$, apply Lemma~\ref{lemma measurable embedding provides groupoid and cocycles} and adopt the same notations. If $G$ is amenable, the trivial-kernel cocycle $\rho : \mclG\longrightarrow G$ forces the groupoid $\mclG$ to be amenable, by Lemma~\ref{lemma triv kernel cocycle amenable group amenable groupoid}. The action-type cocycle $\w:\mclG \longrightarrow H$ has the following additional property: if $H$ is nonamenable, then $\mclG$ is everywhere nonamenable, by \cite[Lemma~4.12]{escalier_graph_2026}. In other words, $\w$ is \textit{action-like} in the sense of \cite[Definition~4.9]{escalier_graph_2026}. In conclusion, if $H$ measurably embeds into $G$ and $G$ is amenable, then $H$ is too.
\end{Rem}

\begin{Rem}\label{remark finite indices in me couplings}
Let $\Omega$ be a measurable embedding coupling of $H$ into $G$ with Borel fundamental domains $D_H$ and $D_G$. Let $H^0$ (resp. $G^0$) be a finite-index subgroup of $H$ (resp. $G$) with respective finite systems of representatives $R_H$ and $R_G$. Define the Borel subsets 
$$D_G^0=R_GD_G\text{ and }D_H^0=R_HD_H~.$$
Then, $D_G^0$ (resp. $D_H^0$) is a fundamental domain for the induced action of $G^0$ (resp. $H^0$) on $\Omega$. Moreover, $D_G^0$ has finite positive measure because it is the union of a finite number of translates of $D_G$. Thus, $\Omega$ is also a measurable embedding coupling of $H^0$ into $G^0$. 

Moreover, $D_H^0 = \bigsqcup_{h\in R_H} h D_H$ so $D_H^0$ has finite measure if and only if $D_H$ has finite measure. Therefore, $\Omega$ is a measure equivalence coupling of $H$ and $G$ if and only if it is a measure equivalence coupling of $H^0$ and $G^0$.
\end{Rem}

\begin{Prop}[{\cite[Proposition~2.9]{delabie_quantitative_2022}}]\label{transitivité du plongement ME} Let $G_1,G_2$ and $H$ be countable groups. 

If $G_1$ measurably embeds into $H$ and $H$ into $G_2$, then $G_1$ measurably embeds into $G_2$. If additionally every measurable embedding from $G_1$ to $G_2$ is a measure equivalence, then $H$ is measure equivalent to $G_1$ and $G_2$.
\end{Prop}
\begin{proof} For $i\in\{1,2\}$, fix $G_i\times H\curvearrowright\Omega_i$ a measurable embedding coupling. Denote the corresponding fundamental domains $X_i$ for $H$ and $Y_i$ for $G_i$, i.e.\
\begin{align*}
    \Omega_i&\simeq  H\times X_i\\
          &\simeq G_i\times Y_i~.
\end{align*} 
The Borel subsets $Y_2$ and $X_1$ have finite measure.

The composed coupling is defined by 
$$\Omega =~^{\textstyle\Omega_1\times\Omega_2}\Big/ _{\textstyle  H}$$
where $H$ acts diagonally on $\Omega_1\times\Omega_2$. Using the fact that $\Omega_1/H\simeq X_1$, we identify
$$\Omega\simeq X_1\times\Omega_2~.$$
This exhibits an action of $G_2$ on $\Omega$ by acting on the second coordinate. A Borel fundamental domain for this action is therefore $X_1\times Y_2$, which has finite measure. Symmetrically, the indentification 
$$\Omega\simeq \Omega_1\times X_2 $$
gives an action of $G_1$ on $\Omega$ with Borel fundamental domain $Y_1\times X_2$. The actions of $G_1$ and $G_2$ on $\Omega$ commute and are free, see e.g.\ \cite[Appendix~A.1.3]{bader_integrable_2013} for the explicit formulas. Hence $\Omega$ is indeed a measurable embedding coupling from $G_1$ to $G_2$.

Now, assume in addition that every measurable embedding coupling from $G_1$ into $G_2$ is a measure equivalence. This forces the Borel subset $Y_1\times X_2$ to also have finite measure in $\Omega_1\times X_2$. If either $Y_1$ or $X_2$ had infinite measure, then $Y_1\times X_2$ also would have infinite measure. Thus, $Y_1$ and $X_2$ both have finite measure. Therefore, $H$ is measure equivalent to $G_1$ via the coupling $\Omega_1$ and to $G_2$ via $\Omega_2$.
\end{proof}

\subsection{An obstruction to measurable embeddings of free products}\label{section adams and free into direct products}

Proposition~\ref{prop adams maximal map} recalls an important construction by Adams in \cite{adams_indecomposability_1994}. We will use it to prove the following Lemma~\ref{lemma A,B in F2 nonamenable}, which is a specific version of Lemma~\ref{lemme action-type groupe non moyennable donne groupoide non moyennable} that holds for any cocycle with trivial kernel. The latter is then used to prove Lemma~\ref{lemma free product into direct product}, which provides us with an obstruction to the image of a nonabelian free group via a measurable embedding.

Let $F_2$ be a free group over two generators and denote by $\partial F_2$ its boundary. Denote by $\Prob(\partial F_2)$ the probability measures on the compact metrisable space $\partial F_2$. The structure of measurable space of $\Prob(\partial F_2)$ is the coarsest one that makes the maps $\mu\mapsto\mu(A)$ measurable for all Borel subsets $A$ of $\partial F_2$, see e.g.\ \cite[Theorem~17.24]{kechris_classical_1995}. Denote by $\Prob_{\leq 2}(\partial F_2)$ the probability measures on $\partial F_2$ that have a support of cardinality $1$ or $2$. Note that $\Prob_{\leq 2}(\partial F_2)$ is a Borel subset of $\Prob(\partial F_2)$.

\begin{Prop}[Adams {\cite{adams_indecomposability_1994}}]\label{prop adams maximal map}
    Let $\mclG$ be an infinite-type amenable measured groupoid on a standard probability space $X$. Let $\rho : \mclG\longrightarrow F_2$ be a cocycle with trivial kernel. 
    
    Then, there exists a measurable and $(\mclG,\rho)$-equivariant map $\varphimax : X\longrightarrow\Prob(\partial F_2)$ that verifies the following.
    \begin{enumerate}[$(1)$]
        \item \label{1support}\textbf{Support.} There exists a conull Borel subset $X^*$ of $X$ such that $\varphimax|_{X^*}$ takes values in $\Prob_{\leq 2}(\partial F_2)$.
        \item \label{2maximality}\textbf{Maximality.} If $\varphi: X \longrightarrow \Prob( \partial F_2)$ is a measurable and $(\mclG,\rho)$-equivariant map, then for almost every $x\in X$, $\Supp(\varphi(x))\subseteq \Supp (\varphimax(x))$.
        \item \label{3uniqueness}\textbf{Uniqueness.} If there exists another map $\varphimax'$ that verifies Points~\ref{1support} and \ref{2maximality}, then for almost every $x\in X$, $\varphimax(x)=\varphimax'(x)$.
    \end{enumerate}
\end{Prop}
\begin{proof}[Sketch of proof]
   The groupoid $\mclG$ being amenable, by \cite[Proposition~4.14]{kida_introduction_2009}, there exists a measurable and $(\mclG,\rho)$-equivariant map ${X\longrightarrow \Prob(\partial F_2)}$. In \cite[Proposition~5.1]{adams_indecomposability_1994}, Adams states that any measurable $(\mclG,\rho)$-equivariant map ${X\longrightarrow \Prob(\partial F_2)}$ takes values in $\Prob_{\leq 2}(\partial F_2)$ almost everywhere. This map is therefore almost everywhere valued in $\Prob_{\leq 2}(\partial F_2)$.
    Then, this existence allows us to apply \cite[Lemma~3.2]{adams_indecomposability_1994} which gives the existence of a map ${\varphimax : X\longrightarrow \Prob_{\leq 2}(\partial F_2)}$ that verifies Points~\ref{1support} and \ref{2maximality}. By \cite[Lemma~3.3]{adams_indecomposability_1994}, Point~\ref{3uniqueness} holds as well.
\end{proof}

\begin{Rem}
    Strictly speaking, Adams states the previous results only for action-type cocycles because he works with probability measure-preserving actions. Fortunately, the proof is exactly the same for cocycles with trivial kernel. For instance, we refer to the proof of \cite[Lemma~12.2]{guirardel_measure_2025} where Adams arguments are rewritten for trivial-kernel cocycles.
\end{Rem}

\begin{Lemma}[Kida {\cite[Lemma~3.20]{kida_measure_2010}}]\label{lemma A,B in F2 nonamenable}
    Let $\mclG$ be a measured groupoid on a standard probability space $X$. Let ${F_2=\gen{a,b}}$ be a nonabelian free group and let $\rho : \mclG\longrightarrow F_2$ be a cocycle with trivial kernel. Let $\mclA,\mclB$ be infinite-type subgroupoids of $\mclG$ such that $\rho(\mclA)\subseteq \gen{a}$ and $\rho(\mclB)\subseteq \gen{b}$.
    
    Then, for every Borel subset $U$ of $X$ of positive measure, the subgroupoid $\gen{\mclA|_U,\mclB|_U}$ is everywhere nonamenable.
\end{Lemma}

\begin{proof}
Assume by contradiction that there exists a positive measure Borel subset $Y$ of $U$ such that the subgroupoid $$\mclH=\gen{\mclA|_U,\mclB|_U}|_Y$$ is amenable. Then, by Proposition~\ref{prop adams maximal map}, there exists a measurable and $(\mclH,\rho)$-equivariant map $$\varphi_{\mclH} : Y\longrightarrow \Prob(\partial F_2)~.$$

Denote by $\zeta_a,\eta_a$ (resp. $\zeta_b,\eta_b$) the two fixed ends of $a$ (resp. $b$) in $\partial F_2$. Denote by $\delta_{\zeta}$ the Dirac measure at an end $\zeta\in\partial F_2$.
The constant map $$\varphi_a : Y\longrightarrow \left\{\frac{1}{2}\delta_{\zeta_a}+\frac{1}{2}\delta_{\eta_a}\right\}$$
is measurable and $(\mclA|_Y,\rho)$-equivariant. Similarly, the constant map 
$$\varphi_b : Y\longrightarrow \left\{\frac{1}{2}\delta_{\zeta_b}+\frac{1}{2}\delta_{\eta_b}\right\}$$ is measurable and $(\mclB|_Y,\rho)$-equivariant map.

The groupoid $\mclA|_Y$ is amenable because it is sent into the amenable group $\gen a$ by a cocycle with trivial kernel, by Lemma~\ref{lemma triv kernel cocycle amenable group amenable groupoid}. Proposition~\ref{prop adams maximal map} applied to the groupoid $\mclA|_Y$ yields the existence of an almost everywhere unique measurable and $(\mclA|_Y,\rho)$-equivariant map $Y\longrightarrow \Prob_{\leq 2}(\partial F_2)$ that is maximal amongst all measurable $(\mclA|_Y,\rho)$-equivariant maps $Y\longrightarrow\Prob(\partial F_2)$. Now, notice that the constant map $\varphi_a$ is necessarily equal almost everywhere to this unique maximal map because it has support of cardinality exactly $2$.
Similarly, by Proposition~\ref{prop adams maximal map} applied to $\mclB|_Y$, the constant map 
$$\varphi_b : Y\longrightarrow \left\{\frac{1}{2}\delta_{\zeta_b}+\frac{1}{2}\delta_{\eta_b}\right\}$$ is the (almost everywhere) unique maximal measurable $(\mclB|_Y,\rho)$-equivariant map amongst all measurable $(\mclB|_Y,\rho)$-equivariant maps $Y\longrightarrow \Prob(\partial F_2)$.

Note that $\mclA|_Y,\mclB|_Y\subseteq \mclH$ so the map $\varphi_{\mclH}$ is in particular both $(\mclA|_Y,\rho)$-equivariant and $(\mclB|_Y,\rho)$-equivariant. By maximality of $\varphi_a$ amongst all measurable and $(\mclA|_Y,\rho)$-equivariant maps from $Y$ to $\Prob(\partial F_2)$, for almost every $y\in Y$, $$\Supp(\varphi_{\mclH}(y))\subseteq\Supp (\varphi_a(y))=\{\zeta_a,\eta_a\}~.$$
Similarly, for almost every $y\in Y$, $$\Supp(\varphi_{\mclH}(y))\subseteq \{\zeta_b,\eta_b\}~.$$
But $a$ and $b$ generate a free group so $\{\zeta_a,\eta_a\}\cap \{\zeta_b,\eta_b\}$ is empty. Therefore, $\mclH$ cannot be amenable. In conclusion, $\gen{\mclA|_U,\mclB|_U}$ is everywhere nonamenable.
\end{proof}

\begin{Lemma}\label{lemma free product into direct product}
    Let $\mclG$ be a measured groupoid over a standard probability space $X$. Let $G$ and $K$ be infinite countable groups and let $A$ be an infinite cyclic group. Assume that there are two cocycles $$\w:\mclG \longrightarrow K\ast A$$ and $$\rho:\mclG \longrightarrow G$$ that both have trivial kernel.
    Let $\mclK, \mclA$ be infinite-type subgroupoids of $\mclG$ such that 
    \begin{align*}
        \w(\mclK)\subseteq K~\text{and } \w(\mclA) \subseteq A~.
    \end{align*}
    Then, $\rho(\mclA)$ cannot be contained in the centre of $G$.
\end{Lemma}

\begin{proof}
    Assume by contradiction that $\rho(\mclA)\subseteq Z(G)$, where $Z(G)$ is the centre of $G$. For $k\in K$ and $g\in G$, 
    $$B_{g,k}=\{\mfk k\in\mclK~|~ \rho(\mfk k)=g,~\w(\mfk k)=k\}$$
    is a bisection of $\mclG$ because $\w$ and $\rho$ have trivial kernel. Note that $$\mclK=\bigcup_{\substack{g\in G\\ k\in K}}B_{g,k}~.$$
From now on, fix nontrivial $k\in K$ and $g\in G$ such that the bisection $B_{g,k}$, which we will now simply denote by $B$, has positive measure. Name $U$ the source of $B$ and $V$ its range, so that $U,V$ are Borel subsets of $X$ of positive measure. Now, define the following subgroupoid of $\mclG|_V$
    $$\mclH=\gen{\mclA|_V,B\mclA|_U B^{-1}}~.$$
    We will show that $\mclH$ is both amenable and nonamenable which gives a contradiction. 

    On one hand, we show that $\mclH$ is amenable. Firstly, $\rho(\mclA|_V)\subseteq Z(G)$ by assumption. Secondly, fix $\mfk a\in\mclA|_U$ and $\mfk b,\mfk b'\in B$ such that $\mfk{bab'}^{-1}$ is well defined. Then,
    \begin{align*}
        \rho(\mfk{bab'}^{-1})&=\rho(\mfk b)\rho(\mfk a)\rho(\mfk b')^{-1}\\
                            &=g\rho(\mfk a)g^{-1} = \rho(\mfk a)\in Z(G)
    \end{align*}
    so $\rho(B\mclA|_U B^{-1})\subseteq Z(G)$. Thus $\rho(\mclH)\subseteq Z(G)$. The cocycle $\rho$ having trivial kernel, this forces $\mclH$ to be amenable by Lemma~\ref{lemma triv kernel cocycle amenable group amenable groupoid}.

    On the other hand, let us show that $\mclH$ is not amenable. Let $A=\gen{a}$. Note that $\w(\mclA|_V)\subseteq \gen{a}$ and
    $$\w(B\mclA|_U B^{-1})\subseteq k\w(\mclA|_U)k^{-1} \subseteq \gen{kak^{-1}}~.$$
    Using that $a$ and $kak^{-1}$ generate a free group, apply Lemma~\ref{lemma A,B in F2 nonamenable} to the trivial-kernel cocycle $\w$ to find that $\mclH$ is indeed nonamenable.
\end{proof}

\section{Groupoids with cocycles towards mapping class groups, after Kida}\label{section groupoids and mapping class groups}

\subsection{Canonical reduction sets and multicurves}\label{section canonical multicurves}

Fix a connected surface $\S$ and recall (see Theorem~\ref{modz th}) that Ivanov proved that any $f\in\Modz$ is \textit{pure}, i.e.\ there exists a (possibly empty) multicurve $\s$ on $\S$ such that every connected component $Q$ of $\S\setminus \s$ is $f$-invariant and either $f$ fixes every curve of $Q$ or $f$ fixes no curve of $Q$. Then $\s$ is called a \textit{reduction multicurve} for $f$. More generally, in \cite[Section~7.2]{ivanov_subgroups_1992}, Ivanov constructs a \textit{canonical reduction multicurve} for any subgroup of $\Modz$, which is called canonical because it is the unique reduction multicurve that is invariant by the normaliser of the subgroup. In this section, following Kida \cite[Theorem~3.6]{kida_measure_2010} and Hensel--Horbez \cite[Section~3.2]{hensel_measure_2025}, we explain that every subgroupoid of a groupoid $\mclG$ with a cocycle towards $\Modz$ also has this canonical multicurve, up to a countable partition of the base space $X$.

\begin{Def}
    Let $\mclG$ be a measured groupoid over a standard probability space $X$. Let $\S$ be a connected surface and let $\rho : \mclG\longrightarrow \Modz$ be a cocycle.
    
    We say that $(\mclG,\rho)$ \textit{stabilises} a multicurve $\s$ if ${\rho(\mclG|_{X^*})\subseteq \Stab\s }$ for some conull Borel subset $X^*$ of $X$. 
    We say that $(\mclG,\rho)$ \textit{stabilises} an open subsurface $Q$ if $(\mclG,\rho)$ stabilises the multicurve $\partial Q$. In other words, $Q$ is a component of $\S\setminus \s$ for some multicurve $\s$ that $\mclH$ stabilises.
\end{Def}

\begin{Def}
     Let $\mclG$ be a measured groupoid over a standard probability space $X$. Let $\S$ be a connected surface and let $\rho : \mclG\longrightarrow \Modz$ be a cocycle. A set $\mclC$ of curves on $\S$ is the \textit{canonical reduction set} of $(\mclG,\rho)$ if 
     \begin{itemize}[nolistsep, label=-]
        \item $(\mclG,\rho)$ stabilises all the curves in $\mclC$;
        \item for every Borel subset $U$ of $X$ of positive measure, any curve that is stabilised by $(\mclG|_U,\rho)$ belongs to $\mclC$.
     \end{itemize}
     The \textit{canonical multicurve} of $(\mclG,\rho)$ is the boundary of the subsurface filled by $\mclC$.
\end{Def}

If the canonical reduction set exists, then it is exactly the set of all the curves stabilised by $(\mclG,\rho)$, which makes it unique. The canonical multicurve is then also unique. Note also that if $\s$ is the canonical multicurve of $(\mclG,\rho)$ then it is also the canonical multicurve of $(\mclG|_U,\rho)$ for any Borel subset $U$ of positive measure. The next lemma states the existence of the canonical reduction set and multicurve, up to a countable Borel partition of $X$.

\begin{Lemma}[{\cite[Lemma~3.12]{hensel_measure_2025},~\cite[Theorem~3.6]{kida_measure_2010}}]\label{canonical reduction set exists}
    Let $\mclG$ be a measured groupoid over a standard probability space $X$. Let $\S$ be a connected surface and let $\rho : \mclG\longrightarrow \Modz$ be a cocycle.
    
    Then, there exists a countable conull Borel partition $X^*=\bigsqcup_n X_n$ of $X$ such that for every $n$, the canonical reduction set and multicurve exist for $(\mclG|_{X_n},\rho)$.
\end{Lemma}

\begin{Rem}\label{rem we can always choose U for the canonical multicurve to exist}
In view of the previous lemma, given a finite number of subgroupoids $\mclH_1,\dots,
$ $\mclH_k$ of $\mclG$, there always exists a countable conull Borel partition $X^*=\bigsqcup_n X_n$ such that the canonical multicurves of $(\mclH_i|_ {X_n},\rho)$ exist for all $1\leq i\leq k$ and all $n\in \N$. Indeed, it suffices to refine the $k$ partitions given by Lemma~\ref{canonical reduction set exists} into a common partition. In this case, we will say that such a partition is \textit{adapted} to the subgroupoids $(\mclH_i,\rho)$ for all $i$.
\end{Rem}

The canonical multicurve is called canonical because of the following key result.

\begin{Lemma}[{\cite[Lemma~3.13]{hensel_measure_2025},~\cite[Theorem~3.6]{kida_measure_2010}}]\label{lemma multicurve is canonical}
    Let $\mclG$ be a measured groupoid over a standard probability space $X$. Let $\S$ be a connected surface and let $\rho : \mclG\longrightarrow \Modz$ be a cocycle. Let $\mclA$ and $\mclH$ be subgroupoids such that $\mclH$ normalises $\mclA$. Assume that the canonical multicurve $\s$ of $(\mclA,\rho)$ exists. 
    
    Then, $(\mclH,\rho)$ stabilises $\s$.
If in addition the canonical multicurve $\s'$ of $(\mclH,\rho)$ exists, then $\s$ and $\s'$ are disjoint.
\end{Lemma}

\subsection{Classification of subgroupoids, following Kida}\label{section classification subgroupoids}
Following Kida \cite[Chapter~5]{kida_introduction_2009}, we classify subgroupoids of a groupoid with a cocycle towards $\Modz$. On a surface of positive complexity, a subgroupoid can be of exactly one of the three following types: reducible, irreducible amenable or irreducible nonamenable. 

In the algebraic setting, reducible subgroups are exactly subgroups of curve stabilisers. An example of an irreducible amenable subgroup of $\Modz$ is a cyclic group generated by one pseudo-Anosov mapping class. The group $\Modz$ itself is irreducible nonamenable. A nonabelian free group generated by two pseudo-Anosov mapping classes is irreducible nonamenable as well.

\begin{Def}
    Let $\mclG$ be a measured groupoid over a standard probability space $X$. Let $\S$ be a connected surface and let $\rho : \mclG\longrightarrow \Modz$ be a cocycle with trivial kernel. Let $\mclH$ be a subgroupoid of $\mclG$ such that $(\mclH,\rho)$ stabilises an open subsurface $Q$ of $\S$. We say that
\begin{itemize}
        \item $(\mclH,\rho)$ is \textit{trivial} on $Q$ when $\rho(\mclH|_{X^*})\subseteq\Modfix{Q}$ for some conull Borel subset $X^*\subseteq X$.
    \item $(\mclH,\rho)$ is \textit{reducible} on $Q$ when $(\mclH,\rho)$ stabilises a curve $c\in\cxcourbes{Q}$ \item $(\mclH,\rho)$ is \textit{irreducible} on $Q$ if it is nontrivial and not reducible on $Q$. 
\end{itemize}
When $(\mclH,\rho)$ is irreducible on $Q$, we say that
\begin{itemize} 
    \item $(\mclH,\rho)$ is \textit{irreducible amenable} on $Q$ when there exists a measurable and $(\mclH,\rho)$-equivariant map $X\longrightarrow \partial_{\leq 2}\cxcourbes Q$,
    \item $(\mclH,\rho)$ is \textit{irreducible nonamenable} on $Q$ when there does not exist such a map.
\end{itemize}
\end{Def}

Note that if $(\mclH,\rho)$ stabilises a subsurface $Q$ with $\kappa(Q)=0$, then $(\mclH,\rho)$ is automatically trivial on $Q$ because $$\Modact{Q}=\gen{T_c}_{c\in\partial Q} \subseteq \Modfix{Q}~.$$
When $\kappa(Q)>0$, trivial on $Q$ implies reducible on $Q$.
\medbreak
First, we state that if a subgroupoid is reducible and nontrivial on a subsurface $Q$ then its canonical multicurve contains a curve of $\cxcourbes{Q}$.

\begin{Prop}\label{reducible nontrivial implies existence of canonical curve}
    Let $\mclG$ be a measured groupoid over a standard probability space $X$. Let $\S$ be a connected surface and let $\rho : \mclG\longrightarrow \Modz$ be a cocycle with trivial kernel. Let $\mclH$ be a subgroupoid of $\mclG$ and assume that the canonical multicurve $\s$ of $(\mclG,\rho)$ exists. 
    
    If $(\mclH,\rho)$ stabilises an open connected subsurface $Q$ and is reducible and nontrivial on $Q$, then $\s$ contains a curve of $\cxcourbes{Q}$.
\end{Prop}
\begin{proof}
    Consider the subsurface $\S_{\mclC}$ filled by the canonical reduction set of $(\mclH,\rho)$. We have that $\S_{\mclC}\cap Q\neq \varnothing$ because $(\mclH,\rho)$ fixes some curve of $\cxcourbes{Q}$. We also have that $Q\nsubseteq \S_{\mclC}$ because $(\mclH,\rho)$ is nontrivial on $Q$.

    By contradiction, assume that $\cxcourbes{Q}$ does not contain any curve of $\s$. Then, the fact that $(\mclH,\rho)$ stabilises $Q$ means that $Q$ is contained in a connected component of $\S \setminus\partial \S_{\mclC}$, because $\s=\partial\S_{\mclC}$. This contradicts either $\S_{\mclC}\cap Q\neq \varnothing$ or $Q\nsubseteq \S_{\mclC}$.
    
    Therefore, there exists a curve $c\in\cxcourbes{Q}\cap\partial \S_{\mclC}$.
\end{proof}

\begin{Rem}
    
Note that in \cite[Section~7]{kida_introduction_2009} and \cite[Definition~5.4]{kida_introduction_2009}, Kida defines trivial, irreducible amenable and irreducible nonamenable for a subgroupoid $(\mclH,\rho)$ only on components $Q$ of the surface cut along the canonical multicurve $\s$ of $(\mclH,\rho)$, assuming it exists. Let us explain that by considering the more general case of a subsurface $Q$ stabilised by $(\mclH,\rho)$, $Q$ is either a component of $\S\setminus\s$ like for Kida, or $(\mclH,\rho)$ is reducible or trivial on $Q$. Assume that $(\mclH,\rho)$ is nontrivial on an open subsurface $Q$ that it stabilises. \begin{itemize}[noitemsep, nolistsep]
    \item If $\partial Q\nsubseteq \s$, then $(\mclH,\rho)$ is reducible on $Q$ because the subsurface filled by the canonical reduction set of $(\mclH,\rho)$ has nontrivial intersection with $Q$.
    \item If $\partial Q\subseteq \s$, then $(\mclH,\rho)$ is reducible on $Q$ if and only if there is another curve of $\s$ in $Q$. 
\end{itemize}
So if a subgroupoid $(\mclH,\rho)$ is irreducible on $Q$, then $Q$ is necessarily a component of $\S\setminus \s$. Conversely, if $Q$ is a component of $\S\setminus\s$, then $(\mclH,\rho)$ can only be trivial or irreducible on $Q$. Moreover, our definitions of irreducible amenable and irreducible nonamenable coincide with Kida's definition, as stated in \cite[Corollary~5.14]{kida_introduction_2009}. 
\end{Rem}

The following lemma is Kida's groupoid-theoretic version of the fact that the action of the mapping class group on the boundary of the curve graph is amenable, which is itself a theorem independently proven by Kida \cite[Lemma~4.32]{kida_introduction_2009} and Hamenstädt \cite{hamenstadt_geometry_2009}. 

Recall that a subgroupoid $(\mclH,\rho)$ stabilises a multicurve $\s$ on $\S$ when
${\rho (\mclH)\subseteq \Stab{\s}}$.
Note that the subgroup $\Stab{\s}$ acts on $\cxcourbes{Q}$ for every connected component $Q$ of $\S\setminus\s$.
The equivariance of the map in the next lemma is in the sense of this action of $\Stab \s$.

\begin{Lemma}[Kida {\cite[Proposition~5.18]{kida_introduction_2009}}]\label{map to product of boundaries of curve graphs of components}
    Let $\mclG$ be a measured groupoid over a standard probability space $X$ and let $\rho : \mclG\longrightarrow\Modz$ be a cocycle with trivial kernel. Let $\mclH$ be a subgroupoid of $\mclG$ such that $(\mclH,\rho)$ stabilises a multicurve $\s$ on $\S$. Let $Q_1,\dots,Q_k$ be the connected components of $\S\setminus\s$ on which $(\mclH,\rho)$ is nontrivial. Assume that there exists a measurable and $(\mclH,\rho)$-equivariant map 
    $$X\longrightarrow\prod_{1\leq i\leq k}\partial_{\leq 2}\cxcourbes{Q_i}~.$$
    Then $\mclH$ is amenable.
\end{Lemma}

The following corollary justifies the terminology. If a subgroupoid is irreducible amenable on every connected component of $\S\setminus \s$, then it is forced to be an amenable subgroupoid. Note that the contraposition of Corollary~\ref{if IA or T everywhere, then amen. If nonamen then IN somewhere} is the following: if $\mclH$ is a nonamenable subgroupoid such that $(\mclH,\rho)$ has a canonical multicurve $\s$, then it has to be irreducible nonamenable on some connected component of $\S\setminus\s$.

\begin{Cr}\label{if IA or T everywhere, then amen. If nonamen then IN somewhere}
    If $(\mclH,\rho)$ stabilises a multicurve $\s$ and is trivial or irreducible amenable on every connected component of $\S\setminus\s$, then $\mclH$ is amenable.
\end{Cr}

The next two propositions are a key result from Kida's work. It is a general groupoid-theoretic version of the fact that if a subgroup $H$ of a mapping class group normalises an infinite amenable subgroup, then $H$ is either itself amenable or it stabilises a curve of the surface.

\begin{Prop}[Kida {\cite[Section~5.3]{kida_introduction_2009}}]\label{amen is red or IA}
    Let $\mclG$ be a measured groupoid over a standard probability space $X$. Let $\S$ be a connected surface and let $\rho : \mclG\longrightarrow \Modz$ be a cocycle with trivial kernel. Let $\mclA$ be an infinite-type amenable subgroupoid of $\mclG$. Assume that the canonical reduction set of $(\mclA,\rho)$ exists. 
    
    Then, $(\mclA,\rho)$ is reducible or irreducible amenable on any subsurface that it stabilises. In other words, $(\mclA,\rho)$ is never irreducible nonamenable.
\end{Prop}
\begin{proof}[Sketch of proof]
    When $Q$ is stabilised by $(\mclA,\rho)$, the cocycle $\rho$ sends $\mclA$ to $\Stab {\partial Q}$. This subgroup acts on the compact metrisable space $\mcl{PMF}$ of projective measured foliations on the surface $Q$, as explained in \cite[Section~5.2]{kida_introduction_2009}. By \cite[Proposition~4.14]{kida_introduction_2009}, there exists a measurable $(\mclA,\rho)$-equivariant map $$\varphi : X\longrightarrow M(\mcl{PMF})$$ where $M(\mcl{PMF})$ are the probability measures on $\mcl{PMF}$. Then, by \cite[Theorem~5.3]{kida_introduction_2009}, there exists a partition $X= X_1\sqcup X_2$ of $X$ such that $\varphi|_{X_1}$ takes values in $M(\mcl{MIN})$, where $\mcl{MIN}$ are minimal measured foliations, and $\varphi|_{X_2}$ takes values in $M(\mcl{PMF}\setminus \mcl{MIN})$.

    On $X_1$, Kida uses the map $\mcl{MIN}\longrightarrow\partial \cxcourbes{Q}$ constructed by Klarreich in \cite{klarreich_boundary_2022} to get an equivariant map towards $M(\partial \cxcourbes{Q})$. Then, \cite[Proposition~4.29]{kida_introduction_2009} states that this map has support of cardinality at most $2$, which makes $(\mclA|_{X_1},\rho)$ reducible or irreducible amenable on $Q$. On $X_2$, Kida uses the fact that any foliation in $\mcl{PMF}\setminus \mcl{MIN}$ contains a canonical multicurve, which is then stabilised by $(\mclA|_{X_2},\rho)$, so the latter is reducible on $Q$. But the canonical reduction set of $(\mclA,\rho)$ exists. Thus, if $X_2$ has positive measure, then the whole $(\mclA,\rho)$ is reducible on $Q$. If $X_2$ has null measure, then $(\mclA,\rho)$ is reducible or irreducible amenable on $Q$.
\end{proof}

\begin{Prop}[Kida {\cite[Theorem~4.62]{kida_introduction_2009}}]\label{if H normalises A then its reducible or IA}
    Let $\mclG$ be a measured groupoid over a standard probability space $X$. Let ${\rho : \mclG\longrightarrow \Modz}$ be a cocycle with trivial kernel. Let $\mclA$ and $\mclH$ be infinite-type subgroupoids of $\mclG$ such that $\mclH$ normalises $\mclA$. Assume that $(\mclA,\rho)$ and $(\mclH,\rho)$ both stabilise an open subsurface $Q$ and $(\mclA,\rho)$ has a canonical multicurve.
    
    Assume that $(\mclA,\rho)$ is nontrivial and reducible or irreducible amenable on $Q$. Then, $(\mclH,\rho)$ is reducible or irreducible amenable on $Q$.

    In particular, if $\mclA$ is amenable and nontrivial on $Q$, then $(\mclH,\rho)$ is reducible or irreducible amenable on $Q$.
\end{Prop}
\begin{proof}The subsurface $Q$ has positive complexity because $(\mclA,\rho)$ is nontrivial on $Q$.

If $(\mclA,\rho)$ is irreducible amenable on $Q$, then there exists a measurable and $(\mclA,\rho)$-equivariant map ${X\longrightarrow \partial_{\leq 2}\cxcourbes Q}$ that is maximal for inclusion, as shown in \cite[Lemma~5.8]{kida_introduction_2009}. This maximality is enough for it to also be equivariant by a subgroupoid that normalises $\mclA$, namely $(\mclH,\rho)$, as shown in \cite[Lemma~6.7]{kida_introduction_2009}. This means that $(\mclH,\rho)$ is either reducible or irreducible amenable on $Q$.

If $(\mclA,\rho)$ is reducible, let $\s$ be the canonical multicurve of $(\mclA,\rho)$. By Proposition~\ref{reducible nontrivial implies existence of canonical curve}, $\s$ contains a curve of $\cxcourbes{Q}$, which is stabilised by $(\mclH,\rho)$ by Lemma~\ref{lemma multicurve is canonical}. This makes $(\mclH,\rho)$ reducible.

In particular, if $\mclA$ is amenable then $(\mclA,\rho)$ is reducible or irreducible amenable on $Q$ by Proposition~\ref{amen is red or IA}, so the latter holds.
\end{proof}

\section{Rigidity of self measurable embeddings of mapping class groups}\label{section proof theorem}

\begin{Setting}\label{setting for thm}
We fix two connected orientable finite-type surfaces $\Sigma$ and $\S'$ without boundary and of topological complexity at least $1$. Fix $H=\Modz$ and $G=\Modzprime$ the finite-index subgroups of $\ModS$ and $\ModSprime$ respectively as in Definition~\ref{def mapping class group, subgroups of subsurfaces, modz def}. Fix a measurable embedding coupling $\Omega$ of $H$ into $G$. Apply Lemma~\ref{lemma measurable embedding provides groupoid and cocycles} and rescale the measure on $\Omega$ to get a Borel subset $X$ of $\Omega$ of measure $1$ such that $\left(G\times H\right) X$ is conull in $\Omega$ and there is a measured groupoid $\mclG$ over the standard probability space $X$ with 
\begin{itemize}[noitemsep, nolistsep, label=-]
    \item an action-type cocycle $\w :\mclG\longrightarrow H$, 
    \item a trivial-kernel cocycle $\rho : \mclG\longrightarrow G$.
\end{itemize}
\end{Setting}
The first group $H=\Modz$ is the \textit{source} of the measurable embedding and the second group $G=\Modzprime$ is the \textit{target} of the measurable embedding. 
Recall that a subgroup $H'$ of $H$ is \textit{transported into} a subgroup $G'$ of $G$ by the measurable embedding on some positive measure Borel subset $U$ of $X$ when 
$\rho\left(\w^{-1}(H')|_U\right)\subseteq G'$. 

Note that we prove the main theorem about a self measurable embedding, so $\S=\S'$. As explained in Remark~\ref{rem on equality or not of S and S'}, we do not assume equality of the source and the target from the beggining because it allows us to prove monotonicity of topological complexity (Corollary~\ref{cr monotonicity of complexity}) in Section~\ref{section maximal decompo meas}. Directly after proving this monotonicity, we assume $\S=\S'$ for the rest of the proof.

\paragraph{Structure of the proof.} The goal of Section~\ref{section proof theorem} is to prove Theorem~\ref{thm existence of tautening map}, which is a stronger version of Theorem~\ref{main thm} (see Section~\ref{conclusion subsection}). The key step of the proof is to show that Dehn twist cyclic subgroups are transported into Dehn twist cyclic subgroups by the self measurable embedding up to a partition of $X$: this is the key Proposition~\ref{prop recap courbes dans courbes}. To obtain enough rigidity at the target, we retain more information on the source group $H$ than just Dehn twists: we retain $k$-maximal decompositions (see Definition~\ref{def kmax decompo}). Section~\ref{subsection prelim results for the proof of thm} contains two preliminary lemmas about the behaviour of subgroupoids sent by $\rho$ into Dehn twists or multitwists. Then, in Section~\ref{section maximal decompo meas}, we show that $k$-maximal decompositions are transported into $k$-maximal decompositions by the self measurable embedding. Then, Section~\ref{section curves meas} contains the proof of the key Proposition~\ref{prop recap courbes dans courbes}. Finally, Section~\ref{final subsection proof thm} passes from the groupoid-theoretic Proposition~\ref{prop recap courbes dans courbes} to the coupling-theoretic Theorem~\ref{thm existence of tautening map}.

\subsection{Subgroupoids with cocycles to Dehn twists}\label{subsection prelim results for the proof of thm}

The next lemma is a groupoid-theoretic characterisation of adjacency of two vertices in the curve graph.

\begin{Lemma}\label{adjacency curves in groupoidal terms}
Under Setting~\ref{setting for thm}, let $\mclA,\mclB$ be infinite-type subgroupoids of $\mclG$ and $\a,\b\in\cxcourbes{\S'}$. Assume that $\rho(\mclA)\subseteq \gen{T_{\a}}$ and $\rho(\mclB)\subseteq \gen{T_{\b}}$.
\begin{itemize}
    \item If $\a,\b$ are disjoint, then $\gen{\mclA,\mclB}$ is amenable.
    \item If $\a,\b$ are not disjoint, then $\gen{\mclA|_U,\mclB|_U}$ is everywhere nonamenable, for every Borel subset $U\subseteq X$ of positive measure.
\end{itemize}
\end{Lemma}
\begin{proof}
    If $\a,\b$ are disjoint curves, then $\gen{\mclA,\mclB}$ is sent by $\rho$ into $\gen{T_{\a},T_{\b}}$ which is amenable. The cocycle having trivial kernel, this forces $\gen{\mclA,\mclB}$ to be amenable by Lemma~\ref{lemma triv kernel cocycle amenable group amenable groupoid}.

    If $\a$ and $\b$ are not disjoint, then $\gen{T_{\a}}\cap \Modzprime$ and $\gen{T_{\b}}\cap \Modzprime$ generate a nonabelian free group by Corollary~\ref{prop in modz all twists are free groups}. Then, apply Proposition~\ref{lemma A,B in F2 nonamenable} to 
    $$\rho(\mclA)\subseteq \gen{T_{\a}}\cap \Modzprime$$
    and 
    $$\rho(\mclB)\subseteq \gen{T_{\b}}\cap \Modzprime$$
    to obtain that for every positive measure subset $U$ of $X$, $\gen{\mclA|_U,\mclB|_U}$ is everywhere nonamenable.
\end{proof}

We will need some technical information on how a curve intersecting a subsurface behaves through the measurable embedding. The following lemma is an application of Lemma~\ref{lemma free product into direct product} to the subgroupoids that we will study in the following sections. The assumptions on the image of the cocycle $\rho$ come from the context of Lemma~\ref{lemma max decompo are sent to max decompo} in the next section.

\begin{Lemma}\label{lemma b intersects S}
Under Setting~\ref{setting for thm}, let $S$ be a connected open subsurface of $\S$ with $\kappa(S)\geq 1$ and let $b$ be a curve of $\S$ that intersects $S$ nontrivially. Let $U$ be a positive measure Borel subset of $X$. Assume that there exists an open subsurface $Q$ of $\S'$ and a multicurve $\s$ of $\S'\setminus Q$ such that 
$$\rho(\w^{-1}(\Modact{S})|_U)\subseteq\Modact{Q}\times \gen{T_c}_{c\in\s\setminus \partial Q}~.$$
Assume also that there exists a multicurve $\b$ such that
$$\rho(\w^{-1}(\gen{T_b})|_U)\subseteq  \gen{T_c}_{c\in\b}~.$$

Then, $\b$ intersects $Q$ nontrivially. In particular, if the canonical multicurve of $(\w^{-1}(\gen{T_b})|_U,\rho)$ is $\b$, then $(\w^{-1}(\gen{T_b})|_U,\rho)$ is nontrivial on $Q$.

\end{Lemma}
\begin{proof}
    By Koberda's \cite[Theorem~1.1]{koberda_right-angled_2012}, if three curves pairwise intersect each other, then a big enough power of the associated Dehn twists generate a free group. Thus, using that $\kappa(S)\geq 1$, take $a$ and $a'$ two curves of $\cxcourbes{S}$ such that there exists an integer $k$ such that $T_a^k,T_{a'}^k$ and $T_b^k$ generate a free group $F_3$. For instance, $a$ and $a'$ can be obtained by taking two intersecting curves on $S$ and applying a big enough power of a pseudo-Anosov of $S$ to these curves.

    Now, denote by 
    \begin{align*}
        \mclB^0=\w^{-1}(\gen{T_b^k})|_U\, ,\\
        \mclA^0=\w^{-1}(\gen{T_a^k})|_U\, ,\\
        \mclA'^0=\w^{-1}(\gen{T_{a'}^k})|_U
    \end{align*}

    and note that $\gen{\mclA^0,\mclA'^0}$ is everywhere nonamenable by Lemma~\ref{lemme action-type groupe non moyennable donne groupoide non moyennable}. Moreover, note that $\gen{\mclA^0,\mclA'^0}\subseteq \w^{-1}(\Modact{S})|_U$ which in particular means that 
        $$\rho \left(\gen{\mclA^0,\mclA'^0}\right) \subseteq \Modactprime{Q}\times \gen{T_c}_{c\in\s\setminus \partial Q}~.$$
    Define the projection $$\bar\rho :\gen{\mclA^0,\mclA'^0}\longrightarrow \gen{T_c}_{c\in\s\setminus \partial Q}$$ and define the subgroupoid $\mclK$ as the kernel of $\bar\rho$, i.e.\
    $$ \mclK=\{\mfk g\in \gen{\mclA^0,\mclA'^0}~|~ \rho(\mfk g)\in\Modactprime{Q}\}~.$$

    Let us show that $\mclK$ is of infinite type. For $n\in\N\cup\{\infty\}$, the Borel subsets
    $$U_n=\{x\in U~|~\#\mclK_{x}=n\}$$
    give the partition
    $$U=\bigsqcup_{n\in\N\cup\{\infty\}}U_n~.$$
    Each $U_n$ for $n\in\N$ can be cut up into $n$ Borel pieces such that $\mclK$ is trivial when restricted to each piece. Therefore, 
    $$U=U_{\infty}\sqcup\bigsqcup_{m\in\N} V_m$$
    such that $\mclK|_{V_m}$ is trivial for all $m$. But if $V_m$ has positive measure for some $m$, then $$\bar\rho : \gen{\mclA^0,\mclA'^0}|_{V_m}\longrightarrow\gen{T_c}_{c\in\s\setminus \partial Q}$$ is a cocycle with trivial kernel that sends a nonamenable subgroupoid to an amenable group. Therefore, by Lemma~\ref{lemma triv kernel cocycle amenable group amenable groupoid}, $V_m$ has measure zero for all $m$. Thus $U_{\infty}$ is conull in $U$ which exactly means that $\mclK$ is of infinite type.

    Assume that $\b$ intersects $Q$ trivially. Then, every curve of $\b$ is disjoint from $Q$. This means that the multitwist $\gen{T_c}_{c\in\b}$ is centralised by $\Modactprime{Q}$. In other words, if we denote $$G=\Modactprime{Q}\times \gen{T_c}_{c\in\b\setminus \partial Q}~,$$then $$\gen{T_c}_{c\in\b}\subseteq Z(G)~.$$
    Thus, 
    $$\rho(\mclK)\subseteq G\text{ and }\rho(\mclB^0)\subseteq Z(G)$$
    but
    $$\w(\mclK)\subseteq \gen{T_{a}^k,T_{a'}^k}\text{ and }\w(\mclB^0)\subseteq \gen{T_{b}^k}$$
    and we chose the curves $a,a'$ and $b$ to verify $$\gen{T_{a}^k,T_{a'}^k,T_{b}^k}=\gen{T_{a}^k,T_{a'}^k}\ast\gen{T_{b}^k}$$
    This contradicts Lemma~\ref{lemma free product into direct product}, hence $\b$ intersects $Q$ nontrivially.
    
    In particular, if $\b$ is the canonical reduction multicurve of $(\mclB,\rho)$ and if $(\mclB,\rho)$ were trivial on $Q$, then $Q$ would be contained in the subsurface filled by the canonical reduction set of $(\mclB,\rho)$. This would force $\b$ to be disjoint from the open subsurface $Q$, which is not the case here.
\end{proof}

\subsection{Maximal decompositions through a measurable embedding}\label{section maximal decompo meas}

In this section, we prove that $k$-maximal decompositions of $\S$ studied in Section~\ref{section combinatoire max decompos} are preserved by a self measurable embedding, up to a countable partition of the base space $X$. For this, we prove a strengthening of Kida's \cite[Corollary~5.25]{kida_introduction_2009}, based on similar ideas.

\begin{Lemma}\label{lemma max decompo are sent to max decompo}

Under Setting~\ref{setting for thm}, take a $k$-maximal decomposition of $\S$ and denote by $S_1,\dots,S_n$ the corresponding connected components and $\bar b =(b_1,\dots,b_{n-1})$ the corresponding multicurve, with $n=\nsig-k$.
For $1\leq i\leq n$ and $1\leq j< n$, let
\begin{align*}
\mclS_i = \w^{-1}(\Modact {S_i})\\
\mclB_j = \w^{-1}(\gen{T_{b_j}})~.
\end{align*}

Then, it holds that $\kappa(\S)\leq \kappa(\S')$. 

Assume additionally that $\S=\S'$. Then, there exists a countable conull partition of $X$ into Borel subsets such that for every piece $U$ of positive measure, there exists a $k$-maximal decomposition of $\S$ into connected components $Q_1,\dots,Q_n$ by a multicurve $\s$ such that the following hold.

\begin{enumerate}[(i)]
    \item \label{i max decompos}For all $i\le n$, $\kappa(S_i)=\kappa(Q_i)$ and $({\mclS_i}|_U,\rho)$ is irreducible nonamenable on $Q_i$ and trivial on $Q_j$ for every $j\neq i$. In particular, 
    $$\rho({\mclS_i}|_U)\subseteq\Modact{Q_i}\times \gen{T_c}_{c\in\s\setminus \partial Q_i}~.$$

    \item \label{ii max decompos}For all $j<n$, $$\rho({\mclB_j}|_U)\subseteq\gen{T_c}_{c\in\s}~.$$
    \item \label{iii max decompos}The map
    $$f_{\bar b,\s} : \begin{array}{ccl}
     V(\G_{\bar b})&\longrightarrow &V(\G_{\s}) \\
     S_i&\longmapsto & Q_i
\end{array}$$
    is a graph isomorphism. In other words, for any $i,j\le n$, 
    $S_i$ and $S_j$ share a boundary curve if and only if $Q_i$ and $Q_j$ do.
\end{enumerate}
\end{Lemma}

\begin{Rem}\label{rem on equality or not of S and S'}
    The first half of the proof of Lemma~\ref{lemma max decompo are sent to max decompo} (Points~\ref{1 proof} through \ref{5 proof}) holds when $\S$ and $\S'$ are not necessarily equal and implies monotonicity of topological complexity, used to prove Corollary~\ref{cr monotonicity of complexity} in Section~\ref{conclusion subsection}. It is the only result we show in the full generality of Setting~\ref{setting for thm}. After proving $\kappa(\S)\leq \kappa(\S')$, we will always assume that $\S=\S'$.
\end{Rem}

\begin{proof}
    By definition of $k$-maximal decompositions (see Definition~\ref{def kmax decompo}), choose the numbering of the $(S_i)_i$ such that the surfaces $S_1,\dots,S_{2k+\dsig}$ have topological complexity $2$ and the surfaces $S_{2k+\dsig+1},\dots,S_n$ have topological complexity $1$. 
    
    On $S_i$ for $i\leq 2k+\dsig$, let us choose five curves $a_i^{(1)}, \dots,a_i^{(5)}$ that have the following property: each of these five curves is disjoint from two curves of these five that intersect each other. Figure~\ref{fig five curves} shows such a choice of curves on a twice-punctured torus and a $5$-punctured sphere. On $S_i$ for $i>2k+\dsig$, simply take curves $a_i^{(1)}$ and $a_i^{(2)}$ that intersect. Denote $$\mclA_i^{(t)}=\w^{-1}(\gen{T_{a_i^{(t)}}})$$ for all these curves.
    \begin{figure}[ht]
	\centering
	\includesvg[width=0.9\textwidth]{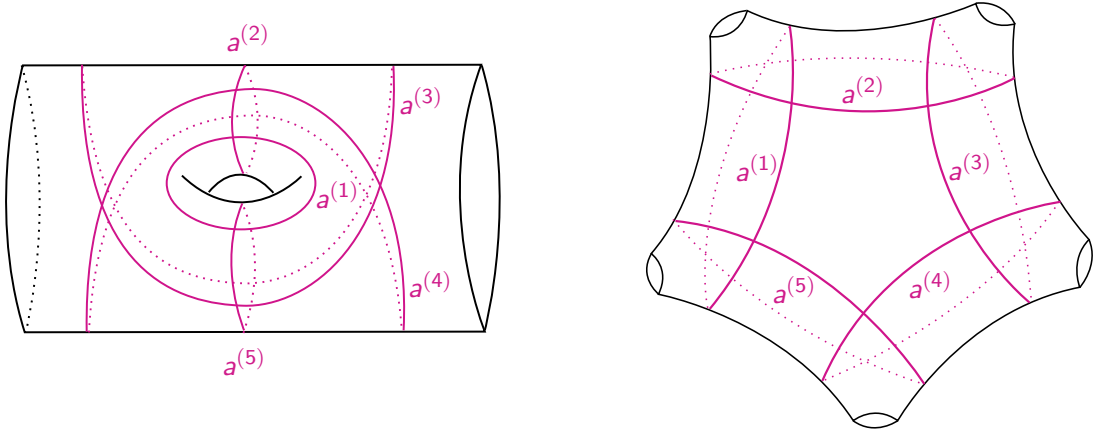}
    \caption{Curves $a_i^{(1)}, \dots,a_i^{(5)}$ on some component $S_i$ of complexity $2$.}
	\label{fig five curves}
\end{figure}

    Fix a conull partition of $X$ adapted to the canonical multicurves of $\mclS_i$, $\mclA_i^{(t)}$, and $\mclB_j$ for all $i,j$ with respect to the cocycle $\rho$. It is always possible by Lemma~\ref{canonical reduction set exists} and Remark~\ref{rem we can always choose U for the canonical multicurve to exist}. For a piece $Y$ of that partition of positive measure, refine it into a new partition adapted to $$\gen{\mclA_i^{(1)}|_Y,\dots,\mclA_i^{(5)}|_Y}\text{ for }i\leq 2k+\dsig$$ and $$\gen{\mclA_i^{(1)}|_Y,\mclA_i^{(2)}|_Y}\text{ for }i>2k+\dsig$$ with respect to $\rho$.
    Take a piece $U$ of that refined partition and define $$\mclR_i\defeq\gen{\mclA_i^{(1)}|_U,\dots,\mclA_i^{(5)}|_U}\text{ for }i\leq 2k+\dsig$$ and $$\mclR_i\defeq\gen{\mclA_i^{(1)}|_U,\mclA_i^{(2)}|_U}\text{ for }i>2k+\dsig~.$$ From now on and until the end of the proof of this lemma, we can assume that $X=U$ and thus that all the canonical multicurves of all the subgroupoids $\mclS_i$, $\mclA_i^{(t)}$, $\mclB_j$, $$\mclR_i=\gen{\mclA_i^{(1)},\dots,\mclA_i^{(5)}}\text{ for }i\leq 2k+\dsig$$ and $$\mclR_i=\gen{\mclA_i^{(1)},\mclA_i^{(2)}}\text{ for }i>2k+\dsig$$
    with respect to $\rho$ exist.
    
\medbreak
    Let us now list all the information we get on our subgroupoids from their definition via the action-type cocycle $\w$. Outside of this list, we will not use the cocycle $\w$ in the proof of this lemma.
    \begin{itemize}
        \item For all $i\leq n$, $\mclR_i$ and $\mclS_i$ are everywhere nonamenable by Lemma~\ref{lemme action-type groupe non moyennable donne groupoide non moyennable}.
        \item For all $j<n$, $\mclB_j$ is amenable by Lemma~\ref{lemma triv kernel cocycle amenable group amenable groupoid}.
        \item For all $i\neq j\leq n$, $\mclS_i$ and $\mclR_i$ normalise $\mclS_j,\mclR_j$ and $\mclA_j^{(t)}$ for any well-defined $t$ (i.e.\ any $t\leq 5$ if $i\leq 2k+\dsig$ and any $t\leq 2$ if $i>2k+\dsig$), by Lemma~\ref{lemma if groups normalise then groupoids normalise}.
        \item For all $i\leq n$ and $j<n$, $\mclB_j$ and $\mclR_i$ normalise each other, and $\mclB_j$ and $\mclS_i$ normalise each other, again by Lemma~\ref{lemma if groups normalise then groupoids normalise}.
        \item For every $i\leq 2k+\dsig$ and every $t\leq 5$, there exist $s,s'\leq 5$ such that $i(a_i^{(s)},a_i^{(s')})\geq 1$ and $a_i^{(t)}$ is distinct and disjoint from both $a_i^{(s)}$ and $a_i^{(s')}$. Then, it holds that
        \begin{itemize}[noitemsep, nolistsep, label=-]
            \item $\gen{\mclA_i^{(s)},\mclA_i^{(s')}}$ is everywhere nonamenable by Lemma~\ref{adjacency curves in groupoidal terms};
            \item $\gen{\mclA_i^{(s)},\mclA_i^{(s')}}$ and $\mclA_i^{(t)}$ normalise each other by Lemma~\ref{lemma if groups normalise then groupoids normalise};
            \item $\mclA_i^{(t)}$ is amenable by Lemma~\ref{lemma triv kernel cocycle amenable group amenable groupoid}.
        \end{itemize}
    \end{itemize}

    We first study the subgroupoids $\mclR_i$ and obtain monotonicity of $\kappa$.
    Let $\s_i$ be the canonical multicurve of $(\mclR_i,\rho)$ for every $i$. For every $j\neq i$, $\mclR_i$ normalises $\mclR_j$ so by Lemma~\ref{lemma multicurve is canonical}, $\s_i$ and $\s_j$ are disjoint and $(\mclR_i,\rho)$ stabilises $\bigcup_{j\neq i}\s_j$. Therefore, every $(\mclR_i,\rho)$ stabilises the whole multicurve
    $$\s=\bigcup_i\s_i~.$$

     \begin{enumerate}[(1)]
        \item\label{1 proof}  Let $Q$ be a connected component of $\S'\setminus \s$. If $(\mclR_i,\rho)$ is nontrivial and reducible on $Q$, then there would be a curve of $\s_i$ on $Q$ by Proposition~\ref{reducible nontrivial implies existence of canonical curve}, which is not the case. So $(\mclR_i,\rho)$ is either trivial or irreducible on $Q$.
        \item \label{2 proof}For any $i\leq n$, since $\mclR_i$ is everywhere nonamenable, Corollary~\ref{if IA or T everywhere, then amen. If nonamen then IN somewhere} ensures that $(\mclR_i,\rho)$ is irreducible nonamenable on at least one connected component $Q$ of $\S'\setminus \s$. In particular, there necessarily exists a $t$ such that $(\mclA_i^{(t)},\rho)$ is nontrivial on $Q$, otherwise $(\mclR_i,\rho)$ would be trivial on $Q$. Note that the latter statement is the reason why we first study $\mclR_i$ before we can study $\mclS_i$.
        \item\label{3 proof} Fix $i\neq j$ and take $Q$ a connected component of $\S'\setminus\s$ on which $(\mclR_i,\rho)$ is irreducible nonamenable. Let us show that $(\mclR_j,\rho)$ is forced to be trivial on $Q$. Fix $t$ such that $(\mclA_i^{(t)},\rho)$ is nontrivial on $Q$ by Point~\ref{2 proof}. The subgroupoid $\mclR_j$ normalises the amenable subgroupoid $\mclA_i^{(t)}$ so by Proposition~\ref{if H normalises A then its reducible or IA}, $(\mclR_j,\rho)$ is either reducible or irreducible amenable on $Q$. If $(\mclR_j,\rho)$ were nontrivial on $Q$, apply Proposition~\ref{if H normalises A then its reducible or IA} to $\mclR_i$ normalising $\mclR_j$. This would make $(\mclR_i,\rho)$ reducible or irreducible amenable on $Q$, which is a contradiction. So $(\mclR_j,\rho)$ is trivial on $Q$.
    
        \item \label{4 proof}For $i\leq 2k+\dsig$, let us show that $(\mclR_i,\rho)$ is either irreducible nonamenable on some connected component $Q$ of $\S'\setminus\s$ of topological complexity $\kappa(Q)\geq 2$ or it is irreducible nonamenable on at least two connected components of complexity at least $1$. The case to eliminate here is that $(\mclR_i,\rho)$ is irreducible nonamenable on exactly one connected component $Q$ of complexity $1$ and irreducible amenable on some other connected components $Q_1,\dots ,Q_m$, where $m$ can be zero. So let us assume this situation and arrive to a contradiction. 

        Up to renumbering, $(\mclA_i^{(1)},\rho)$ is nontrivial on $Q$ by Point~\ref{2 proof} and is normalised by the everywhere nonamenable groupoid $$\mclH_i\defeq\gen{\mclA_i^{(2)},\mclA_i^{(3)}}~.$$

        Note that by Point~\ref{1 proof}, $(\mclR_i,\rho)$ is trivial on all the components that are not $Q_1,\dots,Q_m$, $Q$. Then, $(\mclH_i,\rho)$ is also trivial on these components because $\mclH_i$ is a subgroupoid of $\mclR_i$. 
        
        For the components $Q_1,\dots,Q_m$, there exists a measurable and $(\mclR_i,\rho)$-equivariant map 
        $$\varphi : X\longrightarrow \partial_{\leq2}\cxcourbes{Q_1}\times \cdots\times \partial_{\leq 2}\cxcourbes{Q_m}~.$$
        by definition of irreducible amenability. This map is in particular $(\mclH_i,\rho)$-equivariant.
        
        It is left to study $(\mclH_i,\rho)$ on $Q$. The subgroupoid $\mclH_i$ normalises the amenable subgroupoid $\mclA_i^{(1)}$ so Proposition~\ref{if H normalises A then its reducible or IA} gives that $(\mclH_i,\rho)$ is either reducible or irreducible amenable on $Q$.
        \begin{itemize}
            \item If $(\mclH_i,\rho)$ is irreducible amenable on $Q$, then there exists an $(\mclH_i,\rho)$-equivariant map $$\psi: X \longrightarrow\partial_{\leq2}\cxcourbes{Q}~.$$
            Then, Lemma~\ref{map to product of boundaries of curve graphs of components} applies to the multicurve $\s$ and the map $\varphi\times\psi$ and makes $\mclH_i$ amenable.
            
            \item If $(\mclH_i,\rho)$ fixes a curve $c$ on $Q$, then 
            $$\rho(\mclH_i)\subseteq \Stab {\s\sqcup\{c\}}$$
            but $Q$ having complexity $1$ means that $Q\setminus c$ is one or two subsurfaces of complexity $0$ so $(\mclH_i,\rho)$ is trivial on $Q\setminus \{c\}$. Hence Lemma~\ref{map to product of boundaries of curve graphs of components} applies to the multicurve $\s\sqcup \{c\}$ and the map $\varphi$ and makes $\mclH_i$ amenable.
        \end{itemize}
        In either case, we have shown that $\mclH_i$ is amenable by applying Kida's Lemma~\ref{map to product of boundaries of curve graphs of components}. But $\mclH_i$ is an everywhere nonamenable subgroupoid, hence the contradiction.

        \item\label{5 proof} Let us consider the dual graph $\G_{\s}=(V,E)$ associated to the multicurve $\s$ of $\S'$. Note that we have not yet eliminated the case where a curve of $\s$ is the boundary of only one component, which is represented by a loop in the graph. Recall (\ref{equation star}) from Remark~\ref{rem equation star}
        $$\kappa(\S')+1=2\#V+\sum_{v\in V}\left(\kappa(v)-1\right) + b_1~.$$
        Denote by $I$ the following set of indices
        $$I\defeq \{i\leq 2k+\dsig~|~ \exists v,v'\in V,~ v\neq v',~(\mclR_i,\rho) \text{ is irreducible nonamenable on }v\text{ and }v'\}~.$$
        For $i\neq j\leq n$, if $(\mclR_i,\rho)$ is irreducible nonamenable on a component $v$ and $(\mclR_j,\rho)$ is irreducible nonamenable on a component $v'$, then $v\neq v'$ by Point~\ref{3 proof}. Thus every $i\leq n$ contributes at least $1$ to $\#V$ and every $i\in I$ contributes at least $2$, which gives
        \begin{equation}\label{inequality on card V}
            \#V\geq \nsig -k+\# I
        \end{equation}
        because $n=\nsig-k$.

        Point~\ref{4 proof} tells us that for $i\leq 2k+\dsig$, if $(\mclR_i,\rho)$ is irreducible nonamenable on only one component $v$, then necessarily $\kappa(v)\geq 2$. Therefore, if $i\leq 2k+\dsig$ but $i\notin I$, then it contributes at least $1$ to the sum $\sum_{v\in V}\left(\kappa(v)-1\right)$, which gives
        \begin{equation}\label{inequality on sum complexities}
            \sum_{v\in V}\left(\kappa(v)-1\right)\ge 2k+\dsig -\# I~.
        \end{equation}
        Summing twice the inequality (\ref{inequality on card V}) with the inequality (\ref{inequality on sum complexities}) and using $\# I\geq 0$ and $b_1\geq 0$, we get
        \begin{align}\label{ineq kappa}
           \nonumber \kappa(\S')+1&=2\#V+\sum_{v\in V}\left(\kappa(v)-1\right) + b_1\\ \nonumber
                        &\geq 2\nsig-2k+2\#I+2k+\dsig -\#I+b_1\\\nonumber
                        &= 2\nsig+\dsig  +\#I+b_1\\
                        &\geq 2\nsig+\dsig = \kappa(\S)+1~.
        \end{align}
        We have shown that $\kappa(\S)\leq \kappa(\S')$.
        \end{enumerate}

        \textbf{From now on, assume that $\S=\S'$.} We first show Point~\ref{i max decompos} of the lemma for the subgroupoids $\mclR_i$ and then we will deduce it for the $\mclS_i$. 
        
        Now that $\S=\S'$, the inequalities of (\ref{ineq kappa}) are all equalities. This forces the inequalities (\ref{inequality on card V}), (\ref{inequality on sum complexities}) as well as $\# I\geq 0$ and $b_1\geq 0$ to all be equalities.
        \begin{itemize}[-]
            \item The fact that $b_1=0$ means that the graph $\G_{\s}$ is a tree.
            \item The fact that $\# I=0$ means that for every $i\leq 2k+\dsig$, $(\mclR_i,\rho)$ is irreducible nonamenable on exactly one component that we denote $Q_i$.
            \item The fact that (\ref{inequality on card V}) is an equality means that for every $j> 2k+\dsig$, $(\mclR_j,\rho)$ is irreducible nonamenable on exactly one component that we denote $Q_j$.
            \item The fact that (\ref{inequality on sum complexities}) is an equality means that $\kappa(Q_i)=2$ for every $i\leq 2k+\dsig$ and $\kappa(Q_j)=1$ for every $j>2k+\dsig$.
        \end{itemize}
        To sum up, we have shown that $\s$ is a $k$-maximal decomposition of $\S$ with connected components $Q_1,\dots,Q_{2k+\dsig}$ of complexity $2$ and $Q_{2k+\dsig+1},\dots,Q_n$ of complexity $1$. Moreover, for every $i\leq n$, $(\mclR_i,\rho)$ is irreducible nonamenable on $Q_i$ and trivial on $Q_j$ for every $j\neq i$. In other words, we have shown Point~\ref{i max decompos} for the subgroupoids $\mclR_i$.

     Now, let us show Point~\ref{i max decompos} for the subgroupoids $\mclS_i$. Fix $1\le i\neq j \le n$.
     \begin{itemize}[nolistsep, noitemsep]
        \item The subgroupoid $(\mclS_i,\rho)$ cannot be trivial on $Q_i$ because it contains $\mclR_i$.
        \item If $(\mclS_i,\rho)$ were reducible on $Q_i$, then there would exist a curve $c$ on $Q_i$ such that 
        $$\rho(\mclS_i)\subseteq \Stab c$$
        but $\mclR_i\subseteq\mclS_i$, so
        $$\rho(\mclR_i)\subseteq \Stab c$$
        which contradicts the irreducibility of $(\mclR_i,\rho)$ on $Q_i$. Thus $(\mclS_i,\rho)$ is irreducible on $Q_i$.
        \item If $(\mclS_i,\rho)$ were irreducible amenable on $Q_i$, then the associated $(\mclS_i,\rho)$-equivariant map would also be $(\mclR_i,\rho)$-equivariant because $\mclR_i\subseteq \mclS_i$. This would contradict the fact that $(\mclR_i,\rho)$ is irreducible nonamenable. Thus $(\mclS_i,\rho)$ is irreducible nonamenable on $Q_i$.
        \item Fix $t$ such that $(\mclA_i^{(t)},\rho)$ is nontrivial on $Q_i$, which exists by Point~\ref{2 proof}. The subgroupoid $\mclS_j$ normalises the amenable subgroupoid $\mclA_i^{(t)}$ so by Proposition~\ref{if H normalises A then its reducible or IA}, $(\mclS_j,\rho)$ is either reducible or irreducible amenable on $Q_i$. If $(\mclS_j,\rho)$ were nontrivial on $Q_i$, apply Proposition~\ref{if H normalises A then its reducible or IA} to $\mclS_i$ normalising $\mclS_j$. This gives a contradiction with the fact that $(\mclS_i,\rho)$ is irreducible nonamenable on $Q_i$. So $(\mclS_j,\rho)$ is trivial on $Q_i$.
     \end{itemize}
     This concludes the proof of Point~\ref{i max decompos}.

\medbreak
     Now, let us prove Point~\ref{ii max decompos}. Fix $j<n$. Note that $\mclB_j$ normalises $\mclR_i$ for all $i\le n$, so $(\mclB_j,\rho)$ stabilises $\s$ by Lemma~\ref{lemma multicurve is canonical}. By contradiction, assume that $(\mclB_j,\rho)$ is nontrivial on $Q_i$ for some $i\le n$. It is amenable and normalised by $\mclR_i$. By Proposition~\ref{if H normalises A then its reducible or IA}, this makes $(\mclR_i,\rho)$ reducible or irreducible amenable on $Q_i$ which contradicts the definition of $Q_i$. Therefore, $(\mclB_j,\rho)$ is trivial on every component of $\S\setminus\s$ so $\mclB_j$ is sent by $\rho$ into the multitwist generated by the curves of $\s$. This concludes the proof of Point~\ref{ii max decompos}.
\medbreak
     Finally, we show Point~\ref{iii max decompos}. 
    Take $\G_{\bar b}$ the graph dual to the decomposition $\bar b$, which has vertices $(S_i)_{i\leq n}$. Similarly, $\G_{\s}$ the graph dual to $\s$ has vertices $(Q_i)_{i\leq n}$. The map $f_{\bar b,\s}$
    that sends $S_i$ to $Q_i$ is injective by Point~\ref{i max decompos}. 

     Now, let us show that $f_{\bar b,\s}$ sends two adjacent vertices to two adjacent vertices. Up to renumbering, assume that $S_1$ and $S_2$ share a boundary curve $b_1$ of the decomposition $\bar b$. Choose a curve $b'$ on $S_1\cup\{b_1\}\cup S_2$ that intersects both $S_1$ and $S_2$ nontrivially and such that the two components of $$\left(S_1\cup\{b_1\}\cup S_2 \right)\setminus b'~,$$ denoted $S_1'$ and $S_2'$, verify $\kappa(S_1)=\kappa(S_1')$ and $\kappa(S_2)=\kappa(S_2')$. 
     Then, $\bar b'=(b',b_2,\dots,b_{n-1})$ is still a $k$-maximal decomposition. 
     Restrict to a smaller positive measure subset $V$ of $U$ to apply Points~\ref{i max decompos} and \ref{ii max decompos} to this new decomposition $\bar b'$. Name $\s'$ the image multicurve and $Q_1',\dots,Q_n'$ the components of the image decomposition. 
     
      For $i\leq n$, the subsurface $Q_i$ (resp. $Q_i'$) is the unique subsurface of $\S$ on which $(\mclS_i|_V,\rho)$ (resp. $(\mclS_i'|_V,\rho)$) is irreducible nonamenable, by Point~\ref{i max decompos}. Thus, the definition of $Q_i$ only depends on the subsurface $S_i$. But recall that the two decompositions $(S_i)_i$ and $(S_i')_i$ at the source differ only by the first two subsurfaces $S_1$ and $S_2$. Thus, for $i\geq 3$, $$\mclS_i|_V=\mclS_i'|_V$$ so $(\mclS_i|_V,\rho)$ is irreducible nonamenable on both $Q_i$ and $Q_i'$. Therefore, $Q_i=Q_i'$ for all $i\geq 3$.

    Assume by contradiction that $Q_1$ and $Q_2$ do not share any boundary curve. This means that
    $$\s=\bigcup_{i\geq 3} \partial Q_i = \bigcup_{i\geq 3} \partial Q_i '= \s'~.$$
    Denote by $\mclB'=\w^{-1}(\gen{T_{b'}})$. By Point~\ref{ii max decompos} for the decomposition $\bar b'$,  
    $$\rho(\mclB'|_V)\subseteq\gen{T_c}_{c\in\s'}~.$$
    But here we have $\s=\s'$, so 
    $$\rho(\mclB'|_V)\subseteq\gen{T_c}_{c\in\s}~.$$
    The subset $V$ is chosen small enough for the canonical multicurve $\b'$ of $(\mclB'|_V,\rho)$ to exist. By the latter, $\b'\subseteq \s$. However, $b'$ intersects $S_1$ and $S_2$ nontrivially so by Lemma~\ref{lemma b intersects S}, $\b'$ intersects both $Q_1$ and $Q_2$ nontrivially, which contradicts $\b'\subseteq \s$.

    We have shown that $f_{\bar b,\s}$ sends two adjacent vertices of $\G_{\bar b}$ to two adjacent vertices of $\G_{\s}$. But these two graphs are finite trees with the same number $n$ of vertices. This means that the map $f_{\bar b,\s}$ is a graph isomorphism.
\end{proof}

\subsection{Dehn twists through a measurable embedding}\label{section curves meas}

The goal of this section is to show that a cyclic subgroup generated by a Dehn twist at the source is sent into exactly one Dehn twist cyclic subgroup at the target, up to a countable partition of the base space $X$. More precisely, the goal is to show the following key proposition. 

\begin{Prop}\label{prop recap courbes dans courbes}
    Under Setting~\ref{setting for thm}, assume that $\S=\S'$ and $\kappa(\S)\geq 2$. Let $a\in\cxcourbes{\S}$ and let $\mclA=\w^{-1}\left(\gen{T_a}\right)$. Then, there exists a countable conull Borel partition $\ds \bigsqcup_m X_m$ of $X$ such that for each $m$, there exists $\alpha_m\in\cxcourbes{\S}$ such that 
    $$\rho(\mclA|_{X_m})\subseteq \gen{T_{\a_m}}~.$$
\end{Prop}

We will use the results we have obtained on $k$-maximal decompositions in Lemma~\ref{lemma max decompo are sent to max decompo}. For this, recall Definition~\ref{def types de courbes} of the different types of separating curves. We prove the key proposition for leaf curves in Proposition~\ref{prop leaf type curves}. Then, using the latter, we independently prove the key proposition for weakly separating and nonseparating curves in Proposition~\ref{nonsep curves prop} and for strongly separating curves in Proposition~\ref{separating curves prop}. Note that if $\kappa(\S)=2$, then Proposition~\ref{nonsep curves prop} suffices because all curves are nonseparating or weakly separating.

\subsubsection{Leaf curves}The first and most technical step is to prove Proposition~\ref{prop recap courbes dans courbes} for leaf curves. Leaf curves are called this way because they correspond to edges that connect to a leaf in the dual graph of a multicurve (see Definition~\ref{def dual graph}). Thus, this specific type of curves verify the following property: the data of a leaf curve forces to fix the data of the whole corresponding leaf subsurface of the decomposition. This additional rigidity is the reason we work on these types of curves first.

\begin{Prop}\label{prop leaf type curves}
    Under Setting~\ref{setting for thm}, assume that $\S=\S'$ and $\kappa(\S)\geq 3$. Let $b\in\mclC_{leaf}$ and let $\mclB=\w^{-1}\left(\gen{T_b}\right)$. Then, there exists a countable conull Borel partition $\bigsqcup_m X_m$ of $X$ such that for each $m$, there exists $\b_m\in\mclC_{leaf}$ such that 
    $$\rho(\mclB|_{X_m})\subseteq \gen{T_{\b_m}}~.$$
\end{Prop}
\begin{proof} When $\kappa(\S)\in\{3,4\}$, $\nsig=2$ so any single leaf curve is a $0$-maximal decomposition multicurve of $\S$. Thus, apply Lemma~\ref{lemma max decompo are sent to max decompo} directly to the $0$-maximal decomposition multicurve $\{b\}$, fix a piece $Y\subseteq X$ of the partition given by Lemma~\ref{lemma max decompo are sent to max decompo} and get a $0$-maximal decomposition multicurve $\s=\{\b\}$ composed of one unique leaf curve $\b$ and verifying $\rho(\mclB|_{Y})\subseteq \gen{T_{\b}}$. This concludes the proof in the case $\kappa(\S)\leq 4$. 

    From now on, assume that $\kappa(\S)\geq 5$.
    \begin{Claim}\label{claim at most one leaf curve}
   Assume that $Y$ is a Borel subset of positive measure such that the canonical multicurve $\b$ of $(\mclB|_Y,\rho)$ exists. Then, $\b$ contains at most one leaf curve.
    \end{Claim}
    \begin{proof}[Proof of Claim~\ref{claim at most one leaf curve}]
    Any leaf curve can be completed into a $0$-maximal decomposition, see for instance the proof of Proposition~\ref{prop D les decompos particulieres avec g+p/3} for a concrete construction. Apply Point~\ref{ii max decompos} of Lemma~\ref{lemma max decompo are sent to max decompo} to get that $\mclB|_{Z}$ is sent by $\rho$ into a multitwist for a small enough positive measure subset $Z\subseteq Y$. In other words, $(\mclB|_Z,\rho)$ is trivial on every connected component of $\S\setminus \b$, so $$\rho\left(\mclB|_Z\right)\subseteq \gen{T_c}_{c\in\b}~.$$
          
    First, let us build everything we need on the surface at the source of the measurable embedding. Complete $b$ into a $0$-maximal decomposition $(b_1,\dots,b_{n-1})$ where $b_1=b$ and denote $S_1,\dots,S_n$ the components of this decomposition, where $S_1$ is an open leaf subsurface separated from $\S$ by $b$. Recall that this means that $n=\nsig$ and all the curves $b_j$ are separating. Now, complete the leaf curve $b$ into another $0$-maximal decomposition $(b_1',b_2',\dots,b_{n-1}')$ where $b_1'=b$ and such that for all $i,j\ge 2$, $b_j'$ intersects the open subsurface $S_i$ nontrivially. Denote $\mclB_j=\w^{-1}(\gen{T_{b_j}})$ and $\mclB_j'=\w^{-1}(\gen{T_{b'_j}})$.

    Now, let us give notations to what happens on the target copy of $\S$. Apply Lemma~\ref{lemma max decompo are sent to max decompo}, restrict to a small enough Borel subset $U$ of positive measure, and denote by $\s$ (resp. $\s'$) the $0$-maximal decomposition multicurve at the target that we get from the decomposition $(b_j)_j$ (resp. $(b_j')_j$). For every $i\leq n$, denote by $Q_i$ the component of $\S\setminus \s$ on which $\w^{-1}(\Modact {S_i})|_U$ is irreducible nonamenable. Similarly, $Q_i'$ is the component of $\S\setminus \s'$ on which $\left(\w^{-1}(\Modact {S_i'})|_U,\rho\right)$ is irreducible nonamenable.
    Apply Lemma~\ref{lemma b intersects S} to $b=b_j'$, $S=S_i$ and $Q=Q_i$ to obtain that $(\mclB_j'|_U,\rho)$ is nontrivial on $Q_i$ for every $2\leq i\leq n$ and $2\leq j \leq n-1$.
    
    Assume by contradiction that $\b$ contains two distinct leaf curves $a$ and $a'$. The multicurve $\b$ is contained in both $\s$ and $\s'$, so $\{a,a'\}\subseteq \s\cap\s'$. This means that there are two leaf subsurfaces of the target $\S$ that are a component in both decompositions. One of them is $Q_1=Q_1'$ coming from $S_1=S_1'$. But there is a second one, so there exist $i,i'\geq 2$ such that $\partial Q_i=\partial Q_{i'}'$ is a leaf curve of $\b$, which forces the leaf subsurfaces $Q_i$ and $Q_{i'}'$ to be equal. For $j\geq 2$, by Point~\ref{ii max decompos} of Lemma~\ref{lemma max decompo are sent to max decompo},
    $$\rho\left(\mclB_j'|_U\right)\subseteq \gen{T_{c}}_{c\in\s'}\subseteq \Modfix{Q_{i'}'}$$
    and $Q_i=Q_{i'}'$ so $$\Modfix{Q_{i}}=\Modfix{Q_{i'}'}~.$$
    Therefore, 
    $$\rho\left(\mclB_j'|_U\right)\subseteq \Modfix{Q_{i}}$$
    which contradicts the fact that $(\mclB_j'|_U,\rho)$ is nontrivial on $Q_i$.
    \end{proof}

    Let us now prove that preimages by $\w$ of leaf curve twists cannot have any non-leaf curves in their canonical multicurve for $\rho$.
    For this, we take a step back and consider the collection $D$ of all the $0$-maximal decomposition multicurves $\s$ of $\S$ that verify 
    $$\# \left(\s\cap \mclC_{leaf}\right) >\# \left(\s\cap \mclC_{nl}\right)~.$$
    To be clear, every element of $D$ is a multicurve of $n-1$ separating curves of $\S$, where $n=\nsig$. By Proposition~\ref{prop D les decompos particulieres avec g+p/3}, any leaf curve of $\S$ can be completed into a multicurve in $D$. 
    We will consider the preimages by $\w$ of all the multicurves in $D$ and all their possible canonical multicurves for $\rho$.

    For every multicurve $\bar c\in D$, there exists a countable conull Borel partition $$X^*=\bigsqcup_{m\in\N} X_m^{\bar c}$$ such that the canonical multicurves of all the $(\w^{-1}(\gen{T_{c}})|_{X_m^{\bar c}},\rho)$, denoted $\g_m(c)$, exist for all $c\in\bar c$ and $m$, as explained in Remark~\ref{rem we can always choose U for the canonical multicurve to exist}. Given a multicurve $\bar c$ in $D$ and an integer $m$, define the following multicurve
    \begin{equation*}\label{def de sigma nl}
        \s_{nl}(\bar c ,X_m^{\bar c})=\mclC_{nl}\cap \bigcup_{c\in \bar c \,\cap\,\mclC_{leaf}}\g_m(c)~.\tag{Def$_{\s_{nl}}$}
    \end{equation*}
    It is the set of non-leaf curves that are in the image of some leaf curve of $\bar c$, on some positive measure subset of $X$.
\medbreak 
    The cardinality of a multicurve of $\S$ is bounded, so we fix a $0$-maximal decomposition $\bar c$ and an $m\in\N$ such that $$\#\s_{nl}(\bar c,X_m^{\bar c})\eqdef N$$ is maximal amongst all possible values of $\s_{nl}$ on $D\times\N$. To prove the proposition, it suffices to show that $N=0$ because of Claim~\ref{claim at most one leaf curve}. So let us assume that $N\geq 1$ and aim towards a contradiction.
    
    To fix more notations, $\bar c = (c_1,\dots,c_{n-1})$ are separating curves that decompose the source $\S$ into components $S_1,\dots,S_n$. Apply Lemma~\ref{lemma max decompo are sent to max decompo} to the latter, restrict to a positive measure Borel subset $U$ of $X_m^{\bar c}$ and get the corresponding $0$-maximal decomposition $Q_1,\dots,Q_n$ by a multicurve $\s$ on the target surface $\S$. Point~\ref{ii max decompos} of Lemma~\ref{lemma max decompo are sent to max decompo} ensures that for every $c\in\bar c$, 
    $$\rho(\w^{-1}(\gen{T_{c}})|_{U})\subseteq \gen{T_a}_{a\in\s}$$
    so in particular it holds that $$\g_m(c)\subseteq \s~.$$Hence
    $$\s_{nl}(\bar c,X_m^{\bar c})\subseteq \s~.$$
    
    For simplicity, denote $\gamma_m(c_i)$ by $\g_i$ and remember that $\g_i\subseteq \s$.
    Renumber $\bar c$ so that $c_1,\dots,c_N$ are leaf curves that verify
    $$\s_{nl}(\bar c,X_m^{\bar c})=\mclC_{nl}\cap\bigcup_{i=1}^N \g_i~.$$
    In other words, we choose exactly $N$ leaf curves from $\bar c$ that give the whole $\s_{nl}(\bar c,X_m^{\bar c})$. Renumber the components $(S_i)_{i\leq n}$ of the decomposition $\bar c$ so that for $i\leq N$, $S_i$ is the component that $c_i$ separates from the rest. Also, note that $N<n$ because $\kappa(\S)\geq 5$ so there exist non-leaf components in the decomposition $(S_i)_{i\leq n}$.

    Now, define a new $0$-maximal decomposition $\bar c '$ by keeping the $c_1,\dots,c_N$ the same and choosing $c_{N+1}',\dots,c_{n-1}'$ that verify the following: for each $N<i\leq n$ and $N<j<n$, $c_j'$ intersects $S_i$ nontrivially. This is possible for instance by applying a pseudo-Anosov mapping class of the subsurface $\S\setminus\bigcup_{1\leq i\leq N}S_i$, which has positive complexity because $N<n$ and is connected because all the components $S_i$ for $i\leq N$ correspond to leaves in the dual graph. Apply Lemma~\ref{lemma max decompo are sent to max decompo} to this new decomposition $\bar c'$. Restrict  to a small enough Borel subset $V$ of $U$ of positive measure on which Lemma~\ref{lemma max decompo are sent to max decompo} holds, so that in particular the canonical multicurve $\g_j'$ of $(\w^{-1}(\gen{T_{c_j'}})|_V,\rho)$ exists. By Point~\ref{ii max decompos} of Lemma~\ref{lemma max decompo are sent to max decompo}, $\w^{-1}(\gen{T_{c_j'}})|_V$ is sent to a multitwist by $\rho$. Thus, by Lemma~\ref{lemma b intersects S} applied to $b=c_j'$, $S=S_i$ and $Q=Q_i$ for all $N<i\leq n$ and $N<j<n$, the multicurve $\g_j'$ intersects $Q_i$ nontrivially.

    Denote by $\s'$ the $0$-maximal decomposition we get at the target from $\bar c'$ on the Borel subset $V$, by Point~\ref{i max decompos} of Lemma~\ref{lemma max decompo are sent to max decompo}. The fact that $\bar c'$ contains $c_1,\dots,c_N$ guarantees that 
    $$\s'\supseteq \bigcup_{i=1}^N \g_i\supseteq \s_{nl}(\bar c,X_m^{\bar c})~.$$
    But recall that for every $j<n$, $\g_j'\subseteq\s' $. In particular, this means that all the $\g_j'$ are disjoint from $\s_{nl}(\bar c,X_m^{\bar c})$. Thus, for every $j<n$, each curve of $\g_j'$ is either in $\s_{nl}(\bar c,X_m^{\bar c})$ or in $\cxcourbes{\S\setminus \s_{nl}(\bar c,X_m^{\bar c})}$. 

    \begin{Claim}\label{claim g_j' at least two curves}
        For $j>N$, $\g_j'$ contains at least two distinct curves in $\cxcourbes{\S\setminus \s_{nl}(\bar c,X_m^{\bar c})}$.
    \end{Claim}
    \begin{proof}[Proof of Claim~\ref{claim g_j' at least two curves}]
        
    Fix $j>N$. To prove this, we show that $\g_j'$ has curves on at least two distinct connected components of $\S\setminus \s_{nl}(\bar c,X_m^{\bar c})$. Consider the dual tree $\G_{\s}$ of the decomposition $\s$ as defined in Definition~\ref{def dual graph}. The edges of $\G_{\s}$ correspond to curves of $\s$. Let $E_{nl}$ be the edges of $\G_{\s}$ corresponding to $\s_{nl}(\bar c,X_m^{\bar c})$ and note that $\#E_{nl}=N$. Let $\G_1,\dots,\G_{N+1}$ be the connected components of the subforest of $\G_{\s}$ with vertices $V(\G_{\s})$ and edges $E(\G_{\s})\setminus E_{nl}$. Let us bound the size of these connected components. The tree $\G_{\s}$ has at least two leaves, so up to renumbering, $\G_1$ and $\G_2$ contain a leaf. But $E_{nl}$ contains no edges that connect to a leaf, so $\#V(\G_1),V(\G_2)\geq 2$. The rest of the connected components verify $\#V(\G_i)\geq 1$, which gives
    \begin{align*}
        \#V(\G_1)&\leq \#V(\G) - \#V(\G_2)-\sum_{j=3}^{N+1} \#V(\G_j)\\
                &\leq n -2 -\sum_{j=3}^{N+1} 1\\
                &\leq n-2-(N-1) = n-(N+1)
    \end{align*}
    and similarly $ \#V(\G_2)\leq n-(N+1)$.
    For $i\leq 3$, it holds that
    \begin{align*}
        \#V(\G_i)&\leq \#V(\G) -\#V(\G_1)- \#V(\G_2)-\sum_{\substack{3\le j\le N+1\\ j\ne i}} \#V(\G_j)\\
                &\leq n -2-2 -(N-2)\\
                &\leq n-(N+1)
    \end{align*}
    In conlusion, every connected component of $\G_{\s}$ with erased $E_{nl}$ edges has at most $n-(N+1)$ vertices.

    By Point~\ref{iii max decompos} of Lemma~\ref{lemma max decompo are sent to max decompo}, the graph isomorphism $f_{\bar c,\s}$ sends the vertices $S_{N+1},\dots,S_n$ to vertices $Q_{N+1},\dots,Q_n$ inside of ${\S \setminus\s_{nl}(\bar c,X_m^{\bar c})}$. In the point of view of dual graphs, this means exactly that $f_{\bar c,\s}$ is sending $n-(N+1)+1$ vertices of $\G_{\bar c}$ into the subforest of $\G_{\s}$ with erased $E_{nl}$ edges. We just showed that each connected component of that subforest has at most $n-(N+1)$ vertices. Thus, there are two vertices sent by $f_{\bar c,\s}$ to two different connected components of the forest. This means that, up to renumbering, $Q_{N+1}$ and $Q_{N+2}$ are in different connected components of $\Sigma\setminus\s_{nl}(\bar c,X_m^{\bar c})$. But $\g_j'$ is disjoint from $\s_{nl}(\bar c,X_m^{\bar c})$ and nontrivial on both $Q_{N+1}$ and $Q_{N+2}$. Therefore, $\g_j'$ contains at least two distinct curves that are in $\cxcourbes{\S\setminus \s_{nl}(\bar c,X_m^{\bar c})}$.
    \end{proof}

    Recall that the trees $\G_{\bar c}$ and $\G_{\s}$ are isomorphic by Point~\ref{iii max decompos} of Lemma~\ref{lemma max decompo are sent to max decompo}. In particular, $\s$ and $\bar c$ have the same number of leaf and non-leaf curves, so 
    $$\#\left(\s\cap \mclC_{nl}\right)=\#\left(\bar c\cap \mclC_{nl}\right)~.$$
    Now, note that $$N=\#\s_{nl}(\bar c,X_m^{\bar c})\leq\# \left(\s\cap \mclC_{nl}\right)$$
    because $\s_{nl}(\bar c,X_m^{\bar c})\subseteq \s\cap \mclC_{nl}$. Moreover, $$\# \left(\bar c\cap \mclC_{nl}\right)<\# \left(\bar c\cap \mclC_{leaf}\right)$$ because $\bar c\in D$.
    Thus it holds that $$N<\# \left(\bar c\cap \mclC_{leaf}\right)~.$$
    This means that we can choose $j>N$ such that $c_j'$ is a leaf curve. By Claim~\ref{claim g_j' at least two curves}, $\g_j'$ contains at least two curves in $\cxcourbes{\S\setminus \s_{nl}(\bar c,X_m^{\bar c})}$. By definition, $\g_j'$ is the canonical multicurve of a leaf curve so by Claim~\ref{claim at most one leaf curve}, only one of these two can be a leaf curve. Therefore, the other one is a non-leaf curve that we denote $\a$. 
    Take the countable conull Borel partition $\sqcup_m X_m^{\bar c'}$ of $X$ used to define the map $\s_{nl}$ in (\ref{def de sigma nl}). Take $m'$ such that $V\cap X_{m'}^{\bar c'}$ has positive measure and note that $\alpha\in \s_{nl}(\bar c',X_{m'}^{\bar c'})$. Moreover, recall that $\s_{nl}(\bar c,X_m^{\bar c})\subseteq \s'$. This gives 
    $$\s_{nl}(\bar c,X_m^{\bar c})\sqcup \{\a\}\subseteq \s_{nl}(\bar c',X_{m'}^{\bar c'})$$
    which contradicts the maximality of $\s_{nl}(\bar c,X_m^{\bar c})$. Therefore, $N=0$.
    
    To sum up, take $b\in\mclC_{leaf}$, $\mclB=\w^{-1}\left(\gen{T_b}\right)$ and $W$ a Borel subset on which the canonical multicurve $\b$ of $(\mclB|_W,\rho)$ exists. The curve $b$ can be completed into a multicurve of $D$ by Proposition~\ref{prop D les decompos particulieres avec g+p/3}. The fact that $N=0$ ensures that $\b$ contains $0$ non-leaf curves. Claim~\ref{claim at most one leaf curve} ensures that $\b$ is therefore reduced to exactly one leaf curve. Point~\ref{ii max decompos} of Lemma~\ref{lemma max decompo are sent to max decompo} ensures that $(\mclB,\rho)$ is trivial on $\S\setminus \b$, which concludes the proof.
\end{proof}

\subsubsection{Nonseparating and weakly separating curves}

When $\kappa(\S)\geq 3$, a weakly separating or nonseparating curve can be drawn on a leaf subsurface. Such a curve is characterised by the leaf curves that are disjoint from it.
When $\kappa(\S)=2$, the proof of the key Proposition~\ref{prop recap courbes dans courbes} is different because maximal decompositions of these are trivial in the way that the surface itself is the one piece of the maximal decomposition, so the decomposition multicurve is empty. 

Note that every curve of a complexity $2$ surface is weakly separating or nonseparating, so the following proposition actually implies the key Proposition~\ref{prop recap courbes dans courbes} when $\kappa(\S)=2$.

\begin{Prop}\label{nonsep curves prop}
Under Setting~\ref{setting for thm}, assume that $\S=\S'$ and $\kappa(\S)\geq 2$. Let $a\in\cxcourbes{\S}$ be a weakly separating or nonseparating curve. Let $\mclA=\w^{-1}\left(\gen{T_a}\right)$. 

Then, there exists a countable conull Borel partition $\bigsqcup_m X_m$ of $X$ such that for each $m$, there exists a curve $\a_m\in\cxcourbes{\S}$ such that 
    $$\rho(\mclA|_{X_m})\subseteq \gen{T_{\a_m}}~.$$
\end{Prop}
\begin{proof}
\textbf{Case 1: $\kappa(\S)=2$.}
Let $\mclH=\w^{-1}\left(\Stab a\right)$. The following is the only information we will use on the action-type cocycle $\w$:
\begin{itemize}[nolistsep, noitemsep]
    \item $\mclH$ and $\mclA$ are of infinite type by Lemma~\ref{cocycles from pmp actions are action-type},
    \item $\mclA$ is amenable by Lemma~\ref{lemma triv kernel cocycle amenable group amenable groupoid},
    \item $\mclH$ is everywhere nonamenable by Lemma~\ref{lemme action-type groupe non moyennable donne groupoide non moyennable},
    \item $\mclH$ and $\mclA$ normalise each other by Lemma~\ref{lemma if groups normalise then groupoids normalise}.
\end{itemize}
 Now, focus on the cocycle $\rho$. Choose a countable conull Borel partition $\bigsqcup_m X_m$ adapted to the canonical reduction sets of both $(\mclA,\rho)$ and $(\mclH,\rho)$, by Lemma~\ref{canonical reduction set exists} and Remark~\ref{rem we can always choose U for the canonical multicurve to exist}. Take $U$ a piece of that partition of positive measure and let $\mclH'=\mclH|_U$ and $\mclA'=\mclA|_U$.

By Corollary~\ref{if IA or T everywhere, then amen. If nonamen then IN somewhere}, $\mclH'$ must be irreducible nonamenable on at least some component $Q$ of $\S\setminus\s$, where $\s$ is the canonical reduction multicurve of $(\mclH',\rho)$. In particular, $\kappa(Q)\geq 1$ because $\mclH'$ is nontrivial on $Q$. 
The subgroupoid $\mclA'$ normalises $\mclH'$ so $(\mclA',\rho)$ stabilises $Q$ by Lemma~\ref{lemma multicurve is canonical}. Assume by contradiction that $(\mclA',\rho)$ is nontrivial on $Q$. Then, Proposition~\ref{if H normalises A then its reducible or IA} applied to $\mclH'$ normalising the amenable $\mclA'$ contradicts the fact that $(\mclH',\rho)$ is irreducible nonamenable on $Q$. Thus, $(\mclA',\rho)$ is trivial on $Q$, i.e.\ $$\rho(\mclA')\subseteq \Modfix{Q}~.$$

The case where $Q=\S$ is impossible because the trivial-kernel cocycle $\rho$ cannot send the infinite-type groupoid $\mclA'$ into the trivial group, by Remark~\ref{rem trivial kernel infinite-type groupoid to infinite group}. Hence $\kappa(Q)=1$ and $\partial Q$ is composed of exactly one curve $\a$. In conclusion, 
$$\rho(\mclA')\subseteq \Modfix{Q}\subseteq \gen {T_{\a}}$$
which proves the proposition in the case $\kappa(\S)=2$.

\medbreak
\textbf{Case 2: $\kappa(\S)\geq 3$.} In this case, $\nsig\geq 2$, i.e.\ the $0$-maximal decompositions of $\S$ have at least $2$ components. In particular, leaf curves are strongly separating.

If $a$ is weakly separating, let $S$ be the interior of a subsurface of $\S$ of type $\S_{0,3,1}$ which contains $a$. If $a$ is nonseparating, let $S$ be the interior of a subsurface of type $\S_{1,0,1}$ containing $a$. 
    In either case, $S$ is a leaf subsurface of complexity $1$, the boundary of $S$ is exactly one leaf curve $b=\partial S$, and $a\in\cxcourbes{S}$. Complete $b$ into a $0$-maximal decomposition $S_1, S_2,\dots,S_n$ of $\S$, with $S=S_1$ and $n=\nsig$. Apply Lemma~\ref{lemma max decompo are sent to max decompo} and restrict to a small enough Borel subset $U$ of positive measure accordingly. Name $Q_1,\dots,Q_n$ the image decomposition components and $\s$ the decomposition multicurve on $U$. Let $\mclA=\w^{-1}(\gen{T_{a}})$ and note that 
    $$\rho(\mclA|_U)\subseteq \rho(\w^{-1}(\Modact{S_1})|_U)\subseteq \Modact{Q_1}\times \gen{T_c}_{c\in\s\setminus \partial Q_1}~. $$
    Let $V$ be a positive measure Borel subset of $U$ on which the canonical multicurve $\a$ of $(\mclA|_V,\rho)$ exists. It is nonempty because $(\mclA|_V,\rho)$ is trivial for instance on $Q_2$. Note that $Q_2$ is well-defined because $\nsig\geq 2$. By the latter, $\a$ contains only curves that are either in $\s$ or in $\cxcourbes{Q_1}$.

    Take $b'$ a leaf curve that is disjoint from $a$ and intersects every $S_i$ nontrivially for $1\le i\le n$. 
    For instance, $b'$ can be taken as an image of $b$ by a pseudo-Anosov on $\S\setminus a$. 
    Then, by Proposition~\ref{separating curves prop}, up to restricting to a smaller positive measure Borel subset $W\subseteq V$, the canonical multicurve of $\mclB'=\w^{-1}(\gen{T_{b'}})|_W$ exists and is exactly one curve $\b'$. By Lemma~\ref{lemma b intersects S}, this curve $\b'$ intersects every $Q_i$ nontrivially for $1\le i\le n$. In particular, $\b'$ being one unique curve, it has to intersect every curve of $\s$.

    The subgroupoid $\mclB'$ normalises $\mclA|_W$ so $(\mclB',\rho)$ stabilises $\a$ by Lemma~\ref{lemma multicurve is canonical}. Thus,
    $$\rho(\mclB')\subseteq \Stab {\a}$$
    but also by Proposition~\ref{separating curves prop},
    $$\rho(\mclB')\subseteq \gen {T_{\b'}}$$
    which gives $$\rho(\mclB')\subseteq \Stab {\a}\cap\gen {T_{\b'}}~.$$
    The groupoid $\mclB'$ being of infinite type and the cocycle $\rho$ having trivial kernel, Remark~\ref{rem trivial kernel infinite-type groupoid to infinite group} forces $\Stab {\a}\cap\gen {T_{\b'}}$ to be an infinite group. Thus, there is a nontrivial power $k$ of $T_{\b'}$ that stabilises the multicurve $\a$. As we are in the torsion-free $\Modz$, $T_{\b'}^k$ actually stabilises each curve of $\a$ which means that $\b'$ and $\a$ are disjoint by Point~\ref{intersection formula dehn twists} of Proposition~\ref{denh twist prop}. But $\b'$ intersects every curve of $\s$, so $\a\cap\s=\varnothing$. Thus, $\a\subseteq \cxcourbes{Q_1}$.

    Remember that $\kappa(Q_1)=1$ by Point~\ref{i max decompos} of Lemma~\ref{lemma max decompo are sent to max decompo}, so any multicurve of $\cxcourbes{Q_1}$ can only be one single curve, so $\a$ is one single curve of $\cxcourbes{Q_1}$. So the groupoid $(\mclA|_W,\rho)$ is trivial on $\S\setminus Q_1$ and stabilises a curve $\a\in\cxcourbes{Q_1}$, i.e.\
    $$\rho(\mclA|_W)\subseteq \Modact{Q_1}\cap \Stab{\a}~.$$
    Note that $\kappa(Q_1)=1$, so $\kappa(Q_1\setminus \a)=0$, i.e.\ $Q_1\setminus\a$ is a pair of pants. Thus,
    $$\Modact{Q_1}\cap \Stab{\a}\subseteq \gen{T_{\a},T_{\partial Q_1}}$$
    which gives 
    $$\rho(\mclA|_W)\subseteq \gen{T_{\a},T_{\partial Q_1}}~.$$
    If $\rho(\mclA|_W)$ had a nontrivial intersection with $\gen{T_{\partial Q_1}}$, then the curve $\partial Q_1$ would be in the canonical multicurve of $(\mclA|_W,\rho)$. But the canonical multicurve of $(\mclA|_W,\rho)$ is the single curve $\a\in\cxcourbes{Q_1}$. Thus, $$\rho(\mclA|_W)\subseteq \gen{T_{\a}}$$
    which concludes the proof.
\end{proof}

\subsubsection{Strongly separating curves} The next proposition concludes the proof of the key Proposition~\ref{prop recap courbes dans courbes}. Any strongly separating curve can be completed into a $0$ or $1$-maximal decomposition by Lemma~\ref{lemma strongly sep curves into decompo}, so Point~\ref{ii max decompos} of Lemma~\ref{lemma max decompo are sent to max decompo} tells us that it is sent into a multitwist. Let us show that it is sent to one single twist, using the result we have for leaf curves (Proposition~\ref{prop leaf type curves}).

\begin{Prop}\label{separating curves prop}
Under Setting~\ref{setting for thm}, assume that $\S=\S'$. Let $a\in\cxcourbes{\S}$ be a strongly separating curve. Let $\mclA=\w^{-1}\left(\gen{T_a}\right)$. Then, there exists a countable conull Borel partition $\bigsqcup_m X_m$ of $X$ such that for each $m$, there exists a curve $\a_m\in\cxcourbes{\S}$ such that 
    $$\rho(\mclA|_{X_m})\subseteq \gen{T_{\a_m}}~.$$
\end{Prop}
\begin{proof}
    If $a$ is a leaf curve, the previous Proposition~\ref{prop leaf type curves} applies. Now assume $a$ is not a leaf curve, and complete $a$ into a $k$-maximal decomposition $\bar a$ for $k$ either $0$ or $1$, by Lemma~\ref{lemma strongly sep curves into decompo}. Denote by $S_1,\dots,S_n$ the components of the decomposition $\bar a$, with $n=\nsig -k$. Number the $(S_i)_i$ so that $S_1,\dots,S_p$ are on one side of $a$ and $S_{p+1},\dots,S_n$ are on the other side of $a$, for some $p<n$.

    Now, let $b$ be a leaf curve that is disjoint from $a$ and that intersects every $S_i$ nontrivially for $1\leq i\leq p$. The curve $b$ exists because $a$ is not a leaf curve. Similarly, let $b'$ be a leaf curve disjoint from $a$ and that intersects $S_i$ nontrivially for $p+1\leq i\leq n$. Denote by $\mclB=\w^{-1}\left(\gen{T_b}\right)$ and $\mclB'=\w^{-1}\left(\gen{T_{b'}}\right)$. Note that $\mclA$ normalises $\mclB$ and $\mclB'$ by Lemma~\ref{lemma if groups normalise then groupoids normalise}, because the curves $a,b$ and $b'$ are disjoint.

    Apply Proposition~\ref{prop leaf type curves} to both $b$ and $b'$ and apply Lemma~\ref{lemma max decompo are sent to max decompo} to the decomposition $\bar a$. Take a conull Borel partition of $X$ that refines the three partitions given by the latter. Thus, for any piece $U$ of the partition,
    \begin{itemize}
    \item by Proposition~\ref{prop leaf type curves}, there exist two leaf curves $\b$ and $\b'$ such that $$\rho(\w^{-1}(\gen {T_b})|_U)\subseteq \gen {T_{\b}}~,$$
    $$\rho(\w^{-1}(\gen {T_{b'}})|_U)\subseteq \gen {T_{\b'}}~;$$
    \item by Point~\ref{ii max decompos} of Lemma~\ref{lemma max decompo are sent to max decompo}, there exists a $k$-maximal decomposition $\s$ of $\S$ into components $Q_1,\dots,Q_n$ such that the canonical multicurve $\a$ of $(\mclA|_U,\rho)$ is constant, $\a\subseteq \s$ and $$\rho(\mclA|_U)\subseteq \gen {T_c}_{c\in\a}~;$$
    \item by Points~\ref{i max decompos} and \ref{iii max decompos} of Lemma~\ref{lemma max decompo are sent to max decompo}, the map $S_i\longmapsto Q_i$ is a graph isomorphism from $\G_{\bar a}$ to $\G_{\s}$ and for every $i\leq n$, $$\rho(\w^{-1}(\Modact{S_i})|_U)\subseteq\Modact{Q_i}\times \gen{T_c}_{c\in\s\setminus \partial Q_i}~.$$
    \end{itemize}

    Let $c$ be the curve of $\s$ that is in the boundary of both $Q_p$ and $Q_{p+1}$. The graph $\G_{\s}$ being a tree, this means that $c$ separates the components $(Q_i)_{i\leq p}$ from the components $(Q_i)_{i>p}$. 
    
    For $1\leq i\leq p$, apply Lemma~\ref{lemma b intersects S} to $b$ intersecting $S_i$ and obtain that $\b$ intersects $Q_i$ nontrivially. Thus, the single curve $\b$ intersects every $Q_i$ for $1\leq i\leq p$. Similarly, the single curve $\b'$ intersects every $Q_i$ nontrivially for $p+1\leq i\leq n$. Therefore, each curve of $\s\setminus \{c\}$ intersects either $\b$ or $\b'$. Thus, the only curve of $\s$ that can be disjoint from both $\b$ and $\b'$ is $c$.

    The groupoid $\mclA$ normalises $\mclB$ (resp.\ $\mclB'$) so by Lemma~\ref{lemma multicurve is canonical}, $\a$ and $\b$ (resp.\ $\b'$) are disjoint. But remember that $\a\subseteq \s$. Therefore, $\a\subseteq \{c\}$, which concludes the proof.
    \end{proof}

\subsection{From groupoid-theoretic rigidity to coupling-theoretic rigidity}\label{final subsection proof thm}

As explained in Setting~\ref{setting for thm}, given a measurable embedding coupling $\Omega$ of $\Modz$ into itself, Lemma~\ref{lemma measurable embedding provides groupoid and cocycles} provides us with a measured groupoid $\mclG$ over a finite measure Borel subset $X$ of $\Omega$. In Sections~\ref{subsection prelim results for the proof of thm}, \ref{section maximal decompo meas} and \ref{section curves meas}, we worked only on this space $X$, and got no particular information on the whole coupling $\Omega$.

To pass from the finite-measure subset $X$ to the infinite-measure space $\Omega$, we adapt an argument that appears in work on measure equivalence rigidity (e.g.\ in \cite{furman_gromovs_1999,monod_orbit_2006,kida_measure_2010,bader_integrable_2013,guirardel_measure_2025,escalier_graph_2026}) to the setting of measurable embeddings. It is the construction of a measurable and $(\Modz\times \Modz)$-equivariant map $\Omega \longrightarrow \operatorname{Aut}\,\cxcourbes{\S}$, called a \textit{tautening} map (see e.g.\ \cite[Appendix~A.3]{bader_integrable_2013} for more on tautening maps). This map sends the abstract coupling $\Omega$ to the discrete \textit{tautological} coupling of $G$ and $H$ acting by left-right translation on their finite-index overgroup $\operatorname{Aut}\,\cxcourbes{\S}$ (see Theorem~\ref{th ivanov}).

As we work with embeddings, we crucially apply the following theorem proven by Irmak \cite{irmak_superinjective_2006} and generalised by Shackleton \cite{shackleton_combinatorial_2007} to obtain the group $\operatorname{Aut}\,\cxcourbes{\S}$ at the target of the map $\Phi$.

\begin{Th}[Irmak {\cite{irmak_superinjective_2006}, Shackleton \cite{shackleton_combinatorial_2007}}]\label{irmak th}
    Let $\S$ be a connected surface of positive complexity. An injective map $\cxcourbes{\S}\longrightarrow\cxcourbes{\S}$ that sends intersecting curves to intersecting curves and disjoint curves to disjoint curves is necessarily an automorphism of the graph $\cxcourbes{\S}$.
\end{Th}

More precisely, the goal of this section is to prove the following theorem, which implies Theorem~\ref{main thm} as shown in Section~\ref{conclusion subsection}.

\begin{Th}\label{thm existence of tautening map}
    Let $\S$ be a connected surface without boundary and with $\kappa(\S)\geq 2$. Let $H$ and $G$ be two copies of the group $\Modz$. Let $\Omega$ be a measurable embedding coupling of $H$ into $G$. Then, there exists a measurable and $(G\times H)$-equivariant map
    $\Phi : \Omega \longrightarrow \operatorname{Aut}\,\cxcourbes{\S}$.
\end{Th}

\begin{proof}[Proof of Theorem~\ref{thm existence of tautening map}]
Fix a measurable embedding coupling $\Omega$ of $H$ into $G$. Apply Lemma~\ref{lemma measurable embedding provides groupoid and cocycles} to get $X=D_G\cap D_H$ where $D_G$ is a Borel fundamental domain for $G$ of finite measure and $D_H$ a Borel fundamental domain for $H$. These fundamental domains are chosen such that $(G\times H)X$ is conull in $\Omega$. The measured groupoid $\mclG=(H\ltimes D_G)|_X$ is equipped with two cocycles: $\w : \mclG \longrightarrow H$ and $\rho : \mclG \longrightarrow G$. In other words, we are in Setting~\ref{setting for thm} with $\S=\S'$.
\medbreak
\textbf{Step 1.} Let us construct a measurable map $\theta : X ^*\longrightarrow\operatorname{Aut}{\cxcourbes{\S}}$ on some conull Borel subset $X^*$ of~$X$.

        \textit{$\drsh$ Definition of $\theta$.} 
    For $a\in\cxcourbes{\S}$, by Proposition~\ref{prop recap courbes dans courbes}, there exists a countable conull Borel partition $X^a_*=\bigsqcup_{n}X_n^a$ of $X$ and $\a_n\in\cxcourbes{\S}$ such that for all $n$,
    $$\rho\left(\w^{-1}(\gen{T_a})\right)|_{X_n^a}\subseteq \gen{T_{\a_n}}~.$$
    Then, define the conull Borel subset $$X^*=\bigcap_{a\in\cxcourbes{\S}} X^a_*$$
and define on $X^*$ the following map 
$$\theta : x\longmapsto \left( a\mapsto \a_n \text{ when }x\in X_n^a\right)~.$$

    Let us show that $\theta$ takes values in $\operatorname{Aut}{\cxcourbes{\S}}$. Fix $a,b,c\in\cxcourbes{\S}$ such that $a\neq b$ and $c$ intersects $a$ but is disjoint from $b$. Take a countable partition of $X ^*$ that refines the ones for $a,b$ and $c$. Take $x\in X^*$ and $U$ the piece of the partition that contains $x$. 
    
     \textit{$\drsh$ Adjacency of $\theta(x)$.}
   Apply Lemma~\ref{adjacency curves in groupoidal terms} to 
    $$\rho \left( \w^{-1}(\gen{T_a})|_U\right)\subseteq \gen{T_{\theta(x)(a)}}$$
    and 
    $$\rho \left(\w^{-1}(\gen{T_b})|_U\right)\subseteq  \gen{T_{\theta(x)(b)}}~.$$
    This gives us that $\theta(x)(a)$ and $\theta(x)(b)$ are adjacent in $\cxcourbes{\S}$ if and only if $a$ and $b$ are.

    \textit{$\drsh$ Injectivity of $\theta(x)$.}
    The adjacency property of $\theta$ guarantees that $\theta(x)(c)$ intersects $\theta(x)(a)$ but is disjoint from $\theta(x)(b)$. In particular, $\theta(x)(a)$ and $\theta(x)(b)$ are distinct vertices of the curve graph.

    We have shown that for every $x$ in $X^*$, $\theta(x)$ verifies the assumptions of Theorem~\ref{irmak th}, therefore it is an element of $\operatorname{Aut}{\cxcourbes{\S}}$.

\medbreak
\textbf{Step 2.} Now we show a form of $(\w,\rho)$-equivariance of $\theta$. More precisely, let us show that for almost every $\mfk g\in\mclG$,
    \begin{align}\label{equivariance of theta}
        \rho(\mfk g)\theta(s(\mfk g)) =\theta(r(\mfk g))\w(\mfk g)~.
    \end{align}
    Fix $g,h\in\Modz$ and a bisection of positive measure $B$ on which $\rho$ is constant equal to $g$ and $\w$ is constant equal to $h$.
We will show (\ref{equivariance of theta}) for every $\mfk g\in B$, which suffices because
$$\mclG=\bigsqcup_{g,h\in\Modz} \w^{-1}(\{h\})\cap \rho^{-1}(\{g\})~.$$

Fix $a\in\cxcourbes{\S}$. Let $U$ a Borel subset of positive measure of the domain of $B$ small enough for $\alpha = \theta(x)(a)$ and $\b=\theta(Bx)(ha)$ to be well defined and constant when $x$ varies in $U$. For $x\in U$, there is a unique $\mfk g\in B$ with source $x$ because $B$ is a bisection. Then, the range of $\mfk g$ is $Bx$, $\rho(\mfk g)=g$ and $\w(\mfk g)=h$ so that
$$g\alpha =\rho(\mfk g)\theta(s(\mfk g))(a)$$
and
$$\b=\theta(r(\mfk g))(\w(\mfk g)a)~.$$Thus, to show (\ref{equivariance of theta}) it is now enough to show that $\b=g\a$.

By definition of $\theta$, 
\begin{equation}\label{U}
    \w^{-1}(\gen{T_a})|_U\subseteq \rho^{-1}(\gen{T_{\a}})|_U~,
\end{equation}
\begin{equation}\label{BU}
    \w^{-1}(\gen{T_{ha}})|_{BU}\subseteq \rho^{-1}(\gen{T_{\b}})|_{BU}~.
\end{equation}
Then, it holds that 
\begin{align*}
\w^{-1}(\gen{T_{ha}})|_{BU}&=B\w^{-1}(\gen{T_a})|_U B^{-1}\\
                        &\subseteq B\rho^{-1}(\gen{T_{\a}})|_U B^{-1}\text{  by }(\ref{U})\\
                        &=\rho^{-1}(\gen{T_{g\a}})|_{BU}
\end{align*}
which then gives by $(\ref{BU})$
$$\w^{-1}(\gen{T_{ha}})|_{BU}\subseteq \rho^{-1}(\gen{T_{\b}})|_{BU}\cap \rho^{-1}(\gen{T_{g\a}})|_{BU}~.$$
The subgroupoid $\w^{-1}(\gen{T_{ha}})|_{BU}$ is of infinite type and $\rho$ has trivial kernel. By Remark~\ref{rem trivial kernel infinite-type groupoid to infinite group}, this forces the subgroup $\gen{T_{\b}}\cap \gen{T_{g\a}}$ to be nontrivial, which means that $\b=g\a$ by Point~\ref{dehn twists injectivity} of Proposition~\ref{denh twist prop}. This concludes the proof of $(\ref{equivariance of theta})$.
\medbreak
\textbf{Step 3.} Finally, we construct a tautening map on the whole space $\Omega$, i.e.\ a measurable map $$\Phi : \Omega \longrightarrow \operatorname{Aut}\,\cxcourbes{\S}$$
which is equivariant for the $G\times H$ action. More precisely, for every $g,h\in\Modz$ and almost every $y\in\Omega$, 
$$\Phi((g,h)y)=g\Phi(y)h^{-1}~.$$

To construct $\Phi$, we extend $\theta$ by equivariance to the conull Borel subset $(G\times H)X^*$. For a fixed $y\in\Omega$, there exist $g,h\in G$ and $x\in X^*$ such that $y=(g,h)x$. Define
    $$\Phi((g,h)x)=g \theta(x)h^{-1}$$
    and note that it is automatically $(G\times H)$-equivariant. 
    It is left to show that $\Phi$ is well defined almost everywhere. Take $g_1,g_2\in G$, $h_1,h_2\in H$ and $x_1,x_2\in X^*$ such that 
    $$y=(g_1,h_1)x_1=(g_2,h_2)x_2~.$$
    Thus, for $g=g_1 ^{-1}g_2$ and $h=h_1^{-1}h_2$ it holds that $x_1=(g,h)x_2$. In particular, the element $$\mfk g\defeq (g,x_2)=(h,x_2)$$ in the groupoid $\mclG$ has source $x_2$ and range $x_1$. Applying (\ref{equivariance of theta}) to $\mfk g$ and using that $\rho(\mfk g)=g$ and $\w(\mfk g)=h$, we get 
    \begin{equation}\label{application of equivariance}
      g\theta(x_2)=\theta(x_1)h
    \end{equation}
    for almost every $x_1,x_2\in X^*$. Therefore, 
    \begin{align*}
        \Phi((g_1,h_1)x_1)&=g_1\theta(x_1)h_1^{-1}\\
                        &=g_1\left(g_1 ^{-1}g_2\right)\theta(x_2)\left(h_1 ^{-1}h_2\right)^{-1}h_1^{-1}\text{ by~(\ref{application of equivariance})}\\
                        &=g_2\theta(x_2)h_2^{-1}\\
                        &=\Phi((g_2,h_2)x_2)
    \end{align*}
    for almost every $x_1,x_2\in X$ as wanted.
\end{proof}

\section{Conclusion}\label{conclusion subsection}

In this section, we restate our main results, we explain how our main theorem follows from Theorem~\ref{thm existence of tautening map} and we prove the corollaries. Recall that our definition of a \textit{surface} is an orientable $2$-manifold of finite type.

\begin{manualtheorem}{\ref{main thm}}
Let $\S$ be a connected surface without boundary and with $\kappa(\S)\geq 2$. Then, any measurable embedding coupling from $\ModS$ to $\ModS$ is necessarily a measure equivalence coupling.
\end{manualtheorem}
\begin{proof}
    As explained in Remark~\ref{remark finite indices in me couplings}, it suffices to study a measurable embedding of the finite-index subgroup $\Modz$ of $\ModS$ into itself. Hence, fix two copies $h$ and $G$ of $\Modz$ and fix $\Omega$ a measurable embedding coupling of $H$ into $G$. By Theorem~\ref{thm existence of tautening map}, there exists a measurable and $(G\times H)$-equivariant map ${\Phi : \Omega \longrightarrow \operatorname{Aut}\,\cxcourbes{\S}}$.
    
    By Ivanov's Theorem~\ref{th ivanov}, $\Modz$ is a finite-index subgroup of $\operatorname{Aut}\,\cxcourbes{\S}$. This means that $H$ acts on $\operatorname{Aut}\,\cxcourbes{\S}$ with a finite fundamental domain $K$. The tautening map provides us with the Borel subset $\Phi^{-1}(K)$. It is a fundamental domain for $H$ acting on $\Omega$, and it has finite measure because $K$ is finite. This proves that $\Omega$ is a measure equivalence coupling.
\end{proof}

A direct consequence of this theorem is the description of countable groups that are measurably biembeddable with a mapping class group, using Kida's measure equivalence rigidity theorem \cite{kida_measure_2010}.

\begin{manualcor}{\ref{cr biembedability}}
Let $\S$ be a connected surface without boundary and with $\kappa(\S)\geq 2$. If a countable group $H$ measurably embeds into $\ModS$ and $\ModS$ measurably embeds into $H$, then $H$ is virtually isomorphic to $\ModS$ up to finite kernels.
\end{manualcor}

\begin{proof}
    By Theorem~\ref{main thm} and Proposition~\ref{transitivité du plongement ME}, $H$ is measure equivalent to $\ModS$. By Kida's \cite[Theorem~1.1]{kida_measure_2010}, $H$ is virtually isomorphic to $\ModS$ up to finite kernels.
\end{proof}

As explained in Remark~\ref{rem on equality or not of S and S'}, Lemma~\ref{lemma max decompo are sent to max decompo} gives monotonicity of topological complexity for measurable embeddings.

\begin{manualcor}{\ref{cr monotonicity of complexity}}
    Let $\S$ and $\S'$ be connected surfaces without boundaries and with $\kappa(\S),\kappa(\S')\geq 0$. If $\ModS$ measurably embeds into $\Mod {\S'}$, then ${\kappa(\S)\leq \kappa(\S')}$. 
\end{manualcor}
\begin{proof}
    Assume that there is a measurable embedding from $\ModS$ to $\ModSprime$.
    If $\kappa(\S')=0$, then $\ModSprime$ is amenable so $\ModS$ is also amenable by Remark~\ref{rem amen is preserved by meas emb}, which means that $\kappa(\S)=0$.
    Now, assume that $\kappa(\S')\geq 1$. By Remark~\ref{remark finite indices in me couplings}, there is a measurable embedding from $\Modz$ to $\Modzprime$ so we find ourselves in Setting~\ref{setting for thm}. Thus, Lemma~\ref{lemma max decompo are sent to max decompo} applies and gives ${\kappa(\S)\leq \kappa(\S')}$.
\end{proof}

A consequence of Theorem~\ref{thm existence of tautening map} is the following result about free probability-measure preserving actions of $\ModS$. See for instance \cite[Definition~1.3]{kida_orbit_2008} or \cite[Definition~1.1]{furman_orbit_1999} for the definition of virtual conjugacy of probability measure-preserving actions.

\begin{manualcor}{\ref{cr OE rigidity}}Let $\S$ be a connected surface without boundary and with $\kappa(\S)\geq 2$ and let $G$ and $H$ be two copies of $\ModS$. Assume that there exist free probability measure-preserving actions of $G$ and $H$ on a standard probability space $(X,\mu)$ such that for almost every $x\in X$, $H\cdot x\subseteq G\cdot x$. 
 
Then, for almost every $x\in X$, $H\cdot x=G\cdot x$. If additionally the actions are ergodic, then they are virtually conjugate.
\end{manualcor}
\begin{proof}

 First, let us construct a measurable embedding coupling of $H$ into $G$ from this context. This argument can be found in \cite[Remark~2.36]{delabie_quantitative_2022} with $X=Y$ and in \cite[Proposition~2.5]{demir_measurable_2022}. 
 Let
 $$\Omega =\{ (x,y)\in X\times X\mid y\in G\cdot x\}$$
 the countable Borel equivalence relation associated to the action of $G$ on $X$. It is a standard measured space equipped with the usual measure denoted $m$ and defined in \cite{feldman_ergodic_1977}. 
 Now, for $(x,y)\in\Omega$, $g\in G$ and $h\in H$, let
 $$g(x,y)=(x,g\cdot y)\text{ and }h(x,y)=(h\cdot x,y)~.$$
 This defines two free commuting measure-preserving actions of $G$ and $H$ on $(\Omega,m)$. To be clear, the p.m.p.\ actions on $X$ are denoted by $\cdot$ and the commuting actions on the space $\Omega$ of infinite measure have no notation. 

 Now, the Borel subset
 $$\Delta =\{(x,x)\mid x\in X\}$$
of $\Omega$ has finite measure $m(\Delta)=\mu(X)$. On one hand, $\Delta$ is a fundamental domain for the action of $G$ on $\Omega$. On the other hand, the action of $H$ on $X$ is free, so for every $h,h'\in H$, $h\Delta$ and $h'\Delta$ are disjoint. This provides a way to construct a Borel fundamental domain for the action of $H$ on $\Omega$ that contains $\Delta$. For the detailed construction, see \cite[Proposition~2.5]{demir_measurable_2022}. In particular, $\Omega$ is a measurable embedding coupling of $H$ into $G$. 

\begin{Claim}\label{claim HDelta is conull in Omega}
    The Borel subset $H\Delta$ is conull in $\Omega$.
\end{Claim}
\begin{proof}[Proof of Claim~\ref{claim HDelta is conull in Omega}]
Let $G^0\leq G$ and $H^0\leq H$ be two copies of $\Modz$ and denote $n$ the finite index of these subgroups. By Remark~\ref{remark finite indices in me couplings}, $\Omega$ is a measurable embedding coupling of $H^0$ into $G^0$. As in that remark, fix finite systems of representatives $R_H$ and $R_G$ for $H^0\leq H$ and $G^0\leq G$. On one hand, $R_G\Delta$ is a fundamental domain for $G^0$ acting on $\Omega$ and $m(R_G\Delta)=nm(\Delta)$. On the other hand, for every $h,h'\in H^0$, $h R_H\Delta$ and $h' R_H\Delta$ are disjoint.

Apply Theorem~\ref{thm existence of tautening map} to the measurable embedding coupling $\Omega$ of $H^0$ into $G^0$. There exists a measurable and $(G^0\times H^0)$-equivariant map $\Phi : \Omega\longrightarrow \operatorname{Aut}\,\cxcourbes{\S}$. Recall that $G^0\times H^0$ acts by left and right translations on $\operatorname{Aut}\,\cxcourbes{\S}$, so there exists a common finite fundamental domain $K$ for these two actions. Then, the Borel subset $$D=\Phi^{-1}(K)$$
of $\Omega$ is a common fundamental domain for the actions of $G^0$ and $H^0$ on $\Omega$.

Note that any two Borel fundamental domains for the action of $G^0$ (resp. $H^0$) on $\Omega$ have the same measure. 
Both $D$ and $R_G\Delta$ are Borel fundamental domains for $G^0$ acting on $\Omega$, so $$m(D)=m(R_G\Delta)=nm(\Delta)~.$$
The $H^0$-translates of $R_H\Delta$ are all pairwise disjoint, so $$R_H\Delta \sqcup \left(D\setminus H^0R_H\Delta\right)$$ is a Borel fundamental domain for $H^0$ acting on $\Omega$. Hence, this subset has the same measure as $D$, so
\begin{align*}
    m(D)&=m\left(R_H\Delta \sqcup \left(D\setminus H^0R_H\Delta\right)\right)\\
        &=m(R_H\Delta)+m\left(D\setminus H^0R_H\Delta\right)\\
        &=nm(\Delta)+m\left(D\setminus H^0R_H\Delta\right)\\
        &=m(D)+m\left(D\setminus H^0R_H\Delta\right)~.
\end{align*}
This gives $m\left(D\setminus H^0R_H\Delta\right)= 0$.

Recall that $H^0R_H=H$ so that $D\setminus H\Delta$ has null measure. Thus, the Borel subset
$$H\left(D\setminus H\Delta\right) = \Omega\setminus H\Delta$$
of $\Omega$ also has null measure. In conclusion, $H\Delta$ is conull in $\Omega$. 
\end{proof}

\begin{Claim}\label{claim almost every x}
For almost every $x\in X$ and every $g\in G$, $g(x,x)\in H\Delta$.
\end{Claim}
\begin{proof}[Proof of Claim~\ref{claim almost every x}]
    Take $g\in G$ and a Borel subset $U$ of $X$ such that for every $x\in U$, $g(x,x)\notin H\Delta$. Denote $\Delta_U=\{(x,x)\}_{x\in U}$ and note that $m(\Delta_U)=\mu(U)$. The latter gives $g\Delta_U\subseteq \Omega\setminus H\Delta$. By Claim~\ref{claim HDelta is conull in Omega}, $m(\Delta_U)= 0$ so $\mu(U)=0$.
\end{proof}

Let us now show that for almost every $x\in X$, $G\cdot x\subseteq H\cdot x$. 
Take $g\in G$ and $x\in X$ such that $g(x,x)\in H\Delta$. Then, there exists $h\in H$ and $y\in X$ such that $g(x,x)=h(y,y)$. By definition of the actions on $\Omega$, we get $g\cdot x=y$ and $h\cdot y=x$. Thus, $g\cdot x=h^{-1}\cdot x\in H\cdot x$ as wanted. The latter holds for almost every $x$ by Claim~\ref{claim almost every x}.

Assume additionnally that the actions of $G$ and $H$ on $X$ are ergodic. By \cite[Theorem~1.1]{kida_orbit_2008}, these actions are virtually conjugate.
\end{proof}

\bibliographystyle{alpha}
\bibliography{mcgbibfinal}

\end{document}